\documentclass[english,book]{amsart}
\usepackage{amssymb,amscd,amsmath,amsthm}
\usepackage[condensed,math]{anttor}
\usepackage[T1]{fontenc}
\usepackage[utf8]{inputenc}
\usepackage[english]{babel}
\usepackage{csquotes}
\usepackage[colorlinks=true]{hyperref}
\usepackage{bookmark}
\usepackage{dutchcal}
\usepackage{eufrak}
\usepackage[shortlabels]{enumitem}
\usepackage{ dsfont }
\usepackage{mathtools}
\usepackage{mathbbol}
\usepackage{stmaryrd}
\usepackage{float}
\usepackage{tikz}
\usetikzlibrary{cd}
\usetikzlibrary{patterns}
\usepackage{moreenum}
\usepackage{animate}
\usepackage{array}
\usepackage{phoenician}
\usepackage{xstring}
\usepackage{xspace}
\usepackage{adjustbox}
\usepackage{wasysym}
\usepackage{marvosym}
\usepackage{manfnt}
\usepackage[noabbrev]{cleveref}
\usepackage[   letterpaper, total={8.5in, 11in}, top=0.75in, inner=0.75in, outer=0.75in,bottom=0.75in]{geometry}
\usepackage{quiver}

\newcommand{\field}[1]{\mathbb{#1}}
\newcommand{\mc}[1]{\ensuremath{\mathcal{#1}}}
\newcommand{\ul}[1]{\underline{#1}}
\newcommand{\wt}[1]{\widetilde{#1}}

\DeclareMathOperator{\holim}{holim}

\theoremstyle{definition}
\newtheorem{definition}{\ul{Definition}}[subsection]

\newtheorem{theorem}[definition]{\ul{Theorem}}
\newtheorem*{theorem*}{\ul{Theorem}}
\newtheorem{proposition}[definition]{\ul{Proposition}}
\newtheorem{corollary}[definition]{\ul{Corollary}}
\newtheorem{lemma}[definition]{\ul{Lemma}}
\newtheorem{claim}[definition]{\ul{Claim}}
\theoremstyle{definition}
\newtheorem{construction}[definition]{\ul{Construction}}
\newtheorem{notation}[definition]{\ul{Notation}}

\newtheorem*{problem*}{\ul{Problem}}
\newtheorem{example}[definition]{\ul{Example}}

\theoremstyle{remark}
\newtheorem{remark}[definition]{\ul{Remark}}
\addto\captionsenglish{}

\newcommand{\mb}[1]{{\textbf {\textit#1}}}
\newcommand\CE{\textit{C\!E}}

\newcommand{\mf}[1]{\mathfrak{#1}}
\newcommand{\ms}[1]{\mathcal{#1}}

\newcommand{\GF}{(S,K,V,\rho)}
\def\red{{\operatorname{red}}}
\def\DR{\mc{DR}}
\def\Grpd{\mc{Grpd}}
\def\DRG{\mc{DRG}}

\def\GS{\Gamma(\Sigma)}

\def\ve{\varepsilon}

\def\Q{\field{Q}}

\def\Cc{\mc{C}}

\def\C{\field{C}}
\def\VS{V_\Sigma}
\def\CP{\field{CP}}

\def\g{\mf{g}}
\def\N{\field{N}}

\def\ss{\subset}
\def\sse{\subseteq}

\def\cd{\cdot}

\def\Z{\mathbb{Z}}
\def\R{\field{R}}
\def\la{\langle}
\def\ra{\rangle}
\def\NAut{\wt{\Aut}}
\newcommand{\ol}[1]{\overline{#1}}
\newcommand{\wh}[1]{\widehat{#1}}

\DeclareMathOperator{\Lie}{Lie}

\def\bt{\bullet}
\DeclareMathOperator{\Aut}{Aut}

\DeclareMathOperator{\End}{End}

\DeclareMathOperator{\ad}{ad}

\DeclareMathOperator{\id}{id}
\DeclareMathOperator{\Sch}{Sch}

\DeclareMathOperator{\Ker}{ker}
\DeclareMathOperator{\Hopf}{Hopf}
\DeclareMathOperator{\im}{im}
\DeclareMathOperator{\Hom}{Hom}

\def\op{{\operatorname{op}}}

\def\Sc{\ensuremath{\mathcal{S}}}

\def\Top{\ensuremath{\mathcal{Top}}}
\def\In{\mf{I}}
\def\TT{\mb{T}}

\DeclareMathOperator{\Id}{Id}

\def\O{\ensuremath{\mathcal{O}}}

\def\V{\mathcal{V}}
\def\Rge{\R_{\ge 0}}

\def\vp{\varphi}

\def\Sm{\mf{S}}

\def\h{\mf{h}}

\newcolumntype{L}{>{$}l<{$}}
\def\temp{&}
\catcode`&=\active
\let&=\temp
\def\AA{\mathbb{A}}
\def\uSigma{\ul{\Sigma}}
\def\CT{\C^\times}
\def\vk{\varkappa}
\def\ESm{\mf{ES}}
\def\Gs{\mc{G}}
\def\tc{\pi_0^{\Top}}
\DeclareRobustCommand{\phalphnum}[1]{
	\IfEqCase{#1}{%
		{1}{{\textphnc{a}}}%
		{2}{{\textphnc{b}}}%
		{3}{\textphnc{g}}%
		{4}{\textphnc{d}}%
		{5}{\textphnc{h}}
		{6}{\textphnc{f}}
		{7}{\textphnc{w}}
		{8}{\textphnc{z}}
		{9}{\textphnc{H}}
		{10}{\textphnc{T}}
		{11}{\textphnc{y}}
		{12}{\textphnc{k}}
		{13}{\textphnc{l}}
		{14}{\textphnc{m}}
		{15}{\textphnc{n}}
		{16}{\textphnc{s}}
		{17}{\textphnc{o}}
		{18}{\textphnc{p}}
		{19}{\textphnc{x}}
		{20}{\textphnc{q}}
		{21}{\textphnc{r}}
		{22}{\textphnc{S}}
		{23}{\textphnc{t}}
	}[\PackageError{\phalphnum}{Undefined option to phalphnum: #1}{}]%
}
\setlist[enumerate,1]{label={\hspace{-4pt}\phalphnum{\arabic*})}}

\def\doi#1{\href{https://doi.org/#1}{doi:#1}}
\def\arXiv#1{\href{https://arxiv.org/abs/#1}{arXiv:#1}}
\begin{document}
\title[Equivariant automorphisms and applications]{Equivariant automorphisms of the Cox construction and applications}
\author{Gregory Taroyan}
\begin{abstract}
    In the present paper, we give a complete description of the group of holomorphic automorphisms of the Cox construction of a simplicial fan equivariant with respect to a large enough connected complex Lie subgroup of the large torus acting on the Cox construction. We then apply this description to find the groups of holomorphic automorphisms of rational moment-angle manifolds, including Calabi--Eckmann manifolds and Hopf manifolds. We also calculate the groups of equivariant automorphisms of complete simplicial toric varieties. This paper is the first in a series of two papers dedicated to studying the automorphism groups of toric stacks with applications to complex geometry.
\end{abstract}
\maketitle
\setcounter{tocdepth}{1}
\tableofcontents
\setcounter{tocdepth}{3}
\section{Introduction}
Toric geometry is a field that studies algebraic varieties with automorphism groups with large abelian subgroups. The theory of toric varieties was initiated independently from quite different angles by several groups of mathematicians in the 1970s. The first paper on toric geometry due to Demazure used the theory to construct certain subgroups in the Cremona group of the projective space \cite{Demazure}. The approach of Kempf--Knudsen--Mumford--Saint-Donat \cite{KKMS} considered toric varieties from the point of view of geometric invariant theory. The approach of Oda--Miyake also developed a version of the theory of toric varieties \cite{Oda_Miyake}. Finally, the \enquote{Russian School} arrived at the definition of projective toric varieties through the study of subvarieties in the algebraic torus \((\CT)^n\) and the related combinatorics and geometry of Newton polyhedra. The following papers of Khovanskii \cite{Khovanskii} and Danilov \cite{Danilov} exemplify this approach.

The automorphism groups of toric varieties were considered in several papers on the subject. When a toric variety is \emph{complete}, or compact in the classical topology, one can describe the entire group of its regular automorphisms. At various levels of generality, this was done by Demazure in \cite{Demazure}, and Cox in \cite{Cox_Aut} for simplicial toric varieties, and by Sancho--Moreno--Sancho in \cite{Sancho_aut} for general complete toric varieties. 

The essential tool in most of the calculations with automorphisms of toric varieties is the existence of the \emph{Cox construction}. The Cox construction is an abelian reductive torsor over the toric variety with a quasiaffine toric total space. The calculation of automorphisms of toric varieties using the Cox construction usually proceeds as follows. One identifies the automorphisms of the total space of the Cox construction equivariant with respect to the action of the fiber group. Then, one proves a lifting statement that asserts the existence of lifts to the Cox construction. Finally, one identifies the kernel of the descent homomorphism on the automorphism groups from the total space to the base. 

For example, this method was used by Arzhantsev--Hausen--Derenthal \cite{Arzh_Cox} to develop an immensely powerful theory of Cox rings that allows one to describe automorphisms of arbitrary complete varieties with discrete class groups. Our goal is to extend the fundamental ideas of this method to the holomorphic setting.

The present paper provides a generalization of the first step in this description of the automorphism groups. Namely, we describe the group of automorphisms of the Cox construction equivariant with respect to an arbitrarily connected complex Lie subgroup of the large torus acting on the total space. Importantly, our description applies to all \emph{holomorphic} automorphisms of the Cox construction, not just the regular ones. As we discussed above, a result of this type is a preliminary step towards a complete description of the automorphism groups of the new object we will introduce in the sequel: \emph{toric stack associated with a generalized fan}. The ultimate goal of this two-paper series is to provide a uniform method for calculating automorphism groups of complex manifolds that admit rich actions of abelian complex Lie groups. 

The results of this paper, however, are of independent interest because they provide a tool for the calculation of automorphism groups of a large class of non-K\"{a}hler complex manifolds we call \emph{rational moment-angle manifolds}. In addition, we also extend the calculation of the equivariant automorphism groups to arbitrary complete simplicial toric varieties. 
\subsection{Summary of results}
The main result of this paper is the description of the automorphisms of the Cox construction for an arbitrary fan \(\Sigma\) equivariant with respect to a connected complex Lie subgroup of the large torus. To give such a description, we introduce the notion of a \emph{generalized fan}, which is roughly an image of a classical simplicial fan under a linear map. For instance, given a fan \(\Sigma\) in a vector space \(V\), its image under the projection map along a subspace \(W\) is a generalized fan. In these terms, the main result of this paper can be stated as follows. 
\begin{theorem*}[{Theorem \ref{thm:equiv_on_Cox_constr}}]
    Let \(\Sigma=\GF\) be a generalized fan. Let \(H\) be an arbitrary connected complex Lie subgroup of the large torus acting on the Cox construction \(U(\Sigma)\) of \(\Sigma\). Denote by \(\wt{\Sigma}\) the fan defining the Cox construction of \(\Sigma\) as a toric variety. Denote by \(\h^\R\) the image of the Lie algebra of \(H\) under the map of coordinate-wise real parts. Assume that the quotient generalized fan \(\wt{\Sigma}/\h^\R\) is complete. Then, the group of holomorphic automorphisms of \(U(\Sigma)\) equivariant with respect \(H\) is isomorphic is a linear algebraic group \(\NAut(\Sigma)\). As an algebraic group, the connected component of identity in the group \(\NAut(\Sigma)\) is generated by the large torus and one-parametric subgroups corresponding to the Demazure roots of the generalized fan \(\Sigma/\h^\R\). The group of components of the group \(\NAut(\Sigma)\) is a quotient of the symmetry group of the quotient fan \(\Sigma/\h^\R\).
\end{theorem*}
An almost immediate corollary of the above theorem is the following description of equivariant automorphisms of complete simplicial toric varieties.
\begin{theorem*}[{Theorem \ref{thm:auto_of_toric_var}}]
    Let \(V_\Sigma\) be a complete simplicial toric variety with the fan \(\Sigma\). Let \(H\) be an arbitrary connected complex Lie subgroup of the large torus acting on \(X\). Then the group of holomorphic automorphisms of \(V_\Sigma\) equivariant with respect to \(H\) is a linear algebraic group \(\Aut_H(V_\Sigma)\). As an algebraic group, the connected component of identity in the group \(\Aut_H(V_\Sigma)\) is generated by the large torus and one-parametric subgroups corresponding to the Demazure roots of the fan of \(\Sigma/\h^\R\). The group of components of the group \(\Aut_H(V_\Sigma)\) is a quotient of the symmetry group of the quotient fan \(\Sigma/\h^\R\).
\end{theorem*}
Finally, we obtain a result describing the holomorphic automorphism groups of the \emph{rational moment-angle manifolds}. These are a class of non-K\"{a}hler complex manifolds that correspond to fans with ray generators in a discrete lattice of the vector space.
\begin{theorem*}[{Theorem \ref{thm:ma_aut_group_structure}}]
    Let \(\Sigma\) be a rational, complete simplicial fan. Denote by \(\NAut(\Sigma)\) the group of regular automorphisms of the Cox construction \(U(\Sigma)\) normalizing the canonical quasitorus \(G_\Sigma\). Let \(Z_\Sigma=U(\Sigma)/H_\Sigma\) be the complex moment-angle manifold associated with \(\Sigma\). Then, the group of holomorphic automorphisms \(\Aut(Z_\Sigma)\) of the moment-angle manifold \(Z_\Sigma\) associated with \(\Sigma\) fits into the following exact sequence.
    \[
        1\to H_\Sigma\to \NAut(\Sigma)\to \Aut(Z_\Sigma)\to 1.
    \]
\end{theorem*}
We then apply this result to show that the automorphism groups of moment-angle manifolds detect the complex structure. That is, two rational moment-angle manifolds with the same fan are biholomorphic if and only if their automorphism groups are isomorphic. This is shown in Corollary \ref{cor:aut_detect_complex_struct}.
\subsection{Organization of the paper} The paper is divided into two parts. In Part \ref{part:gen_theory}, we develop the theory of generalized fans and calculate the automorphism groups of Cox constructions of generalized fans. In Part \ref{part:applications}, we apply the results of Part \ref{part:gen_theory} to the study of the automorphisms moment-angle manifolds and simplicial toric varieties.

Section \ref{sec:gen_fans} introduces the notion of a generalized fan and explores the basic properties of and constructions with these objects. In \S \ref{subsec:fans_and_gen_fans}, we explain how one can enhance a purely combinatorial notion of a simplicial complex with additional combinatorial data corresponding to a map from the set of vertices of the simplicial complex to a vector space to define a generalized fan. We also compare this new definition with the classical notion of a simplicial fan and show that any simplicial fan with a choice of ray generators gives rise to a generalized fan. In \S \ref{subsec:complete_fans}, we explore the notion of completeness for the generalized fans. In \S \ref{subsec:fan_tori}, \S \ref{subsec:complex_struct_on_tori}, we explain how one associates to a generalized fan several abelian Lie groups. Then, in \S \ref{subsec:Cox_construction} we show that each generalized fan admits a Cox construction. Finally, in \S \ref{subsec:rationality}, we explore the rationality properties of generalized fans and show that any generalized fan admits a canonical \emph{rationalization}.

In section \ref{sec:cox_ring}, we define Cox rings of generalized fans and calculate their graded automorphisms in the case when the generalized fan is complete. The exposition there is very similar to the theory developed by Cox in \cite{Cox_Aut}, and we include it mostly for the sake of completeness. In particular, we show that for a generalized fan, one can define the notion of a \emph{Demazure root} and the corresponding one-parametric subgroup of the large torus acting on the Cox construction. We also show that the Cox ring of a generalized fan is a finitely generated \(\C\)-algebra and that the graded automorphisms of this algebra form a linear algebraic group. Finally, we describe the root decomposition for the Lie algebra of this algebraic group, see \S \ref{subsec:lie_theor_pic}.

Section \ref{sec:Norm_of_GH} contains proof of a GAGA-type result for the holomorphic automorphisms of the Cox construction that normalize the tori associated to a complete fan. Concretely, we show that all such automorphisms are algebraic using Newton polyhedra. Section \ref{sec:equi_aut_of_cox} summarizes the results of Part \ref{part:gen_theory} and contains the paper's main result: a complete description of the group of holomorphic automorphisms of the Cox construction of a generalized fan equivariant with respect to a connected complex Lie subgroup of the large torus.

Part \ref{part:applications} starts with Section \ref{sec:aut_toric} in which we apply the result of Section \ref{sec:equi_aut_of_cox} to calculate the equivariant automorphisms of complete simplicial toric varieties. In Section \ref{sec:moment_angle}, we define moment-angle manifolds and quickly recall several important results of Ishida \cite{ishida_13}, \cite{ishida_15}, and \cite{ishida_19} required to calculate automorphisms. The calculation of automorphisms is done in Section \ref{sec:aut_mam}, where we leverage the description of the automorphisms of the Cox construction to obtain a description of the automorphisms of the moment-angle manifolds. In Section \ref{sec:CE_aut}, we apply the results of Section \ref{sec:aut_mam} to describe the automorphism groups of Calabi--Eckmann manifolds and Hopf manifolds. In addition, we also provide a comparison with the results of Namba \cite{namba} on automorphisms of Hopf surfaces. 

Finally, in Section \ref{sec:conclusion}, we discuss how the paper's results fit into the theory of toric stacks we aim to provide in the sequel. In particular, we define a toric stack as a sheaf of groupoids and show how to use generalized fans to encode toric stacks. We also sketch the relationship with other theories of generalized toric varieties that have emerged over the last years. However, the list of possible connections we provide is in no way exhaustive.
\subsection*{Some common notation}
\begin{itemize}
    \item Let \(S\) be a set. Then \(\Re, \Im\) denote the maps \(\C^S\to \R^S\) that assign the real and imaginary parts to each vector coordinate-wise, respectively.
    \item Let \(S\) be a set. Then \(\exp_\R\) denotes the map \(\R^S\to \R_{>0}^S\) is the coordinate-wise exponential map.
    \item Let \(S\) be a set. Then \(\exp_\C\) denotes the map \(\C^S\to (\C^\times)^S\) is the coordinate-wise exponential map.
    \item We denote by \(\Aut(X)\) the group of automorphisms of either algebraic variety or complex-analytic space \(X\). If \(X\) is an algebraic variety, we consider regular automorphisms, and if \(X\) is a complex analytic space, we consider biholomorphisms. 
\end{itemize}
\subsection*{Acknowledgements} The present paper was prepared with various degrees of intensity over the last five years. This, of course, necessitates quite a long list of acknowledgements.

First and foremost, we would like to thank our undergraduate advisor, Taras Panov, for posing the problem and providing many helpful comments and suggestions. We are deeply grateful to Ivan Arzhantsev for explaining the extreme importance of the Cox construction in toric geometry. In addition, we would like to thank Askold Khovanskii for many helpful and inspiring discussions on toric geometry. We also thank Dima Kaledin and Sergey Gorchinsky for discussing the paper's subject. 

Alexander Perepechko and Roman Krutowski deserve our special gratitude for reading the manuscript at various stages of completeness and providing many helpful insights. We also want to express our appreciation to Kostya Loginov and Alexey Gorinov for reigniting our interest in the topic and encouraging the completion of the paper. In addition, we thank Vasya Rogov for his clarifying remarks on the GAGA principle. We are also grateful to Marco Gualtieri for inspiring lectures that motivated us to turn again to the study of toric stacks. 

Finally, we would like to thank our partner, Alisa Chistopolskaya, for her constant moral support and key insights into Lie theory.
\part{Equivariant automorphisms of the Cox construction}\label{part:gen_theory}
\section{Generalized fans}\label{sec:gen_fans}
\subsection{Fans and generalized fans}\label{subsec:fans_and_gen_fans} 
Conventions and notation of this section mostly follow exposition in \cite[\S 2.1]{bupa15}.
\begin{definition}[{\cite[Definition 2.1.1]{bupa15}}]\label{def:conv_cone}
    Let \(V\) be a finite-dimensional vector space over \(\R\). Then, a finite collection of vectors \(v_1,\ldots v_k\in V\) defines a \emph{convex polyhedral cone}, or just a \emph{cone} as follows.
    \[
        C=\R_{\ge 0}\left\la v_1,\ldots, v_k\right\ra=\left\{\lambda_1 v_1+\ldots+\lambda_k v_k\mid \lambda_i\in \R_{\ge 0}\right\}.
    \]
    A cone is \emph{strongly convex} if it does not contain a line. A cone is \emph{simplicial} if it is generated by a part of a basis of \(V\).
\end{definition}
\begin{definition}\label{def:supp}
    Consider a subset \(S\sse V\) of a real finite-dimensional vector space \(V\) and an affine hyperplane \(A\) in \(V\) given by the following formula. \[A=\{v\in V\mid \la v, u\ra =b\}.\] 
    Define a pair of closed affine half-spaces in \(V\) denoted by \(A_+\) and \(A_-\) as follows.
    \[
        A_+=\{v\in V\mid \la v, u\ra \ge b\},\quad A_-=\{v\in V\mid \la v, u\ra \le b\}.
    \]
    Then the hyperplane \(A\) is called a  \emph{support hyperplane} for the subset \(S\) if \(S\cap A\) is nonempty, and either \(S\sse A_+\) or~\(S\sse A_-\). A half-space as above that contains \(S\) is called a \emph{support half-space}.
\end{definition}
\begin{remark}
    Note that a convex cone of Definition \ref{def:conv_cone} is an intersection of finitely many support half-spaces in the sense of Definition \ref{def:supp}.
\end{remark}
\begin{definition}\label{def:emb_fan}
    An \emph{embedded fan} (or simply a \emph{fan}) is a finite collection \(\Sigma\) of \emph{strongly convex cones} in \(V\) such that the following two conditions are satisfied.
    \begin{enumerate}
        \item For any two cones \(C_1,C_2\in \Sigma\) their intersection \(C_1\cap C_2\) is a face of both \(C_1\) and \(C_2\);
        \item If \(C_1\in \Sigma\) is a cone, then all its faces belong to \(\Sigma\).
    \end{enumerate}
    A fan is \emph{simplicial} if all its cones are simplicial. Given a collection of strongly convex cones \(\Sc\), a fan \(\Sigma\) is said to be \emph{generated} by \(\Sc\) if it is the smallest (with respect to inclusion relation) fan that contains \(\Sc\). 
\end{definition}
The following definition is a slight refinement of the notion of a fan.
\begin{definition}\label{def:marked_fan}
    A \emph{marked embedded fan} (or simply a \emph{marked fan}) is a fan \(\Sigma\) together with a choice of a generating vector for each one-dimensional cone in \(\Sigma\). All properties of unmarked fans are immediately extended to marked fans by forgetting the marking. For example, if the underlying unmarked fan is simplicial, a marked fan is simplicial.
\end{definition}
In the sequel, all fans are assumed to be simplicial.
\begin{example}\label{ex:CP2_fan}
    We define a fan \(\Sigma_{\CP^2}\) in \(\R^2\) with the standard basis \(\{e_1,e_2\}\) to be generated by the following collection of cones. \[\Sc=\left\{C_1=\R_{\ge 0}\la e_1,e_2\ra,\quad C_2=\Rge\la e_2,-e_1-e_2\ra,\quad C_3=\Rge\la e_1,-e_1-e_2\ra\right\}.\]
    It is well known (see, e.g., \cite[\S 2]{CoxLittleSchenk}) that in toric geometry, this fan characterizes toric variety \(\CP^2\) with the standard torus action. 
    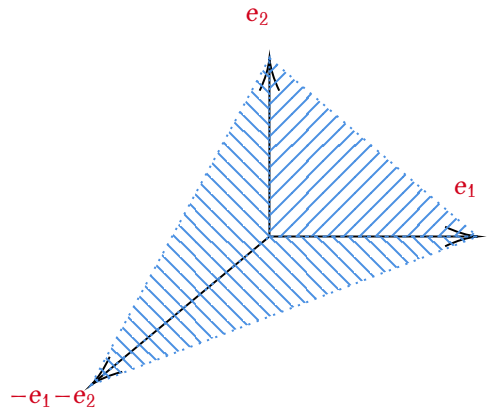
\begin{figure}[H]

 
\tikzset{
pattern size/.store in=\mcSize, 
pattern size = 5pt,
pattern thickness/.store in=\mcThickness, 
pattern thickness = 0.3pt,
pattern radius/.store in=\mcRadius, 
pattern radius = 1pt}
\makeatletter
\pgfutil@ifundefined{pgf@pattern@name@_j1h0aaw8v}{
\pgfdeclarepatternformonly[\mcThickness,\mcSize]{_j1h0aaw8v}
{\pgfqpoint{0pt}{0pt}}
{\pgfpoint{\mcSize+\mcThickness}{\mcSize+\mcThickness}}
{\pgfpoint{\mcSize}{\mcSize}}
{
\pgfsetcolor{\tikz@pattern@color}
\pgfsetlinewidth{\mcThickness}
\pgfpathmoveto{\pgfqpoint{0pt}{0pt}}
\pgfpathlineto{\pgfpoint{\mcSize+\mcThickness}{\mcSize+\mcThickness}}
\pgfusepath{stroke}
}}
\makeatother

 
\tikzset{
pattern size/.store in=\mcSize, 
pattern size = 5pt,
pattern thickness/.store in=\mcThickness, 
pattern thickness = 0.3pt,
pattern radius/.store in=\mcRadius, 
pattern radius = 1pt}
\makeatletter
\pgfutil@ifundefined{pgf@pattern@name@_hg9lkcmgt}{
\pgfdeclarepatternformonly[\mcThickness,\mcSize]{_hg9lkcmgt}
{\pgfqpoint{0pt}{-\mcThickness}}
{\pgfpoint{\mcSize}{\mcSize}}
{\pgfpoint{\mcSize}{\mcSize}}
{
\pgfsetcolor{\tikz@pattern@color}
\pgfsetlinewidth{\mcThickness}
\pgfpathmoveto{\pgfqpoint{0pt}{\mcSize}}
\pgfpathlineto{\pgfpoint{\mcSize+\mcThickness}{-\mcThickness}}
\pgfusepath{stroke}
}}
\makeatother

 
\tikzset{
pattern size/.store in=\mcSize, 
pattern size = 5pt,
pattern thickness/.store in=\mcThickness, 
pattern thickness = 0.3pt,
pattern radius/.store in=\mcRadius, 
pattern radius = 1pt}
\makeatletter
\pgfutil@ifundefined{pgf@pattern@name@_gjnyeutxh}{
\pgfdeclarepatternformonly[\mcThickness,\mcSize]{_gjnyeutxh}
{\pgfqpoint{0pt}{-\mcThickness}}
{\pgfpoint{\mcSize}{\mcSize}}
{\pgfpoint{\mcSize}{\mcSize}}
{
\pgfsetcolor{\tikz@pattern@color}
\pgfsetlinewidth{\mcThickness}
\pgfpathmoveto{\pgfqpoint{0pt}{\mcSize}}
\pgfpathlineto{\pgfpoint{\mcSize+\mcThickness}{-\mcThickness}}
\pgfusepath{stroke}
}}
\makeatother
\tikzset{every picture/.style={line width=0.75pt}} 

\begin{tikzpicture}[x=0.75pt,y=0.75pt,yscale=-1,xscale=1,scale=1.5]

\draw    (240,200) -- (308,200) ;
\draw [shift={(310,200)}, rotate = 180] [color={rgb, 255:red, 0; green, 0; blue, 0 }  ][line width=0.75]    (10.93,-3.29) .. controls (6.95,-1.4) and (3.31,-0.3) .. (0,0) .. controls (3.31,0.3) and (6.95,1.4) .. (10.93,3.29)   ;
\draw    (240,200) -- (240,142) ;
\draw [shift={(240,140)}, rotate = 90] [color={rgb, 255:red, 0; green, 0; blue, 0 }  ][line width=0.75]    (10.93,-3.29) .. controls (6.95,-1.4) and (3.31,-0.3) .. (0,0) .. controls (3.31,0.3) and (6.95,1.4) .. (10.93,3.29)   ;
\draw    (240,200) -- (181.54,248.72) ;
\draw [shift={(180,250)}, rotate = 320.19] [color={rgb, 255:red, 0; green, 0; blue, 0 }  ][line width=0.75]    (10.93,-3.29) .. controls (6.95,-1.4) and (3.31,-0.3) .. (0,0) .. controls (3.31,0.3) and (6.95,1.4) .. (10.93,3.29)   ;
\draw  [color={rgb, 255:red, 74; green, 144; blue, 226 }  ,draw opacity=0.13 ][pattern=_j1h0aaw8v,pattern size=6pt,pattern thickness=0.75pt,pattern radius=0pt, pattern color={rgb, 255:red, 74; green, 144; blue, 226}][dash pattern={on 0.84pt off 2.51pt}] (240,140) -- (310,200) -- (240,200) -- (240,140) -- cycle ;
\draw  [color={rgb, 255:red, 74; green, 144; blue, 226 }  ,draw opacity=0.13 ][pattern=_hg9lkcmgt,pattern size=6pt,pattern thickness=0.75pt,pattern radius=0pt, pattern color={rgb, 255:red, 74; green, 144; blue, 226}][dash pattern={on 0.84pt off 2.51pt}] (240,200) -- (310,200) -- (180,250) -- cycle ;
\draw  [color={rgb, 255:red, 74; green, 144; blue, 226 }  ,draw opacity=0.11 ][pattern=_gjnyeutxh,pattern size=6pt,pattern thickness=0.75pt,pattern radius=0pt, pattern color={rgb, 255:red, 74; green, 144; blue, 226}][dash pattern={on 0.84pt off 2.51pt}] (240,200) -- (240,140) -- (180,250) -- cycle ;

\draw (231,122.4) node [anchor=north west][inner sep=0.75pt]    {$\textcolor[rgb]{0.82,0.01,0.11}{e}\textcolor[rgb]{0.82,0.01,0.11}{_{2}}$};
\draw (301,180.4) node [anchor=north west][inner sep=0.75pt]    {$\textcolor[rgb]{0.82,0.01,0.11}{e}\textcolor[rgb]{0.82,0.01,0.11}{_{1}}$};
\draw (151,250.4) node [anchor=north west][inner sep=0.75pt]    {$\textcolor[rgb]{0.82,0.01,0.11}{-e}\textcolor[rgb]{0.82,0.01,0.11}{_{1}}\textcolor[rgb]{0.82,0.01,0.11}{-e}\textcolor[rgb]{0.82,0.01,0.11}{_{2}}$};

\end{tikzpicture}
    \caption{The standard fan of the toric variety \(\CP^2\).}
    \end{figure}
\end{example}
\begin{definition}
    A simplicial complex on a \emph{finite} set \(S\) is a collection of subsets \(K\) of the set \(S\) such that for any \(\sigma \in K\) any subset \(\tau\sse \sigma\) also belongs to \(K.\) The elements of \(S\) are called \emph{vertices} of \(K.\) The elements of \(K\) are called \emph{faces} of \(K\). If for \(s\in S\) the one-element subset \(\{s\}\) does not belong to \(K,\) then we call \(s\) a \emph{ghost vertex}. A face of a simplicial complex \(K\) is simply an element of \(K.\)
\end{definition}
The following definition is a generalization of the notion of a fan of Definition \ref{def:emb_fan}.
\begin{definition}\label{def:gen_fan}
    Let \(V\) be a finite-dimensional real vector space. Let \(K\) be a simplicial complex on a set \(S.\) Then, a \emph{generalized fan} \(\Sigma\) is a data of a map \(\rho\) that assigns to each vertex \(s\) of \(K\) a vector \(\rho(s)\) in \(V.\) We denote this object by a quadruple \(\Sigma=(S,K,V,\rho).\) A \emph{cone} of a generalized fan \(\Sigma=(S,K,V,\rho)\) associated to a \emph{face} \(\sigma\in K\) is a convex subset of \(V\) of the form
    \[
        C(\sigma)=\R_{\ge 0}\left\la \rho(s) \right\ra_{s\in \sigma}.
    \]
\end{definition}
\begin{definition}
    Let \(\Sigma=(S,K,V,\rho)\) and \(\Sigma'=(S',K',V',\rho')\) be two generalized fans.
    A \emph{morphism of generalized fans} \(\Phi:\Sigma\to \Sigma'\) is the following datum.
    \begin{enumerate}
        \item A simplicial morphism \(\Phi^c:(S,K)\to (S',K')\);
        \item A map of vector spaces \(\Phi^v:V\to V'\) which is linear on each cone of the fan \(\Sigma\);
        \item A compatibility condition on these two maps: \(\Phi_v(C(s))\sse C(\rho'(\Phi_c(s)))\) for all \(s\in S\).
    \end{enumerate} 
\end{definition}

\begin{construction}
    Given a marked fan \(\Sigma\), one associates with it a generalized fan in the following way. The set \(S\) is defined to be the set of \(1\)-dimensional cones of \(\Sigma\). Simplicial complex \(K\) consists of subsets of \(1\)-dimensional cones that lie in the same cone of \(\Sigma\). The map \(\rho\) is given by associating to a \(1\)-dimensional cone its generator given by the marked structure. In the sequel, we will not distinguish between a marked fan and the corresponding generalized fan.
\end{construction}
\begin{construction}\label{con:pushforward}
    Given a generalized fan \(\Sigma=(S,K,V,\rho)\) in the vector space \(V\) and a linear map \(A:V\to W\) one defines a \emph{pushforward generalized fan} \(A_*\Sigma\) as follows.
    \[
        A_*\Sigma=(S,K,W,A\circ \rho).
    \]
    Note that the map \(A\) gives a canonical morphism of generalized fans \(\Sigma\to A_*\Sigma\).
\end{construction}
\begin{example}\label{ex:CP2_fan_proj}
    Consider the fan of Example \ref{ex:CP2_fan} and a linear map \(P\) that projects \(\R^2\) along the line \(l\) generated by~\(e_1+\sqrt{2}e_2\).
    \begin{figure}[H]

 
\tikzset{
pattern size/.store in=\mcSize, 
pattern size = 5pt,
pattern thickness/.store in=\mcThickness, 
pattern thickness = 0.3pt,
pattern radius/.store in=\mcRadius, 
pattern radius = 1pt}
\makeatletter
\pgfutil@ifundefined{pgf@pattern@name@_p94sq5pyz}{
\pgfdeclarepatternformonly[\mcThickness,\mcSize]{_p94sq5pyz}
{\pgfqpoint{0pt}{0pt}}
{\pgfpoint{\mcSize+\mcThickness}{\mcSize+\mcThickness}}
{\pgfpoint{\mcSize}{\mcSize}}
{
\pgfsetcolor{\tikz@pattern@color}
\pgfsetlinewidth{\mcThickness}
\pgfpathmoveto{\pgfqpoint{0pt}{0pt}}
\pgfpathlineto{\pgfpoint{\mcSize+\mcThickness}{\mcSize+\mcThickness}}
\pgfusepath{stroke}
}}
\makeatother

 
\tikzset{
pattern size/.store in=\mcSize, 
pattern size = 5pt,
pattern thickness/.store in=\mcThickness, 
pattern thickness = 0.3pt,
pattern radius/.store in=\mcRadius, 
pattern radius = 1pt}
\makeatletter
\pgfutil@ifundefined{pgf@pattern@name@_yxf10snax}{
\pgfdeclarepatternformonly[\mcThickness,\mcSize]{_yxf10snax}
{\pgfqpoint{0pt}{-\mcThickness}}
{\pgfpoint{\mcSize}{\mcSize}}
{\pgfpoint{\mcSize}{\mcSize}}
{
\pgfsetcolor{\tikz@pattern@color}
\pgfsetlinewidth{\mcThickness}
\pgfpathmoveto{\pgfqpoint{0pt}{\mcSize}}
\pgfpathlineto{\pgfpoint{\mcSize+\mcThickness}{-\mcThickness}}
\pgfusepath{stroke}
}}
\makeatother

 
\tikzset{
pattern size/.store in=\mcSize, 
pattern size = 5pt,
pattern thickness/.store in=\mcThickness, 
pattern thickness = 0.3pt,
pattern radius/.store in=\mcRadius, 
pattern radius = 1pt}
\makeatletter
\pgfutil@ifundefined{pgf@pattern@name@_nwi0sm1gc}{
\pgfdeclarepatternformonly[\mcThickness,\mcSize]{_nwi0sm1gc}
{\pgfqpoint{0pt}{-\mcThickness}}
{\pgfpoint{\mcSize}{\mcSize}}
{\pgfpoint{\mcSize}{\mcSize}}
{
\pgfsetcolor{\tikz@pattern@color}
\pgfsetlinewidth{\mcThickness}
\pgfpathmoveto{\pgfqpoint{0pt}{\mcSize}}
\pgfpathlineto{\pgfpoint{\mcSize+\mcThickness}{-\mcThickness}}
\pgfusepath{stroke}
}}
\makeatother
\tikzset{every picture/.style={line width=0.75pt}} 

\begin{tikzpicture}[x=0.75pt,y=0.75pt,yscale=-1,xscale=1,scale=1.5]

\draw    (240,200) -- (308,200) ;
\draw [shift={(310,200)}, rotate = 180] [color={rgb, 255:red, 0; green, 0; blue, 0 }  ][line width=0.75]    (10.93,-3.29) .. controls (6.95,-1.4) and (3.31,-0.3) .. (0,0) .. controls (3.31,0.3) and (6.95,1.4) .. (10.93,3.29)   ;
\draw    (240,200) -- (240,142) ;
\draw [shift={(240,140)}, rotate = 90] [color={rgb, 255:red, 0; green, 0; blue, 0 }  ][line width=0.75]    (10.93,-3.29) .. controls (6.95,-1.4) and (3.31,-0.3) .. (0,0) .. controls (3.31,0.3) and (6.95,1.4) .. (10.93,3.29)   ;
\draw    (240,200) -- (181.54,248.72) ;
\draw [shift={(180,250)}, rotate = 320.19] [color={rgb, 255:red, 0; green, 0; blue, 0 }  ][line width=0.75]    (10.93,-3.29) .. controls (6.95,-1.4) and (3.31,-0.3) .. (0,0) .. controls (3.31,0.3) and (6.95,1.4) .. (10.93,3.29)   ;
\draw  [color={rgb, 255:red, 74; green, 144; blue, 226 }  ,draw opacity=0.13 ][pattern=_p94sq5pyz,pattern size=6pt,pattern thickness=0.75pt,pattern radius=0pt, pattern color={rgb, 255:red, 74; green, 144; blue, 226}][dash pattern={on 0.84pt off 2.51pt}] (240,140) -- (310,200) -- (240,200) -- (240,140) -- cycle ;
\draw  [color={rgb, 255:red, 74; green, 144; blue, 226 }  ,draw opacity=0.13 ][pattern=_yxf10snax,pattern size=6pt,pattern thickness=0.75pt,pattern radius=0pt, pattern color={rgb, 255:red, 74; green, 144; blue, 226}][dash pattern={on 0.84pt off 2.51pt}] (240,200) -- (310,200) -- (180,250) -- cycle ;
\draw  [color={rgb, 255:red, 74; green, 144; blue, 226 }  ,draw opacity=0.11 ][pattern=_nwi0sm1gc,pattern size=6pt,pattern thickness=0.75pt,pattern radius=0pt, pattern color={rgb, 255:red, 74; green, 144; blue, 226}][dash pattern={on 0.84pt off 2.51pt}] (240,200) -- (240,140) -- (180,250) -- cycle ;
\draw [color={rgb, 255:red, 65; green, 117; blue, 5 }  ,draw opacity=1 ][line width=1.5]    (240,200) -- (280,140) ;
\draw [color={rgb, 255:red, 65; green, 117; blue, 5 }  ,draw opacity=1 ][line width=1.5]    (240,200) -- (200,260) ;

\draw (231,122.4) node [anchor=north west][inner sep=0.75pt]    {$\textcolor[rgb]{0.82,0.01,0.11}{e}\textcolor[rgb]{0.82,0.01,0.11}{_{2}}$};
\draw (301,180.4) node [anchor=north west][inner sep=0.75pt]    {$\textcolor[rgb]{0.82,0.01,0.11}{e}\textcolor[rgb]{0.82,0.01,0.11}{_{1}}$};
\draw (141,250.4) node [anchor=north west][inner sep=0.75pt]    {$\textcolor[rgb]{0.82,0.01,0.11}{-e}\textcolor[rgb]{0.82,0.01,0.11}{_{1}}\textcolor[rgb]{0.82,0.01,0.11}{-e}\textcolor[rgb]{0.82,0.01,0.11}{_{2}}$};
\draw (271,132.4) node [anchor=north west][inner sep=0.75pt]    {$\textcolor[rgb]{0.25,0.46,0.02}{l}$};

\end{tikzpicture}
        \caption{Fan \(\Sigma_{\CP^2}\) of Example \ref{ex:CP2_fan} with an irrational line \(l\).}
    \end{figure}
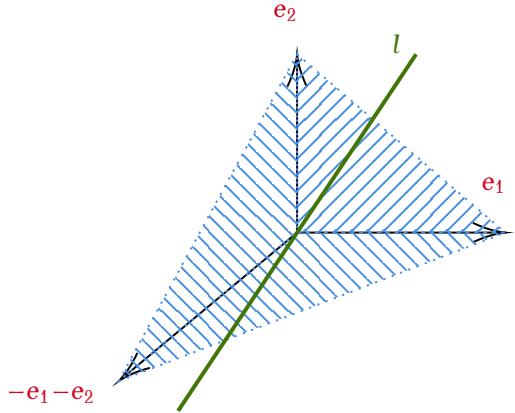 
    Then, the resulting pushforward generalized fan has the following structure.
    \begin{align*}
        P_*\Sigma=\left(\{0,1,2\}, \Delta^2, \R^2/ l, P\circ \rho(s)=\begin{cases}
            e_1+l,& s=0;\\
            e_2+l,& s=1;\\
            -e_1-e_2+l,& s=2.
        \end{cases}\right)
    \end{align*}
    The generalized fan \(P_*\Sigma\) does not arise from an embedded fan. Indeed, consider the cone \(C(\{0,1\})\) it is generated by \(e_1+l\) and \(e_2+l\) and as such contains the whole line \(l\). On the other hand, any embedded fan can only contain strongly convex cones.
    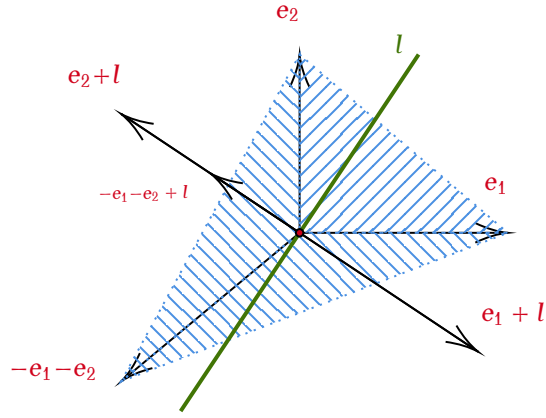
\begin{figure}[H]

 
\tikzset{
pattern size/.store in=\mcSize, 
pattern size = 5pt,
pattern thickness/.store in=\mcThickness, 
pattern thickness = 0.3pt,
pattern radius/.store in=\mcRadius, 
pattern radius = 1pt}
\makeatletter
\pgfutil@ifundefined{pgf@pattern@name@_t9wrkvyx1}{
\pgfdeclarepatternformonly[\mcThickness,\mcSize]{_t9wrkvyx1}
{\pgfqpoint{0pt}{0pt}}
{\pgfpoint{\mcSize+\mcThickness}{\mcSize+\mcThickness}}
{\pgfpoint{\mcSize}{\mcSize}}
{
\pgfsetcolor{\tikz@pattern@color}
\pgfsetlinewidth{\mcThickness}
\pgfpathmoveto{\pgfqpoint{0pt}{0pt}}
\pgfpathlineto{\pgfpoint{\mcSize+\mcThickness}{\mcSize+\mcThickness}}
\pgfusepath{stroke}
}}
\makeatother

 
\tikzset{
pattern size/.store in=\mcSize, 
pattern size = 5pt,
pattern thickness/.store in=\mcThickness, 
pattern thickness = 0.3pt,
pattern radius/.store in=\mcRadius, 
pattern radius = 1pt}
\makeatletter
\pgfutil@ifundefined{pgf@pattern@name@_i0ji3lfp4}{
\pgfdeclarepatternformonly[\mcThickness,\mcSize]{_i0ji3lfp4}
{\pgfqpoint{0pt}{-\mcThickness}}
{\pgfpoint{\mcSize}{\mcSize}}
{\pgfpoint{\mcSize}{\mcSize}}
{
\pgfsetcolor{\tikz@pattern@color}
\pgfsetlinewidth{\mcThickness}
\pgfpathmoveto{\pgfqpoint{0pt}{\mcSize}}
\pgfpathlineto{\pgfpoint{\mcSize+\mcThickness}{-\mcThickness}}
\pgfusepath{stroke}
}}
\makeatother

 
\tikzset{
pattern size/.store in=\mcSize, 
pattern size = 5pt,
pattern thickness/.store in=\mcThickness, 
pattern thickness = 0.3pt,
pattern radius/.store in=\mcRadius, 
pattern radius = 1pt}
\makeatletter
\pgfutil@ifundefined{pgf@pattern@name@_vtjutosti}{
\pgfdeclarepatternformonly[\mcThickness,\mcSize]{_vtjutosti}
{\pgfqpoint{0pt}{-\mcThickness}}
{\pgfpoint{\mcSize}{\mcSize}}
{\pgfpoint{\mcSize}{\mcSize}}
{
\pgfsetcolor{\tikz@pattern@color}
\pgfsetlinewidth{\mcThickness}
\pgfpathmoveto{\pgfqpoint{0pt}{\mcSize}}
\pgfpathlineto{\pgfpoint{\mcSize+\mcThickness}{-\mcThickness}}
\pgfusepath{stroke}
}}
\makeatother
\tikzset{every picture/.style={line width=0.75pt}} 

\begin{tikzpicture}[x=0.75pt,y=0.75pt,yscale=-1,xscale=1,scale=1.5]

\draw    (240,200) -- (308,200) ;
\draw [shift={(310,200)}, rotate = 180] [color={rgb, 255:red, 0; green, 0; blue, 0 }  ][line width=0.75]    (10.93,-3.29) .. controls (6.95,-1.4) and (3.31,-0.3) .. (0,0) .. controls (3.31,0.3) and (6.95,1.4) .. (10.93,3.29)   ;
\draw    (240,200) -- (240,142) ;
\draw [shift={(240,140)}, rotate = 90] [color={rgb, 255:red, 0; green, 0; blue, 0 }  ][line width=0.75]    (10.93,-3.29) .. controls (6.95,-1.4) and (3.31,-0.3) .. (0,0) .. controls (3.31,0.3) and (6.95,1.4) .. (10.93,3.29)   ;
\draw    (240,200) -- (181.54,248.72) ;
\draw [shift={(180,250)}, rotate = 320.19] [color={rgb, 255:red, 0; green, 0; blue, 0 }  ][line width=0.75]    (10.93,-3.29) .. controls (6.95,-1.4) and (3.31,-0.3) .. (0,0) .. controls (3.31,0.3) and (6.95,1.4) .. (10.93,3.29)   ;
\draw  [color={rgb, 255:red, 74; green, 144; blue, 226 }  ,draw opacity=0.13 ][pattern=_t9wrkvyx1,pattern size=6pt,pattern thickness=0.75pt,pattern radius=0pt, pattern color={rgb, 255:red, 74; green, 144; blue, 226}][dash pattern={on 0.84pt off 2.51pt}] (240,140) -- (310,200) -- (240,200) -- (240,140) -- cycle ;
\draw  [color={rgb, 255:red, 74; green, 144; blue, 226 }  ,draw opacity=0.13 ][pattern=_i0ji3lfp4,pattern size=6pt,pattern thickness=0.75pt,pattern radius=0pt, pattern color={rgb, 255:red, 74; green, 144; blue, 226}][dash pattern={on 0.84pt off 2.51pt}] (240,200) -- (310,200) -- (180,250) -- cycle ;
\draw  [color={rgb, 255:red, 74; green, 144; blue, 226 }  ,draw opacity=0.11 ][pattern=_vtjutosti,pattern size=6pt,pattern thickness=0.75pt,pattern radius=0pt, pattern color={rgb, 255:red, 74; green, 144; blue, 226}][dash pattern={on 0.84pt off 2.51pt}] (240,200) -- (240,140) -- (180,250) -- cycle ;
\draw [color={rgb, 255:red, 65; green, 117; blue, 5 }  ,draw opacity=1 ][line width=1.5]    (240,200) -- (280,140) ;
\draw [color={rgb, 255:red, 65; green, 117; blue, 5 }  ,draw opacity=1 ][line width=1.5]    (240,200) -- (200,260) ;
\draw [color={rgb, 255:red, 0; green, 0; blue, 0 }  ,draw opacity=1 ]   (180,160) -- (300,240) ;
\draw    (240,200) -- (298.34,238.89) ;
\draw [shift={(300,240)}, rotate = 213.69] [color={rgb, 255:red, 0; green, 0; blue, 0 }  ][line width=0.75]    (10.93,-3.29) .. controls (6.95,-1.4) and (3.31,-0.3) .. (0,0) .. controls (3.31,0.3) and (6.95,1.4) .. (10.93,3.29)   ;
\draw    (240,200) -- (181.66,161.11) ;
\draw [shift={(180,160)}, rotate = 33.69] [color={rgb, 255:red, 0; green, 0; blue, 0 }  ][line width=0.75]    (10.93,-3.29) .. controls (6.95,-1.4) and (3.31,-0.3) .. (0,0) .. controls (3.31,0.3) and (6.95,1.4) .. (10.93,3.29)   ;
\draw    (240,200) -- (211.66,181.11) ;
\draw [shift={(210,180)}, rotate = 33.69] [color={rgb, 255:red, 0; green, 0; blue, 0 }  ][line width=0.75]    (10.93,-3.29) .. controls (6.95,-1.4) and (3.31,-0.3) .. (0,0) .. controls (3.31,0.3) and (6.95,1.4) .. (10.93,3.29)   ;
\draw  [color={rgb, 255:red, 0; green, 0; blue, 0 }  ,draw opacity=0 ][fill={rgb, 255:red, 208; green, 2; blue, 27 }  ,fill opacity=1 ] (241.25,200) .. controls (241.25,199.31) and (240.69,198.75) .. (240,198.75) .. controls (239.31,198.75) and (238.75,199.31) .. (238.75,200) .. controls (238.75,200.69) and (239.31,201.25) .. (240,201.25) .. controls (240.69,201.25) and (241.25,200.69) .. (241.25,200) -- cycle ;

\draw (231,122.4) node [anchor=north west][inner sep=0.75pt]    {$\textcolor[rgb]{0.82,0.01,0.11}{e}\textcolor[rgb]{0.82,0.01,0.11}{_{2}}$};
\draw (301,180.4) node [anchor=north west][inner sep=0.75pt]    {$\textcolor[rgb]{0.82,0.01,0.11}{e}\textcolor[rgb]{0.82,0.01,0.11}{_{1}}$};
\draw (141,242.4) node [anchor=north west][inner sep=0.75pt]    {$\textcolor[rgb]{0.82,0.01,0.11}{-e}\textcolor[rgb]{0.82,0.01,0.11}{_{1}}\textcolor[rgb]{0.82,0.01,0.11}{-e}\textcolor[rgb]{0.82,0.01,0.11}{_{2}}$};
\draw (271,132.4) node [anchor=north west][inner sep=0.75pt]    {$\textcolor[rgb]{0.25,0.46,0.02}{l}$};
\draw (300,222.4) node [anchor=north west][inner sep=0.75pt]    {$\textcolor[rgb]{0.82,0.01,0.11}{e_{1} +l}$};
\draw (161,142.4) node [anchor=north west][inner sep=0.75pt]    {$\textcolor[rgb]{0.82,0.01,0.11}{e}\textcolor[rgb]{0.82,0.01,0.11}{_{2}}\textcolor[rgb]{0.82,0.01,0.11}{+l}$};
\draw (171.04,183.42) node [anchor=north west][inner sep=0.75pt]  [font=\scriptsize,rotate=-358.99]  {$\textcolor[rgb]{0.82,0.01,0.11}{-e_{1}}\textcolor[rgb]{0.82,0.01,0.11}{-e_{2} +l}$};

\end{tikzpicture}
        \caption{Projection generalized fan \(\Sigma_{\CP^2}/l\) of the fan \(\Sigma_{\CP^2}\) along the line \(l\).}
    \end{figure}
\end{example}
\begin{construction}\label{con:reduced_fan}
    Let \(\Sigma=\GF\) be a generalized fan. Let \(S^\red\) be the set of non-ghost vertices of \(K\). That is for all \(s\in S^\red\) there exists some \(\sigma\in K\) such that \(s\in \sigma\). Then, we define the \emph{reduced fan} \(\Sigma^\red\) associated to \(\Sigma\) as follows.
    \begin{align*}
        \Sigma^\red=(S^\red, K\vert_{S^\red}, V,\rho\vert_{S^\red}).
    \end{align*}
\end{construction}
\subsection{Complete fans}\label{subsec:complete_fans}
\begin{definition}
    A generalized fan \(\Sigma=(S,K,V,\rho)\) is \emph{complete}. If for any point \(v\in V\) there exists a face \(\sigma\in K\) such that~\(v\in C(\sigma).\)
\end{definition}
\begin{lemma}\label{lemma:positive_combination}
    Given a complete generalized fan $\Sigma=(S,K,V,\rho).$ The following identity holds for some positive real numbers $\lambda_s$.
    $$
        \sum_{s\in S}\lambda_s \cd \rho(s)=0,\quad\quad \lambda_s>0,\; \forall s\in S.
    $$
\end{lemma}
\begin{proof}
    Let \(s\in S\) be an element of the support set for the generalized fan \(\Sigma\). Since the fan $\Sigma$ is complete the opposite vector $-\rho(s)$ lies in some cone $C(\sigma)$ for \(\sigma\in K\). Thus $$
    -\rho(s)=\sum_{t\in S}\vk^{\rho(s)}_{t}\cd \rho(t),\quad \vk^\nu_t\ge 0,\quad \forall t\in S.$$
    Now we have a linear dependency with non-negative coefficients:
    $$
    \rho(s)+\sum_{t\in S}\vk^{\rho(s)}_{t}\cd \rho(t)=0.
    $$
    We can obtain such a dependency for each element \(s'\in S\). Thus for any $s'\in S$ we have:
    $$
        \rho(s')+\sum_{t\in S}\vk^{\rho(s')}_{t}\cd \rho(t)=0.
    $$
    Taking the sum of all such dependencies, we obtain the desired dependency with strictly positive coefficients:
    \[
        \sum_{s\in S}\left[\rho(s)+\sum_{t\in S}\vk^{\rho(s)}_t \cd\rho(t)\right]=0. \qedhere
    \]
    
\end{proof}
\begin{lemma}\label{lem:comb_ghost}
    Let \(\Sigma=\GF\) be a complete generalized fan. Then, for any ghost vertex \(s'\in S\setminus S^\red\), there exists a linear combination with non-negative coefficients of the following form for some non-negative real numbers \(\lambda^{s'}_s.\)
    \[
        \sum_{s\in S^\red} \lambda_s^{s'} \rho(s)=\rho(s'),\quad \lambda_s^{s'}\ge 0.
    \]
\end{lemma}
\begin{proof}
    By definition of a complete fan the vector \(\rho(s')\) belongs to some cone \(C(\sigma)\) for \(\sigma \in K\). Consequently, a non-negative linear combination of the following form exists.
    \[
        \sum_{s\in\sigma}\tau_s\rho(s)=\rho(s'),\quad \tau_s\ge 0.
    \]
    The desired linear combination is now obtained by defining \(\lambda_s^{s'}\) as follows.
    \[
        \lambda_s^{s'}=\begin{cases}
            \tau_s,& s\in \sigma;\\
            0,& s\notin \sigma.
        \end{cases}
        \qedhere
    \]
\end{proof}
\begin{remark}
    By the same argument, the vector \(-\rho(s)\) can also be presented as a non-negative combination of \(\rho(s)\) for~\(s\in S^\red\). 
\end{remark}
\begin{definition}\label{def:strongly_complete}
    A generalized fan \(\Sigma=\GF\) is \emph{strongly complete} if for any \(m\in V^\vee\) such that there exists a unique \(s\in S\) with \(\la m,\rho(s)\ra>0\) and any \(\sigma\in K\) with \(s\notin \sigma\) there exists \(\tau\in K\) such that 
    \[
        s\cup \{s'\in \sigma\mid \la m,\rho(s')\ra =0\}\sse \tau.
    \]
\end{definition}
\begin{lemma}
    Let \(\Sigma=\GF\) be a complete generalized fan associated to an embedded complete fan. Then for any linear operator \(A:V\to W\), the pushforward fan \(A_*\Sigma\) is strongly complete.
\end{lemma}
\begin{proof}
    Let \(m\in W^\vee\) be such that there exists a unique \(s\in S\) with \(\la m,A\circ\rho(s)\ra>0\). And let \(\sigma\in K\) be an arbitrary face that does not contain \(s.\) Clearly, the pullback covector \(A^*m\) satisfies the property that for all elements \(s'\in S\setminus \{s\}\) the value of \(A^*m\) is non-positive. Thus, by assumption and \cite[Proposition 1.1]{Oda} there exists a face \(\tau\) with the desired properties. 
\end{proof}
\begin{definition}
    Let \(\Sigma=\GF\) be a generalized fan. We say that \(\wh{\Sigma}\) is a \emph{development} of \(\Sigma\) if it is an embedded fan with a morphism of generalized fans \(f:\wh{\Sigma}\to \Sigma\) such that \(f_*\wh{\Sigma}=\Sigma\).
\end{definition}
\begin{remark}
    The terminology here is borrowed from the theory of complexes of groups, see \cite[\S 12.15]{Bridson_Haefliger}.
    There, a complex of groups is called developable if it can be presented as a quotient of a genuine simplicial complex with respect to a group action.
\end{remark}
\begin{construction}\label{con:can_developement}
    Any generalized fan \(\Sigma=\GF\) admits a canonical developement \(\wt{\Sigma}=(\wt{S},\wt{K},\wt{V},\wt{\rho})\) defined as follows.
    \begin{itemize}
        \item The set \(\wt{S}\) is just the set \(S\). 
        \item The simplicial complex \(\wt{K}\) is just \(K\).
        \item The vector space \(\wt{V}\) is the vector space \(\R^S\).
        \item The map \(\wt{\rho}\) sends \(s\) to \(e_s\).
    \end{itemize}
\end{construction}
\begin{lemma}
    The generalized fan \(\wt{\Sigma}\) of Construction \ref{con:can_developement} is a development of the generalized fan \(\Sigma.\)
\end{lemma}
\begin{proof}
    Define the map \(\wt{\Sigma}\to \Sigma\) to be the canonical linear map \(\R^S\to V\) associated with \(\rho.\)
    The result is immediate from the definitions. 
\end{proof}
\begin{example}
    Any generalized fan \(\Sigma=\GF\) can be pushed forward to a strongly complete fan along the zero linear map. Indeed, the condition of Definition \ref{def:strongly_complete} becomes void when the ambient space is zero because there are no covectors \(m\) such that \(\la m,0\ra>0\).
\end{example}
\subsection{Fan operators and tori}\label{subsec:fan_tori}
\begin{definition}
    Given a generalized fan \(\Sigma=(S,K,V,\rho).\) We associate with it two linear operators.
    \begin{align}
        A_\Sigma^\R:\bigoplus_{s\in S}\R\la e_s\ra=\R^S\to V,\quad e_s\mapsto \rho(s);\\
        A_\Sigma^\R\otimes \C=A_\Sigma^\C:\bigoplus_{s\in S}\C\la e_s\ra=\C^S\to V\otimes\C,\quad e_s\mapsto \rho(s).
    \end{align}
    These two operators give rise to two vector spaces:
    \begin{align}
        \wt{H}^\R_\Sigma:=\ker A_\Sigma^\R;\label{eq:HStilde}\\
        \wt{G}_\Sigma:=\ker A_\Sigma^\C.\label{eq:GStilde}
    \end{align}
\end{definition}
\begin{definition}
    There are two natural exponential homomorphisms: the real one and the complex one. They are defined coordinatewise as follows.
    \begin{align}
        \exp_\R:\R^S\to \R_{>0}^S,\quad (x_s)_{s\in S}\mapsto (e^{x_s})_{s\in S}\label{eq:exp_R};\\
        \exp_\C:\C^S\to (\CT)^S,\quad (x_s)_{s\in S}\mapsto (e^{x_s})_{s\in S}.\label{eq:exp_C}
    \end{align}
\end{definition}
\begin{definition}\label{def:fan_lattice}
    Let \(\Sigma=(S,K,V,\rho)\) be a generalized fan. Denote by \(\Gamma(\Sigma)\) the additive subgroup of \(V\) generated by the set \(\rho(S).\) We call this group the generalized lattice of the generalized fan \(\Sigma.\) The automorphism group of \(\Gamma(\Sigma)\) is defined as follows. \[\Aut(\Gamma(\Sigma))=\left\{g\in GL(V)\mid g\Gamma(\Sigma)=\Gamma(\Sigma)\right\}\]
\end{definition}
\begin{definition}\label{def:tori}
    The two vector spaces \(\wt{H}_\Sigma^\R, \wt{G}_\Sigma\) (see \eqref{eq:HStilde}, \eqref{eq:GStilde}) together with the exponential homomorphisms defined above (see \eqref{eq:exp_R}, \eqref{eq:exp_C}) give rise to two Lie groups:
    \begin{align}
        H^\R_\Sigma:=\exp_\R(\wt{H}_\Sigma)\sse (\R^\times)^S;\\
        G_\Sigma:=\exp_\C(\wt{G}_\Sigma)\sse (\CT)^S.
    \end{align}
    These two groups are referred to as \emph{fan tori}. More specifically, we refer to \(H_\Sigma^\R\) as the \emph{real fan torus}, and to \(G_\Sigma\) as the complex fan torus. This notation comes from the fact that both groups are Lie subgroups of algebraic tori over \(\R\) and~\(\C\), respectively. 
\end{definition}
\begin{remark}
    Note that the group \(H_\Sigma^\R\) is a \emph{real form} of the group \(G_\Sigma.\)
\end{remark}
\begin{construction}
    An alternative description of \(\Gamma(\Sigma)\) of Definition \ref{def:fan_lattice} is the that \(\Gamma(\Sigma)\) is the image of \(\Z^S\) under the map \(A_\Sigma^\R\). In particular, we see that \(\Gamma(\Sigma)\) is a topological abelian group with the topology coming from the ambient real vector space \(V\).

    We define the \emph{dual lattice} \(\Gamma(\Sigma)^\vee\) of a generalized fan \(\Sigma\). As the topological abelian group of \emph{continuous homomorphisms} from \(\Gamma(\Sigma)\) into the group \(\Z\) with the discrete topology. 

    We also denote by \(A_\Sigma^\Z\) the map \(A_\Sigma^\R\) restricted to \(\Z^S\) and corestricted to \(\Gamma(\Sigma)\).
\end{construction}
\begin{lemma}\label{lem:integral_tori_dual_map}
    The following formula gives the map of abelian groups dual to \(A_\Sigma^\Z\).
    \[
        (A_\Sigma^Z)^\vee:\Gamma(\Sigma)^\vee\to (\Z^S)^\vee,\quad \gamma\xmapsto{(A_\Sigma^\Z)^\vee} \sum_{s\in S}\la \gamma, A_{\Sigma}^\R(e_s)\ra e_s^*\in (\Z^S)^\vee.
    \]
\end{lemma}
\begin{proof}
    The result is immediate from the definitions.
\end{proof}
\subsection{Complex structures on real tori}\label{subsec:complex_struct_on_tori}
\begin{construction}\label{con:complex_struct}
    If the dimension of the vector space \(\wt{H}_\Sigma^\R\) is even, one can introduce a complex structure on it. Concretely, let \(\wt{H}_\Sigma^\R\) be of \emph{real dimension} \(2l\). We consider a complex linear subspace \(\mf{h}_\Sigma\sse \C^S\) of \emph{complex dimension} \(l\) that project isomorphically to~\(\wt{H}_\Sigma\) under the map \(\Re\).  We denote the image of \(\mf{h}_\Sigma\) under the complex exponential map by~\(H_\Sigma\).
\end{construction}
\begin{remark}
    Note that as a real Lie group \(H_\Sigma\) is isomorphic to \(H_\Sigma^\R\) via the composite map \(\exp_\R\circ\Re\).
\end{remark}
\begin{definition}
    Groups \(H_\Sigma\) and \(G_\Sigma\) give rise to the quotient complex Lie group that we denote \(F_\Sigma:=G_\Sigma/H_\Sigma.\)
\end{definition}
The following example is due to Buchstaber--Panov, see \cite[Example 6.6.2]{bupa15}.
\begin{example}
    Consider a generalized fan of the form \((S,\emptyset,0,\ul{0})\) where \(S\) is an arbitrary finite set with an even number of elements. Then the fan torus \(G_\Sigma\) is isomorphic to the complex torus \((\CT)^S.\) The real torus \(H_\Sigma^\R\) is isomorphic to the real torus \((\R^\times)^S.\) The quotient torus \(F_\Sigma\) is isomorphic to a compact torus \((\CT)^S/(\R^\times)^S.\) Moreover, by varying the complex structure on \(H_\Sigma^\R\) in Construction \ref{con:complex_struct} one can obtain all compact complex tori of dimension \(\frac{|S|}{2}\). A proof of this fact can be found in \cite[Example 6.6.2]{bupa15}.
\end{example}
\subsection{The Cox construction of a fan}\label{subsec:Cox_construction}
\begin{construction}\label{con:uni_space}
    Let \(\Sigma=(S,K,V,\rho)\) be a generalized fan. Below we describe the universal space of the generalized fan \(\Sigma\) and a collection of open subsets of it associated to the faces of the simplicial complex \(K.\)
    \begin{enumerate}
        \item\label{con:uni_space_1} Then, we associate a quasiaffine complex variety with it as follows.
        \begin{align}
                U(\Sigma)=\C^S\setminus \bigcup_{\sigma\in K} \left\{(z_s)_{s\in S}\mid z_s=0 \iff s\notin \sigma\right\}.
        \end{align}
        \item\label{con:uni_space_2} Let \(\sigma\in K\) be a face of the simplicial complex \(K.\) Then we define the following open subset of \(U(\Sigma)\).
         \[
            U(\Sigma,\sigma)=\left\{(z_s)_{s\in S}\in U(\Sigma)\mid  s\notin \sigma \Rightarrow z_s\neq 0\right\}.\]
    \end{enumerate}
    By a slight abuse of notation, we will often refer to the variety (analytical space) \(U(\Sigma)\) as the Cox construction of the generalized fan \(\Sigma.\)
\end{construction}
Observe that the groups \(H_\Sigma, G_\Sigma\) act on \(\C^S\) via the embeddings \(H_\Sigma\hookrightarrow G_\Sigma\hookrightarrow (\CT)^S.\)
Note that the subset \(U(\Sigma)\sse \C^S\) is invariant with respect to these actions.
\begin{lemma}\label{lem:uni_space_struct}
    Let \(\Sigma=(S,K,V,\rho)\) be a generalized fan. Then the universal space \(U(\Sigma)\) of Construction \ref{con:uni_space}.\ref{con:uni_space_1} splits into a union of affine algebraic varieties \(U(\Sigma,\sigma)\) of part \ref{con:uni_space_2} as follows.
    \[
        U(\Sigma)=\bigcup_{\sigma\in K} U(\Sigma,\sigma).
    \]
\end{lemma}
\begin{proof}
    This result is standard in toric geometry. The proof can be found, for instance, in monographs \cite[\S 5.1]{CoxLittleSchenk} and \cite[\S 5.4]{bupa15}.
\end{proof}
\subsection{Rational fans and algebraicity}\label{subsec:rationality}
\begin{definition}\label{def:rat_fan}
    Let \(\Sigma=(S,K,V,\rho)\) be a generalized fan. Then \(\Sigma\) is \emph{rational} if the image of \(\Z^S\) under \(\rho\) is a discrete subgroup of \(V\). Here, we consider \(V\) as a real Lie group.
\end{definition}
We now describe the quotient construction from toric geometry in terms of the above definitions. The following result appeared in many sources; see \cite[\S 2]{Cox_Aut} for an early instance and a discussion of other references. For a textbook treatment see \cite[\S 5.1]{CoxLittleSchenk}, \cite[\S 5.4]{bupa15}.
\begin{theorem}
    Let \(\Sigma\) be a rational marked embedded fan in the sense of Definition \ref{def:marked_fan}. Then, the quotient analytic space~\(U(\Sigma)/G_\Sigma\) has a natural structure of a (possibly singular) algebraic variety. It is denoted by \(V_\Sigma\) and is called the \emph{toric variety} associated to the fan \(\Sigma.\)
\end{theorem}
\begin{proposition}
    The toric variety associated to the canonical developement \(\wt{\Sigma}\) of a generalized fan \(\Sigma\), see Construction~\ref{con:can_developement}, is precisely the Cox construction of the generalized fan \(\Sigma.\)
\end{proposition}
\begin{proof}
    This is evident from the definitions. 
\end{proof}
\begin{construction}[Rationalization]\label{con:rationalization}
    Given a generalized fan \(\Sigma\) one defines its \emph{rationalization} \(\ul{\Sigma}\) as follows. 
    \begin{itemize}
        \item Let \(W\sse \R^S\) be a linear subspace, we denote by \(\ol{W}\) the smallest subspace generated by rational vectors which contains \(W\). 
        \item The generalized fan \(\ul{\Sigma}\) is defined as the pushforward fan (see Construction \ref{con:pushforward} for the definition) with respect to the projection map along the subspace \(\ol{\ker A^\R_\Sigma}\).
    \end{itemize}
    By Construction \ref{con:pushforward} there is a canonical morphism \(\Sigma\to \ul{\Sigma}\). We call this morphism the \emph{rationalizing projection} of the generalized fan \(\Sigma\), and denote it by \(\ul{\pi}_\Sigma\). By abuse of notation we will often denote by \(\ul{\pi}_\Sigma\) the map \(\ul{\pi}_\Sigma:V\to \ul{V}\) as well.
    
    More concretely, the fan \(\ul{\Sigma}\) can be described as follows. Let \(\Sigma=(S,K,V,\rho)\) be a generalized fan in the space \(V\). Then \(\ul{\Sigma}\) is given by the formula below.
    \begin{align}
        \ul{\Sigma}=\left(S,K,\R^S/\ol{\ker A^\R_\Sigma}, \ul{\rho}(s)=e_s+\ol{\ker A^\R_\Sigma}\quad \forall s\in S\right)
    \end{align}
\end{construction}
\begin{example}\label{ex:CP2_fan_proj_rat}
    Consider the fan \(\Sigma_{\CP^2}\) of Example \ref{ex:CP2_fan} then the image of \(\Z^{0,1,2}\) under the map \(A_\Sigma^\Z\) is already discrete (it is the standard lattice \(\Z^2\) in \(\R^2\)). Hence, the rationalization \(\ul{\Sigma_{\CP^2}}\) is naturally isomorphic to \(\Sigma_{\CP^2}\). 
    
    Now consider the projected fan \(\Sigma_{\CP^2}/l\) of Example \ref{ex:CP2_fan_proj}. Denote by \(L\) the vector subspace of \(\R^3\) spanned by \(\left\{
    \begin{pmatrix}
        1\\
        1\\
        1
    \end{pmatrix},
    \begin{pmatrix}
        1\\
        \sqrt{2}\\
        0
    \end{pmatrix}
    \right\}\). By Construction \ref{con:rationalization}, the rationalization of the fan \(\Sigma_{\CP^2}/l\) can be described as the generalized fan with the same simplicial complex as \(\Sigma_{\CP^2}/l\) and the vector space given by the follow quotient. \[\ul{V}=\R^{\{0,1,2\}}/\ol{L}.\]
    Clearly, the space \(\ol{L}\) coincides with \(\R^3\). Hence, the rationalization \(\ul{\Sigma_{\CP^2}/l}\) has a \(0\)-dimensional vector space, and all fan operators are zero.
    \begin{figure}[H]

 
\tikzset{
pattern size/.store in=\mcSize, 
pattern size = 5pt,
pattern thickness/.store in=\mcThickness, 
pattern thickness = 0.3pt,
pattern radius/.store in=\mcRadius, 
pattern radius = 1pt}
\makeatletter
\pgfutil@ifundefined{pgf@pattern@name@_cwn3ev3md}{
\pgfdeclarepatternformonly[\mcThickness,\mcSize]{_cwn3ev3md}
{\pgfqpoint{0pt}{0pt}}
{\pgfpoint{\mcSize+\mcThickness}{\mcSize+\mcThickness}}
{\pgfpoint{\mcSize}{\mcSize}}
{
\pgfsetcolor{\tikz@pattern@color}
\pgfsetlinewidth{\mcThickness}
\pgfpathmoveto{\pgfqpoint{0pt}{0pt}}
\pgfpathlineto{\pgfpoint{\mcSize+\mcThickness}{\mcSize+\mcThickness}}
\pgfusepath{stroke}
}}
\makeatother

 
\tikzset{
pattern size/.store in=\mcSize, 
pattern size = 5pt,
pattern thickness/.store in=\mcThickness, 
pattern thickness = 0.3pt,
pattern radius/.store in=\mcRadius, 
pattern radius = 1pt}
\makeatletter
\pgfutil@ifundefined{pgf@pattern@name@_sovdf35p2}{
\pgfdeclarepatternformonly[\mcThickness,\mcSize]{_sovdf35p2}
{\pgfqpoint{0pt}{-\mcThickness}}
{\pgfpoint{\mcSize}{\mcSize}}
{\pgfpoint{\mcSize}{\mcSize}}
{
\pgfsetcolor{\tikz@pattern@color}
\pgfsetlinewidth{\mcThickness}
\pgfpathmoveto{\pgfqpoint{0pt}{\mcSize}}
\pgfpathlineto{\pgfpoint{\mcSize+\mcThickness}{-\mcThickness}}
\pgfusepath{stroke}
}}
\makeatother

 
\tikzset{
pattern size/.store in=\mcSize, 
pattern size = 5pt,
pattern thickness/.store in=\mcThickness, 
pattern thickness = 0.3pt,
pattern radius/.store in=\mcRadius, 
pattern radius = 1pt}
\makeatletter
\pgfutil@ifundefined{pgf@pattern@name@_qbzyr2g37}{
\pgfdeclarepatternformonly[\mcThickness,\mcSize]{_qbzyr2g37}
{\pgfqpoint{0pt}{-\mcThickness}}
{\pgfpoint{\mcSize}{\mcSize}}
{\pgfpoint{\mcSize}{\mcSize}}
{
\pgfsetcolor{\tikz@pattern@color}
\pgfsetlinewidth{\mcThickness}
\pgfpathmoveto{\pgfqpoint{0pt}{\mcSize}}
\pgfpathlineto{\pgfpoint{\mcSize+\mcThickness}{-\mcThickness}}
\pgfusepath{stroke}
}}
\makeatother
\tikzset{every picture/.style={line width=0.75pt}} 

\begin{tikzpicture}[x=0.75pt,y=0.75pt,yscale=-1,xscale=1,scale=1]

\draw    (240,200) -- (308,200) ;
\draw [shift={(310,200)}, rotate = 180] [color={rgb, 255:red, 0; green, 0; blue, 0 }  ][line width=0.75]    (10.93,-3.29) .. controls (6.95,-1.4) and (3.31,-0.3) .. (0,0) .. controls (3.31,0.3) and (6.95,1.4) .. (10.93,3.29)   ;
\draw    (240,200) -- (240,142) ;
\draw [shift={(240,140)}, rotate = 90] [color={rgb, 255:red, 0; green, 0; blue, 0 }  ][line width=0.75]    (10.93,-3.29) .. controls (6.95,-1.4) and (3.31,-0.3) .. (0,0) .. controls (3.31,0.3) and (6.95,1.4) .. (10.93,3.29)   ;
\draw    (240,200) -- (181.54,248.72) ;
\draw [shift={(180,250)}, rotate = 320.19] [color={rgb, 255:red, 0; green, 0; blue, 0 }  ][line width=0.75]    (10.93,-3.29) .. controls (6.95,-1.4) and (3.31,-0.3) .. (0,0) .. controls (3.31,0.3) and (6.95,1.4) .. (10.93,3.29)   ;
\draw  [color={rgb, 255:red, 74; green, 144; blue, 226 }  ,draw opacity=0.13 ][pattern=_cwn3ev3md,pattern size=6pt,pattern thickness=0.75pt,pattern radius=0pt, pattern color={rgb, 255:red, 74; green, 144; blue, 226}][dash pattern={on 0.84pt off 2.51pt}] (240,140) -- (310,200) -- (240,200) -- (240,140) -- cycle ;
\draw  [color={rgb, 255:red, 74; green, 144; blue, 226 }  ,draw opacity=0.13 ][pattern=_sovdf35p2,pattern size=6pt,pattern thickness=0.75pt,pattern radius=0pt, pattern color={rgb, 255:red, 74; green, 144; blue, 226}][dash pattern={on 0.84pt off 2.51pt}] (240,200) -- (310,200) -- (180,250) -- cycle ;
\draw  [color={rgb, 255:red, 74; green, 144; blue, 226 }  ,draw opacity=0.11 ][pattern=_qbzyr2g37,pattern size=6pt,pattern thickness=0.75pt,pattern radius=0pt, pattern color={rgb, 255:red, 74; green, 144; blue, 226}][dash pattern={on 0.84pt off 2.51pt}] (240,200) -- (240,140) -- (180,250) -- cycle ;
\draw [color={rgb, 255:red, 65; green, 117; blue, 5 }  ,draw opacity=1 ][line width=1.5]    (240,200) -- (280,140) ;
\draw [color={rgb, 255:red, 65; green, 117; blue, 5 }  ,draw opacity=1 ][line width=1.5]    (240,200) -- (200,260) ;
\draw [color={rgb, 255:red, 155; green, 155; blue, 155 }  ,draw opacity=1 ] [dash pattern={on 0.84pt off 2.51pt}]  (180,250) .. controls (173.35,295.12) and (205.77,324.41) .. (209.88,368.65) ;
\draw [shift={(210,370)}, rotate = 265.46] [color={rgb, 255:red, 155; green, 155; blue, 155 }  ,draw opacity=1 ][line width=0.75]    (10.93,-3.29) .. controls (6.95,-1.4) and (3.31,-0.3) .. (0,0) .. controls (3.31,0.3) and (6.95,1.4) .. (10.93,3.29)   ;
\draw [color={rgb, 255:red, 155; green, 155; blue, 155 }  ,draw opacity=1 ] [dash pattern={on 0.84pt off 2.51pt}]  (240,140) .. controls (266.73,172.1) and (243.47,335.68) .. (210.99,369.03) ;
\draw [shift={(210,370)}, rotate = 316.79] [color={rgb, 255:red, 155; green, 155; blue, 155 }  ,draw opacity=1 ][line width=0.75]    (10.93,-3.29) .. controls (6.95,-1.4) and (3.31,-0.3) .. (0,0) .. controls (3.31,0.3) and (6.95,1.4) .. (10.93,3.29)   ;
\draw [color={rgb, 255:red, 155; green, 155; blue, 155 }  ,draw opacity=1 ] [dash pattern={on 0.84pt off 2.51pt}]  (310,200) .. controls (319.19,231.54) and (260.76,383.07) .. (211.49,370.43) ;
\draw [shift={(210,370)}, rotate = 18.04] [color={rgb, 255:red, 155; green, 155; blue, 155 }  ,draw opacity=1 ][line width=0.75]    (10.93,-3.29) .. controls (6.95,-1.4) and (3.31,-0.3) .. (0,0) .. controls (3.31,0.3) and (6.95,1.4) .. (10.93,3.29)   ;
\draw [color={rgb, 255:red, 155; green, 155; blue, 155 }  ,draw opacity=1 ] [dash pattern={on 0.84pt off 2.51pt}]  (240,200) .. controls (255.25,215.84) and (216.67,326.68) .. (212.42,368.15) ;
\draw [shift={(212.25,370)}, rotate = 274.63] [color={rgb, 255:red, 155; green, 155; blue, 155 }  ,draw opacity=1 ][line width=0.75]    (10.93,-3.29) .. controls (6.95,-1.4) and (3.31,-0.3) .. (0,0) .. controls (3.31,0.3) and (6.95,1.4) .. (10.93,3.29)   ;
\draw  [color={rgb, 255:red, 0; green, 0; blue, 0 }  ,draw opacity=0 ][fill={rgb, 255:red, 208; green, 2; blue, 27 }  ,fill opacity=0.57 ] (207.75,370) .. controls (207.75,368.76) and (208.76,367.75) .. (210,367.75) .. controls (211.24,367.75) and (212.25,368.76) .. (212.25,370) .. controls (212.25,371.24) and (211.24,372.25) .. (210,372.25) .. controls (208.76,372.25) and (207.75,371.24) .. (207.75,370) -- cycle ;
\draw  [color={rgb, 255:red, 0; green, 0; blue, 0 }  ,draw opacity=0 ][fill={rgb, 255:red, 208; green, 2; blue, 27 }  ,fill opacity=1 ] (237.75,200) .. controls (237.75,198.76) and (238.76,197.75) .. (240,197.75) .. controls (241.24,197.75) and (242.25,198.76) .. (242.25,200) .. controls (242.25,201.24) and (241.24,202.25) .. (240,202.25) .. controls (238.76,202.25) and (237.75,201.24) .. (237.75,200) -- cycle ;

\draw (231,122.4) node [anchor=north west][inner sep=0.75pt]    {$\textcolor[rgb]{0.82,0.01,0.11}{e}\textcolor[rgb]{0.82,0.01,0.11}{_{2}}$};
\draw (301,180.4) node [anchor=north west][inner sep=0.75pt]    {$\textcolor[rgb]{0.82,0.01,0.11}{e}\textcolor[rgb]{0.82,0.01,0.11}{_{1}}$};
\draw (141,250.4) node [anchor=north west][inner sep=0.75pt]    {$\textcolor[rgb]{0.82,0.01,0.11}{-e}\textcolor[rgb]{0.82,0.01,0.11}{_{1}}\textcolor[rgb]{0.82,0.01,0.11}{-e}\textcolor[rgb]{0.82,0.01,0.11}{_{2}}$};
\draw (271,132.4) node [anchor=north west][inner sep=0.75pt]    {$\textcolor[rgb]{0.25,0.46,0.02}{l}$};
\draw (151.57,362.97) node [anchor=north west][inner sep=0.75pt]    {$\underline{\Sigma _{\mathbb{CP}^{2}} /l}$};

\end{tikzpicture}
        \caption{Rationalization of the fan \(\Sigma_{\CP^2}/l\).}
    \end{figure}
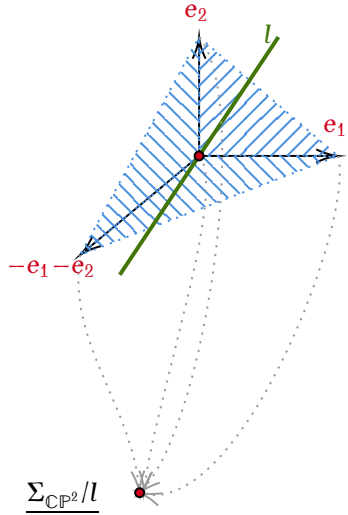
\end{example}
\begin{lemma}\label{lem:image_of_complete_is_complete}
    Let \(\Sigma=(S,K,V,\rho)\) be a complete generalized fan. Then its rationalization \(\ul{\Sigma}\) is also also a complete generalized fan.
\end{lemma}
\begin{proof}
    Let \(W\) be the vector space of the rationalization \(\ul{\Sigma}\). Let \(w\in W\) be an arbitrary vector of \(W\). By construction, there exists a vector \(v\in V\) mapped to \(w\) under the canonical rationalizing projection \(\ul{\pi}_\Sigma:V\to W\). Since \(\Sigma\) is complete \(v\) belongs to \(C(\sigma)\) for some face \(\sigma\in K\). Now, by linearity of \(\ul{\pi}_\Sigma\) the vector \(w\) belongs to \(\ul{\pi}_\Sigma(C(\sigma))\). Hence, the generalized fan \(\ul{\Sigma}\) is complete.
\end{proof}
The following result describes the relationship between the dual lattices of a fan and its rationalization.
\begin{proposition}\label{prop:cont_hom}
    There is the following canonical isomorphism induced by the rationalizing projection. 
    \begin{align*}
        \ul{\pi}_\Sigma^*:\Gamma(\ul{\Sigma})^\vee\to \Gamma(\Sigma)^\vee.
    \end{align*}
\end{proposition}
\begin{proof}
    We observe that since \(\Z\) has the discrete topology, any continuous homomorphism \(\Gamma(\Sigma)\to \Z\) factors in a canonical manner through the group \(\tc(\Gamma(\Sigma)\) of topologically connected components of the group \(\Gamma(\Sigma)\).

    Now, we prove that \(\tc(\Gamma(\Sigma))\) is canonically isomorphic to \(\Gamma(\ul{\Sigma})\) via the rationalizing projection map. To see this, we first observe that there is a canonical map \(\tc(\Gamma(\Sigma))\to \Gamma(\ul{\Sigma})\) since the latter group is discrete. Now, by homogeneity, it is enough to prove that the preimage of \(0\in \Gamma(\ul{\Sigma})\) consists of a single connected component. We note, that the connected component \([0]\) of \(0\) forms a group. If this group is discrete, it consists of a single element, and consequently, the group \(\Gamma(\Sigma)\) was discrete, hence \(\Sigma=\ul{\Sigma}\) and we are done. Assume that this is not the case. Then, consider the closure \(\ol{[0]}\) of \([0]\) in \(V\). 
    
    By a folklore result, \([0]\) is not dense in \(\ker \ul{\pi}_\Sigma\) if and only if there exists a non-zero linear functional on \(\ker \ul{\pi}_\Sigma\) which take integral values on \([0]\), see for instance \cite[Proposition 7]{Tausk_subgroups}. However, any such function pulls back to a linear function on \(\R^S\), which takes rational values on the standard basis vectors \(\{e_s\}_{s\in S}\). By the construction of the rationalization, it must take the same value on the space \(\ker \ul{\pi}_\Sigma\).
\end{proof}
\section{Cox rings of generalized fans and their graded automorphisms}\label{sec:cox_ring}
In this section, we introduce the Cox ring \(R(\Sigma)\) of a generalized fan \(\Sigma\). We describe the group \(\Aut_g(\Sigma)\) of the graded automorphisms of the ring \(R(\Sigma)\) in the case when the fan \(\Sigma\) is complete. The exposition in this section follows \cite{Cox_Aut}.
\subsection{Cox rings of generalized fans}
\begin{definition}
    \(\Sigma=(S,K,V,\rho)\) be a generalized fan. Consider the abelian group \(L_\Sigma\) arising from the following exact sequence.
    \[
        0\to \Gamma(\Sigma)^\vee\xrightarrow{(A_\Sigma^\Z)^\vee} \Z^S\to L_\Sigma\to 0.
    \]
\end{definition}
\begin{remark}
    Observe, that \(L_\Sigma\) is canonically isomorphic to the group \(\Hom(G_\Sigma,\CT)\) and consequently to the group \(\Hom(G_{\ul{\Sigma}},\CT)\).
\end{remark}
The group \(L_\Sigma\) gives a natural grading on the ring of polynomials in \(m\) variables. We describe the resulting graded ring below.
\begin{construction}\label{con:cox_ring}
    Let \(\Sigma=(S,K,V,\rho)\) be a generalized fan. We define \(L_\Sigma\)-grading \(|-|\) on each variable \(x_s\) in \(\C[S]\) as follows.
    \[
        |x_s|=e_s+(A_\Sigma^\Z)^\vee(\Gamma(\Sigma)),\quad \forall s\in S.
    \]
    The resulting graded ring is denoted by \(R(\Sigma)\) and called \emph{Cox ring} of the generalized fan \(\Sigma\).
\end{construction}
\begin{proposition}\label{prop:cox_isom}
    Let \(\Sigma=(S,K,V,\rho)\) be a generalized fan. Recall that by Construction \ref{con:rationalization}, one associates to it a complete rational generalized fan \(\ul{\Sigma}\) called rationalization of \(\Sigma\).
    There is a canonical isomorphism \(R(\Sigma)\cong R(\ul{\Sigma})\).
\end{proposition}
\begin{proof}
    By Proposition \ref{prop:cont_hom} there is a natural isomorphism of of groups \(\Gamma(\ul{\Sigma})^\vee\cong \Gamma(\Sigma)^\vee\). Hence, there is an isomorphism of the quotient groups \(L_{\uSigma}=\Z^S/(A_{\ul{\Sigma}}^\Z)^\vee\Gamma(\ul{\Sigma})^\vee\) and \(L_{\Sigma}=\Z^S/(A_\Sigma^\Z)^\vee\Gamma(\Sigma)^\vee\). This isomorphism induces the isomorphism of the graded rings \(R(\Sigma)\) and \(R(\uSigma)\) given by the identity of the underlying commutative rings.
\end{proof}
\begin{remark}
    Proposition \ref{prop:cox_isom} can be viewed in light of the fact that Construction \ref{con:cox_ring} provides a contravariant functor from the category of generalized fans into the category of commutative rings with abelian grading.
\end{remark}
\begin{remark}
    Another observation is that by Proposition \ref{prop:cox_isom}, the Cox ring construction only reflects the rationalization of a given generalized fan, not the generalized fan itself.
\end{remark}
\begin{example}
    Consider the generalized fan \(\Sigma_{\CP^2}\) of Example \ref{ex:CP2_fan}. Then the Cox ring of \(\Sigma_{\CP^2}\) is the standard Cox ring of the \(2\)-dimensional projective space (see e.g.~\cite[\S 1]{Cox_Aut}).
    \begin{align*}
        R(\Sigma_{\CP^2})=\C[x_0,x_1,x_2],\quad |x_0|=|x_1|=|x_2|=e^\vee_0+\left\la\begin{pmatrix}
            1 & 0 & -1
        \end{pmatrix},\begin{pmatrix}
            1 & -1 &0
        \end{pmatrix}\right\ra\in \Z^{\{0,1,2\}}/\left\la\begin{pmatrix}
            1 & 0 & -1
        \end{pmatrix},\begin{pmatrix}
            1 & -1 &0
        \end{pmatrix}\right\ra.
    \end{align*}
    Now consider the projection fan \(\Sigma_{\CP^2}/l\) of the fan \(\Sigma_{\CP^2}\) along the irrational line \(l\) of Example \ref{ex:CP2_fan_proj}. Then the Cox ring of \(\Sigma_{\CP^2}/l\) coincides with the Cox ring of the rationalization \(\ul{\Sigma_{\CP^2}/l}\) of the fan \(\Sigma_{\CP^2}/l\). By Example \ref{ex:CP2_fan_proj_rat} the vector space of the fan \(\ul{\Sigma_{\CP^2}/l}\) is zero-dimensional. Consequently, the grading group \(L_{\Sigma_{\CP^2}/l}\) coincides with the group \(\Z^{\{0,1,2\}}\). As a result, we see that the Cox ring of \(\Sigma_{\CP^2}/l\) has the following form.
    \begin{align*}
        R(\Sigma_{\CP^2}/l)=\C[x_0,x_1,x_2],\quad |x_0|=e^\vee_0\in \Z^{\{0,1,2\}},\; |x_1|=e^\vee_1\in \Z^{\{0,1,2\}},\; |x_2|=e^\vee_2\in \Z^{\{0,1,2\}}.
    \end{align*}
\end{example}
\begin{construction}
    Let \(\Sigma=(S,K,V,\rho)\) be a generalized fan. 
    \begin{itemize}
        \item The set \(S\) can be split as \(S=S_1\sqcup\ldots\sqcup S_s\) where \(S_i\) consists of elements \(s\) with the same \(L_\Sigma\)-grading of \(x_s\) in \(R(\Sigma)\).
        \item We denote the grading arising from \(S_i\)'s by \(\alpha_1,\ldots,\alpha_s.\) Given \(\alpha_i\) we consider the abelian group \(R(\Sigma)_{\alpha_i}\) of homogeneous elements in \(R(\Sigma)\) of degree \(\alpha_i\). 
        \item The group \(R(\Sigma)_{\alpha_i}\) further splits into the direct sum of the following form.
        \begin{align}
            R(\Sigma)_{\alpha_i}=R(\Sigma)_{\alpha_i}^{g}\oplus R(\Sigma)_{\alpha_i}^d.
        \end{align}
        The subgroup \(R(\Sigma)_{\alpha_i}^{g}\) is spanned by \emph{generators}, i.e. variables \(x_s\) with \(s\in S_i\), and \(R(\Sigma)_{\alpha_i}^d\) is spanned by the remaining monomials, which are by definition \emph{decomposable}.
    \end{itemize}
\end{construction}
The following result is a simple, purely combinatorial analog of Proposition 1.1 of Cox \cite{Cox_Aut}.
\begin{proposition}
    Let \(\Sigma=(S,K,V,\rho)\) be a generalized fan. Let \(\ul{\Sigma}=(S,K,W,\ul{\rho})\) be the rationalization of \(\Sigma\). Let \(\alpha=(\alpha_s)_{s\in S}\in \Z^S\) be an integer vector. Denote by \(P_\alpha\) the subset of \(W^\vee\) of the following form.
    \[
        P_\alpha=\left\{\tau\in W^\vee \mid \la \tau, \rho(s)\ra\ge -\alpha_s\right\}.
    \]
     Then there is a bijection between the intersection \((A_\Sigma^\Z)^\vee (\Gamma(\Sigma)^\vee)\cap P_\alpha\) and the standard monomial basis of \(R(\Sigma)_\alpha\).
\end{proposition}
\begin{proof}
    By the explicit description of the map \((A_\Sigma^\Z)^\vee\) given in Lemma \ref{lem:integral_tori_dual_map}
    we see that an element \(\gamma\in \Gamma(\Sigma)^\vee\) is mapped into the following vector in \(\Z^S.\)
    \[
    \Gamma(\Sigma)^\vee\ni \gamma\xmapsto{(A_\Sigma^\Z)^\vee} \sum_{s\in S}\la \gamma, A_{\ul{\Sigma}}(e_s)\ra e_s^*\in \Z^S.
    \]
    Let \(x^\iota\) be a monomial that lies in the graded component \(R(\Sigma)_\alpha.\) Then for any other monomial \(x^\nu\) in the same component we have an element \(\gamma\) in \(\GS^\vee\) such that \(\iota+(A_\Sigma^\Z)^\vee\gamma=\nu.\) In particular, we have the following identity.
    \[
        \sum_{s\in S} \nu_s e_s = \sum_{s\in S}(\la \gamma, A_{\ul{\Sigma}}(e_s)\ra+\iota_s) e_s^*.
    \]
    Since, \(\iota\) has non-negative coordinates we see that \(\la \gamma, A_{\ul{\Sigma}}(e_s)\ra+\nu_s\ge 0\) for all \(s\). Consequently, \(\gamma\) belongs to the following subset of \(\Gamma(\Sigma)^\vee\).
    \begin{align}\label{eq:P_iota_subset}
        P_\iota=\{\tau\in \Gamma(\Sigma)^\vee\mid \la \tau, A_{\ul{\Sigma}}(e_s)\ra \ge -\iota_s\}.
    \end{align}
    The result is now immediate.
\end{proof}
\begin{corollary}\label{cor:fin_dim_of_cox}
    If \(\Sigma=(S,K,V,\rho)\) is a complete generalized fan, then the ring \(R(\Sigma)\) has finite-dimensional graded components.
\end{corollary}
\begin{proof}
    
    Note that if \(\Sigma\) is complete by Lemma \ref{lem:image_of_complete_is_complete} the rationalization
    \(\ul{\Sigma}=(S,K,W,\ul{\rho})\) is also complete. Hence, the subset~\(P_\iota\) of \mbox{Equation \eqref{eq:P_iota_subset}} is compact by Lemma \ref{lem:compact_intersection}. Thus \(R(\Sigma)_\alpha\) is finite-dimensional over \(\C\) by discreteness of \(\Gamma(\ul{\Sigma})^\vee\) in \(W^\vee\).
\end{proof}
\begin{definition}
    We define the group \(\Aut_g(\Sigma)\) of graded automorphisms of the ring \(R(\Sigma)\) to be the group of automorphisms of the ring \(R(\Sigma)\) that preserve the \(L_\Sigma\)-grading.
\end{definition}
\subsection{Basic description}

In this section, we show how \cite[Proposition 4.3]{Cox_Aut} describing the structure of \(\Aut_g(\Sigma)\) for a complete embedded fan can be adapted to our context.
\begin{proposition}\label{prop:cox_struct}
    The group \(\Aut_g(\Sigma)\) of graded automorphism of the ring \(R(\Sigma)\) has the following properties:
    \begin{enumerate}[(i)]
        \item\label{prop:cox_struct_1} \(\Aut_g(\Sigma)\) is a connected affine algebraic group of dimension \[\left(\sum_{i=1}^s |S_i|\dim_\C R(\Sigma)_{\alpha_i}\right)+|S\setminus S^\red|\]
        \item The unipotent radical \(\mf{R}(\Sigma)\) of the group \(\Aut_g(\Sigma)\) is isomorphic as an affine algebraic variety to an affine space of dimension
        \[
            \dim_\C\mf{R}(\Sigma)=\sum_{i=1}^s |S_i|(\dim_\C R(\Sigma)_{\alpha_i}-|S_i|)
        \]
        \item The group \(\Aut_g(\Sigma)\) contains a reductive subgroup \(\mf{G_r}(\Sigma)\) isomorphic to the product:
        \[
            \mf{G_r}=\left(\prod_{i=1}^s GL(R(\Sigma)_{\alpha_i}^g)\right)\times (\CT)^{S\setminus S^\red}
        \]
        \item The group \(\Aut_g(\Sigma)\) is isomorphic to a semidirect product:
        \[
            \mf{G_r}(\Sigma)\ltimes \mf{R}(\Sigma)
        \]
    \end{enumerate}
\end{proposition}
Our proof is essentially the same as the proof given in \cite[Proof of Proposition 4.3]{Cox_Aut}. However, certain formal care must be taken when dealing with generalized fans, since as we saw in Example \ref{ex:CP2_fan_proj}, not all rational complete generalized fans arise from embedded fans. The proof proceeds verbatim as in \cite{Cox_Aut}, and we don't think that reproducing it here in its entirety would be particularly useful. Hence, we only include the proof of the fact that \(\Aut_g(R(\Sigma))\) is a finite-dimensional algebraic group to illustrate the general approach.
\begin{proof}
    First, we observe that the component corresponding to the coordinates \(z_{s'}\) for \(s'\in S\setminus S^\red\) clearly splits off from the rest of the group. Hence, we can assume that \(S=S^\red\). By \(\C\)-linearity and preservation of grading, there is the following description of the monoid of grading-preserving \emph{endomorphisms} of the ring \(R(\Sigma)\).
    \begin{align}\label{eq:E1}
        \End_g(R(\Sigma))\cong \left(\prod_{i=1}^s\Hom_\C(R(\Sigma)_{\alpha_i}^g,R(\Sigma)_{\alpha_i})\right)
    \end{align}
    There is also the following embedding of algebraic monoids.
    \begin{align}\label{eq:E2}
        \iota:\End_g(R(\Sigma))\hookrightarrow \prod_{i=1}^s\End_\C(R(\Sigma)_{\alpha_i},R(\Sigma)_{\alpha_i}).
    \end{align}
    We first prove that the image of \(\End_g(R(\Sigma))\) is an algebraic subvariety of \(\prod_{i=1}^s\End_\C(R(\Sigma)_{\alpha_i},R(\Sigma)_{\alpha_i}).\) In the standard monomial basis of \(R(\Sigma)\) one presents elements of \(\End_g(R(\Sigma))\) by a collection of \(s\) matrices of the following form.
    \begin{align}\label{eq:E3}
        \begin{pmatrix}
            A_i & 0\\
            B_i & C_i
        \end{pmatrix},\quad i=1,\ldots s.
    \end{align}
    Note that the submatrices \(\begin{pmatrix}
        A_i\\
        B_i
    \end{pmatrix}\) correspond to the elements of the right-hand side in \eqref{eq:E1}. Hence, the matrices \(C_i\) are completely determined by \(A_i\) and \(B_i\). To prove the statement we will need to explicitly describe the composition law in \(\End_g(R(\Sigma))\) in terms of the matrix presentation \eqref{eq:E3}. Note that \(\iota\) of \eqref{eq:E2} is a homomorphism of monoids. Hence given two maps \(\vp,\psi\) with matrix presentations \(\begin{pmatrix}
        A_i & 0\\
        B_i & C_i
    \end{pmatrix}\) and \(\begin{pmatrix}
        A_i' & 0\\
        B_i' & C_i'
    \end{pmatrix}\) their composition \(\vp\circ\psi\) has the matrix presentation as follows.
    \[
        \begin{pmatrix}
            A_iA_i' & 0\\
            B_iA_i'+C_i B_i' & C_i C_i'
        \end{pmatrix}
    \]
    We now describe the equations that cut out \(\End_g(R(\Sigma))\) inside of the product \(\prod_{i=1} \End_\C(R(\Sigma)_{\alpha_i})\) in \eqref{eq:E2}. 
    We have the following two groups of equations.
    \begin{itemize}
        \item The upper-right corner of the matrices in \eqref{eq:E3} must be zero.
        \item Suppose that \(x^\iota,\; x^\nu\in R(\Sigma)_{\alpha_i}^d\) are two decomposable monomials in \(R(\Sigma)_{\alpha_i}\). If \(\vp\in \End_g(R(\Sigma))\) is a graded endomorphism of \(R(\Sigma)\) we have the following.
        \[
            \vp(x^\iota)=\ldots+c_{\iota\nu}^i x^\nu+\ldots.
        \]
        Here \(c_{\iota\nu}^i\) is the entry in the block \(C_i\) of the matrix representation \eqref{eq:E3}. Since \(x^\iota\) is decomposable it is a product of generating variables \(\prod_{j=1}^l x_{s_j}\) not in \(R(\Sigma)^g_{\alpha_i}\). Hence, we have the following equation on the coefficients \(c^i_{\iota\nu}\): the coefficient of \(x^\nu\) in \(\vp(x^\iota)\) must be equal to the coefficient of \(x^\nu\) in \(\prod_{j=1}^l\vp(x_{s_i})\). The images of \(x_{s_j}\)'s are expressed in terms of entries of the matrices \(A_{\alpha_{k_j}}, B_{\alpha_{k_j}}\) where \(\alpha_{k_j}=|x_{s_j}|\). Thus, we have polynomial equations on the entries of the matrix \(C_i\) in terms of the entries in the matrices \(A_j,B_j\).
    \end{itemize}
    Now it remains to use the fact that by Corollary \ref{cor:fin_dim_of_cox} each \(R(\Sigma)_\alpha\) is finite-dimensional to see that \(\End_g(R(\Sigma))\) is a linear algebraic monoid, and thus its group of invertible elements is an algebraic group, see \cite{putcha} for proof of this fact.
\end{proof}
\begin{corollary}\label{cor:aut_g_red_of_gv}
    There is the following isomorphism of algebraic group schemes.
    \[
        \Aut_g(\Sigma)\cong \Aut_g(\Sigma^\red)\times (\CT)^{S\setminus S^\red}.
    \]
\end{corollary}
\begin{proof}
    Given a graded automorphism \(\vp\) of \(R(\Sigma)\) we immediately see that \(\vp(x_{s'})=\lambda_{s'} x_{s'}\) for any \(s'\in S\setminus S^{\red}\) and some \(\lambda_s\in \CT\). Additionally, by definition \(\vp(x_{s})\) does not contain any variables \(x_{s'}\) with \(s'\in S\setminus S^\red\). 
    Consequently, we have a homomorphism \(\Aut_g(\Sigma)\to \Aut_g(\Sigma^\red)\times (\CT)^{S\setminus S^\red}\) which is easily seen to be an isomorphism. Thus, the statement follows.
\end{proof}
\subsection{The roots}
We now turn to describing the roots of the group \(\Aut_g(\Sigma)\) in more detail.
\begin{definition}[{Demazure roots of a generalized fan}]\label{def:Demazure_roots}
    The set of {\it Demazure roots} of a generalized fan \(\Sigma\) is defined via the following formula.
    \begin{gather*}
    \mc{DR}(\Sigma)=\\=\left\{m\in \Gamma(\Sigma)^\vee:\;\exists s\in S\text{ such that } \la m, \rho(s)\ra = 1 \text{ and } \la m,\rho(s')\ra\le 0,\text{ for }s'\in S\setminus\{s\}\right\}.
    \end{gather*}
    The set of Demazure roots \(\DR(\Sigma)\) of a generalized fan \(\Sigma\) can be split into the two following disjoint subsets.
    \begin{itemize}
        \item The set of \emph{semisimple Demazure roots} of the generalized fan \(\Sigma\), denoted by \(\DR_s(\Sigma\), is defined as follows.
        \[
            \DR_s(\Sigma)=\DR(\Sigma)\cap -\DR(\Sigma).
        \]
        \item The set of \emph{unipotent Demazure roots} of the generalized fan \(\Sigma\) denoted by \(\DR_u(\Sigma)\) is defined as follows.
        \[
            \DR_u(\Sigma)=\DR(\Sigma)\setminus \DR_s(\Sigma).
        \]
    \end{itemize}
    If for a given element \(s\in S\) there exists a Demazure root \(m\in\DR(\Sigma)\) with \(\la m,\rho(s)\ra=1\) we denote \(m\) by \(m_s.\)
\end{definition}

\begin{construction}\label{con:aut_assoc_to_a_root}
    Let \(m_s\in\DR(\Sigma)\) be a Demazure root corresponding to the generalized fan \(\Sigma.\) Then we define a one-parametric subgroup of automorphisms of the affine space \(\C^S\) corresponding to \(m\) as follows.
    \[
    y_{m_s}(\lambda):\C^S\ni(z_{s'})_{s'\in S}\mapsto \left(z_s+\lambda\prod_{s'\neq s} z_{s'}^{\la -m_s,\rho(s')\ra},(z_{s'})_{s'\neq s}\right)\in \C^S,\quad \lambda \in \C.
    \]
\end{construction}
The following observation is implicit in the work of Cox \cite{Cox_Aut}. We include it here for the sake of completeness.
\begin{lemma}
    The subgroups \(y_m\) for \(m\in \DR(\Sigma)\) centralize the action of \(G_\Sigma.\)
\end{lemma}
\begin{proof}
    Consider an element of \(g\in G_\Sigma\). By assumption there exists \((\wt{g}_s)_{s\in S}=\wt{g}\in \ker A_\Sigma^\C\sse \C^S\) such that \(\exp_\C(\wt{g})=g\). Then the transformation \(y_m(\lambda) g y_m(-\lambda)\) acts on \(\C^S\) as follows.
    \begin{equation}\label{eq:comp_norm}
    \adjustbox{scale=0.85,center}{
            \begin{tikzcd}
                {(z_{s'})_{s'\in S}} && {\left(z_s-\lambda\prod_{s'\neq s}z_{s'}^{\la -m_s,\rho(s')\ra},(z_{s'})_{s'\neq s}\right)} && {\left(e^{\wt{g}_s}z_s-\lambda e^{\wt{g}_s}\prod_{s'\neq s}z_{s'}^{\la -m_s,\rho(s')\ra},(e^{\wt{g}_{s'}}z_{s'})_{s'\neq s}\right)} \\
                && {\left(e^{\wt{g}_s}z_s-\lambda e^{\wt{g}_s}\prod_{s'\neq s}z_{s'}^{\la -m_s,\rho(s')\ra}+\lambda\prod_{s'\neq s}e^{\wt{g}_{s'}\la -m,\rho(s')\ra}z_{s'}^{\la -m_s,\rho(s')\ra},(e^{\wt{g}_{s'}}z_{s'})_{s'\neq s}\right)}
                \arrow["{y_m(-\lambda)}", maps to, from=1-1, to=1-3]
                \arrow["g", maps to, from=1-3, to=1-5]
                \arrow["{y_m(\lambda)}", maps to, from=1-5, to=2-3]
            \end{tikzcd}}
    \end{equation}
We note that since \(\wt{g}\) belongs to \(\ker A_\Sigma^\C\) we have the following identity on its coordinates.
\[
    \sum_{s\in S} \wt{g}_s\la -m, \rho(s)\ra=0.
\]
In addition, recall that by assumption \(\la m,\rho(s)\ra=1.\) Consequently, we can rewrite the result of the composite map \eqref{eq:comp_norm} as follows.
\begin{gather*}
    \left(e^{\wt{g}_s}z_s-\lambda e^{\wt{g}_s}\prod_{s'\neq s}z_{s'}^{\la -m_s,\rho(s')\ra}+\lambda\prod_{s'\neq s}e^{\wt{g}_{s'}\la -m,\rho(s')\ra}z_{s'}^{\la -m_s,\rho(s')\ra},(e^{\wt{g}_{s'}}z_{s'})_{s'\neq s}\right)=\\=\left(e^{\wt{g}_s}\left(z_s-\lambda \prod_{s'\neq s} z_{s'}^{\la -m_s,\rho(s')\ra}+\lambda \underbrace{\prod_{s'\in S}e^{\wt{g}_{s'}\la -m,\rho(s')\ra}}_{=1}\prod_{s'\neq s}z_{s'}^{\la -m_s,\rho(s')\ra}\right),(e^{\wt{g}_{s'}}z_{s'})_{s'\neq s}\right)=\\=\left(e^{\wt{g}_s} z_s,(e^{\wt{g}_{s'}}z_{s'})_{s'\neq s}\right)
\end{gather*}
The result is now evident.
\end{proof}
\begin{proposition}\label{prop:Aut_g_generation}
    The algebraic group \(\Aut_g(\Sigma)\) is generated by the torus \(\left(\CT\right)^S\) and one-parametric subgroups \(y_{m}\) for \(m\in \DR(\Sigma)\).
\end{proposition}
\begin{proof}
    The proof proceeds verbatim as in the proof of \cite[Proposition 4.5]{Cox_Aut}.
\end{proof}

\begin{lemma}\label{lem:action_on_root_s_of_the_torus}
    Let \(y_m\) be a one-parametric subgroup of \(\Aut_g(\Sigma)\) corresponding to a Demazure root \(m\in \DR(\Sigma).\) Then the torus \((\CT)^S\) acts by conjugation on \(y_m(\lambda)\) as follows.
    \[
        (t_s)_{s\in S}\circ y_m(\lambda)\circ (t_s)_{s\in S}^{-1}=y_m(\lambda t_s\prod_{s'\neq s} t_{s'}^{\la m,\rho(s')\ra}).
    \]
\end{lemma}
\begin{proof}
    A straightforward computation.
\end{proof}
\begin{corollary}\label{cor:non_triv_act_on_roots_sub}
    There are no Demazure roots such that \(y_m(\lambda)\) is centralized by \((\CT)^S\).
\end{corollary}
\begin{proof}
    This is immediately apparent from Lemma \ref{lem:action_on_root_s_of_the_torus}, since we can put \(t_{s'}=1\) for \(s'\neq s\) and \(t_s=2\).
\end{proof}
\subsection{A Lie-theoretic picture}\label{subsec:lie_theor_pic}
\begin{definition}
    Let \(\mf{g}\) be a Lie algebra over \(\C\). A maximal abelian subalgebra of \(\mf{g}\) is called a \emph{Cartan subalgebra} of~\(\mf{g}\). A linear functional \(\alpha\in \mf{h}^\vee\) is called a \emph{root} of \(\h\) if there exists a one-dimensional subalgebra \(\mf{g}_\alpha\), called a \emph{root subspace}, of \(\mf{g}\) such that the following holds.
    \[
        \ad_h(x)=\alpha(h)x, \quad \forall h\in \h,\; x\in \mf{g}_\alpha.
    \]
    We denote the (multi-)set of roots of \(\h\) in \(\g\) by \(\Phi(\h,\g)\).
\end{definition}
The following result is standard. 
\begin{proposition}[{\cite[8.17]{Borel}}]
    Let \(\mf{g}\) be a Lie algebra over \(\C\). Let \(\h\) be a Cartan subalgebra of \(\mf{g}\). Then the set of roots \(\Phi(\h,\g)\) is a finite multi subset of \(\h^\vee\). Moreover, the Lie algebra \(\mf{g}\) has a root decomposition with respect to the action of \(\h\). That is, there is a direct sum decomposition of \(\mf{g}\) as follows.
    \[
        \mf{g}=\mf{g}^\h\oplus \bigoplus_{\alpha\in \Phi(\h,\g)}\mf{g}_\alpha.
    \]
    Here \(\g^\h\) is the subalgebra of elements in \(\g\) that commute with \(\h\). Alternatively, one can think of \(\g^\h\) as a sum of root subspaces corresponding to the root \(\alpha=0\).
\end{proposition}
\begin{lemma}
    Let \(\mf{g}\) be the Lie algebra of the algebraic group \(\Aut_g(\Sigma)\). Then the Lie algebra of the subgroup \((\CT)^S\) is a Cartan subalgebra of \(\mf{g}.\)
\end{lemma}
\begin{proof}
    The result is immediate from Proposition \ref{prop:cox_struct}.
\end{proof}
\begin{lemma}
    Let \(\Sigma=\GF\) be a generalized fan. Then the roots of the Cartan subalgebra \(\Lie((\CT)^S)\) corresponding to the one-parametric subgroup \(y_m(\lambda)\) are linearly independent for different Demazure roots \(m\in \DR(\Sigma).\)
\end{lemma}
\begin{proof}
    By Lemma \ref{lem:action_on_root_s_of_the_torus}, these roots are given by the following formula.
    \[
        \alpha_m=\ve_s+\sum_{s'\neq s} \la m,\rho(s')\ra \ve_{s'}.
    \]
    Note that \(\la m,\rho(s')\ra\le 0\) for \(s'\neq s\). Consequently, if \(m\neq m'\) then the roots \(\alpha_m\) and \(\alpha_{m'}\) are linearly independent.
\end{proof}
The following result is similar to the results of \cite[\S 9.4]{Borel}, and has a fairly standard proof. Which we record here for the sake of completeness.
\begin{proposition}\label{prop:Alisa_Lie}
    Let \(\mf{g}\) be a Lie algebra over \(\C\). Let \(\h\) be a fixed Cartan subalgebra of \(\mf{g}\).
    Consider a root decomposition of \(\mf{g}\) with respect to \(\h\)-action.
    \begin{gather}\label{eq:Lie_dec}
        \mf{g}=\mf{g}^\h\oplus \bigoplus_{\alpha\in \Phi(\h,\g)}\mf{g}_\alpha.
    \end{gather}
    Here \(\mf{g}_\alpha\) are one-dimensional roots subspaces of \(\mf{g}\) corresponding to the roots of \(\h\). The subalgebra \(\mf{g}^\h\) is the subalgebra of elements in \(\mf{g}\) that commute with \(\h\). Assume that \(\mf{k}\) is a Lie subalgebra of \(\mf{g}\) that contains \(\mf{g}^\mf{h}\). Then the decomposition \eqref{eq:Lie_dec} of \(\mf{g}\) is inherited by \(\mf{k}\). That is, we have a splitting of the following form.
    \[
        \mf{k}=\mf{g}^\h\oplus \bigoplus_{\alpha\in \Phi(\h,\g)'}\mf{g}_\alpha.
    \]
    Here \(\Phi(\h,\g)'\) is the subset of the set of roots of \(\mf{g}\) such that the corresponding roots subspace are in \(\mf{k}\).
\end{proposition}
We reduce the proof of the statement to the following simpler claim.
\begin{lemma}\label{lem:Alisa_Vandermonde}
    In the setting of Proposition \ref{prop:Alisa_Lie} let \(\mf{s}=\la \h,\sum_{\alpha\in \Phi(\h,\g)}\lambda_\alpha e_\alpha\ra\) be a subalgebra of the Lie algebra \(\mf{k}\) generated by \(\h\) and a single element \(x=\sum_{i=1}^m\lambda_{\alpha_i} e_{\alpha_i}\) equal to a sum of elements of root subspaces corresponding to \emph{different roots} with non-zero coefficients. Then the subalgebra \(\mf{s}\) contains each of the summands \(e_{\alpha_i}\) present in the decomposition of~\(x\).
\end{lemma}
\begin{proof}
    Consider the action of \(\h\) on the element \(x\). Let \(h\in\h\), then by definition of the root subspaces, we have the following. 
    \[
        \ad_h(x)=\sum_{i=1}^m \lambda_i \alpha_i(h) e_{\alpha_i}.
    \]
    Iterating the action of \(\ad_h\) we obtain the following formula.
    \[
        \ad_h^k(x)=\sum_{i=1}^m \lambda_i \alpha_i(h)^k e_{\alpha_i}.
    \]
    Consequently, using the Vandermonde determinant, we see that the elements \(\ad^k_h(x)\) for \(0\le k\le m\) are linearly independent. Thus, the subalgebra \(\mf{s}\) contains each of the vectors present in the decomposition of \(x\).
\end{proof}
The proposition now follows from Lemma \ref{lem:Alisa_Vandermonde}.
\begin{lemma}
    The Lie algebra of the group \(\Aut_g(\Sigma)\) has a root decomposition with respect to the Cartan subalgebra \(\Lie\left((\CT)^S\right)\). The root subspaces are the Lie subalgebras of the one-parametric subgroups \(y_m(\lambda)\) for \(m\in \DR(\Sigma).\) That is we have the following splitting.
    \[
        \Lie\left(\Aut_g(\Sigma)\right)=\Lie\left((\CT)^S\right)\oplus \bigoplus_{m\in \DR(\Sigma)}\Lie(y_m(\lambda)).
    \]
\end{lemma}
\begin{proof}
    The result is immediate from Proposition \ref{prop:Aut_g_generation}.
\end{proof}

\section{Normalizers of Fan Tori}\label{sec:Norm_of_GH}
In this section, we a prove a GAGA-type result for holomorphic automorphisms of the Cox construction of a complete generalized fan normalizing the fan tori. That is, we show that all such automorphisms are algebraic. 
The following basic result from convex geometry will be extremely useful in the sequel.
\begin{lemma}[{\cite[Corollary 1.2.8]{bupa15}}]\label{lem:compact_intersection}
    The intersection of halfspaces \(\{x\in \R^n\mid \la x,\alpha_i\ra \ge b_i\}\) for \(1\le i\le m\) is compact if and only if there exists a vanishing linear combination of the vectors \(\alpha_i\) with positive coefficients.
\end{lemma}
\subsection{Normalizer of the complex fan torus}

\begin{proposition}\label{prop:Norm_of_G}
    Let \(\Sigma=(S,K,V,\rho)\) be a complete generalized fan.  Let \(\phi\) be a holomorphic automorphism of~\(\C^S\) normalizing the group \(G_\Sigma\). Then \(\phi\) is a polynomial automorphism.
\end{proposition}
\begin{proof}
    First, we assume that \(K\) has no ghost vertices. The condition for \(\phi\) to normalize the group \(G_\Sigma\) is equivalent to the following system of equations:
    \begin{gather*}
        \wh{g}^j \cd \phi^s=(g^*\phi)^s,\quad s\in S. 
    \end{gather*}
    Here \(g\) is a fixed element of \(G_\Sigma\) and the element \(\wh{g}\) is some element of \(G_\Sigma\). Consider the Taylor series expansion of \(\phi\) at the origin:
    \begin{gather}\label{eq:aut_series_decomposition}
        \phi^j(z)=\sum_{\iota\in\N^m} a_\iota^j z^\iota.
    \end{gather}
    The system of equations above can be reformulated in terms of the Taylor series of~\(\phi\) at the origin as follows. 
    \begin{gather}\label{formula:TS_exp_G}
        \wh{g}^s \sum_{\iota\in\N^S} a_\iota^s z^\iota=\sum_{\iota\in\N^S} a_\iota^s (g z)^\iota, \quad s\in S.
    \end{gather}
    Each element of the group \(G_\Sigma\) has the form \(\exp(w)\) for some \(w\in\C^S.\) Thus system of equation \eqref{formula:TS_exp_G} can be reformulated in the following form.
    \begin{gather}\label{formula:TS_exp_G_sep}
        \exp({\wh{w}^s}a^s_\iota) = a^j_\iota \exp({\la w, \iota\ra}),\quad \iota\in\N^S,\quad s\in S.
    \end{gather}
    Here the function \(\la w,\iota\ra\) is defined by the following formula.
    \begin{gather*}
        \la w,\iota\ra=\sum_{s\in S} \iota^s \cd w^s.
    \end{gather*}
    Now system \eqref{formula:TS_exp_G_sep} is equivalent to the following system.
    \begin{gather*}
        \begin{cases}
        e^{\wh{w}^s}=e^{\la w,\iota\ra}\\
        \text{or }\\
        a_\iota^s=0
        \end{cases},\quad \iota\in\N^S,\quad s\in S.
    \end{gather*}
    Further, we have the following equivalent system of equations.
    \begin{gather*}
    \begin{cases}
    \wh{w}^s=\la w,\iota\ra + 2\pi i \Z.\\
    \text{or }\\
    a_\iota^s=0
    \end{cases},\quad \iota\in\N^S,\quad s\in S.
    \end{gather*}
    By considering the imaginary and real parts separately, we obtain the following system of equations.
    \begin{gather*}
    \begin{cases}
    \Re(\wh{w}^s)=\la \Re(w),\iota\ra, \text{ and } \Im(\wh{w}^s)=\la \Im(w),\iota\ra + 2\pi \Z.\\
    \text{or }\\
    a_\iota^s=0
    \end{cases},\quad \iota\in\N^S,\quad s\in S.
    \end{gather*}
    Here \(\la\Re(w),\iota\ra\) and \(\la\Im(w),\iota\ra\) are the standard dot products in \(\R^S.\) Since the fan \(\Sigma\) is complete the Lemma \ref{lemma:positive_combination} implies existence of an element \(w\in \ker A_\Sigma^{\C}\) with strictly positive real and imaginary parts. Hence, the intersection of the level set \(\la w,\iota\ra=c\) for any \(c\in \R\) with the non-negative orthant is compact by Lemma \ref{lem:compact_intersection}.
    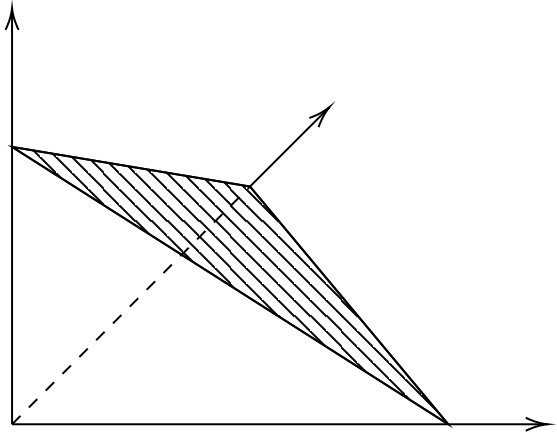
\begin{figure}[H]

 
\tikzset{
pattern size/.store in=\mcSize, 
pattern size = 5pt,
pattern thickness/.store in=\mcThickness, 
pattern thickness = 0.3pt,
pattern radius/.store in=\mcRadius, 
pattern radius = 1pt}
\makeatletter
\pgfutil@ifundefined{pgf@pattern@name@_jum3ag698}{
\pgfdeclarepatternformonly[\mcThickness,\mcSize]{_jum3ag698}
{\pgfqpoint{0pt}{-\mcThickness}}
{\pgfpoint{\mcSize}{\mcSize}}
{\pgfpoint{\mcSize}{\mcSize}}
{
\pgfsetcolor{\tikz@pattern@color}
\pgfsetlinewidth{\mcThickness}
\pgfpathmoveto{\pgfqpoint{0pt}{\mcSize}}
\pgfpathlineto{\pgfpoint{\mcSize+\mcThickness}{-\mcThickness}}
\pgfusepath{stroke}
}}
\makeatother
\tikzset{every picture/.style={line width=0.75pt}} 

\begin{tikzpicture}[x=0.75pt,y=0.75pt,yscale=-1,xscale=1]

\draw    (90,220) -- (358,220) ;
\draw [shift={(360,220)}, rotate = 180] [color={rgb, 255:red, 0; green, 0; blue, 0 }  ][line width=0.75]    (10.93,-3.29) .. controls (6.95,-1.4) and (3.31,-0.3) .. (0,0) .. controls (3.31,0.3) and (6.95,1.4) .. (10.93,3.29)   ;
\draw    (90,220) -- (90,12) ;
\draw [shift={(90,10)}, rotate = 90] [color={rgb, 255:red, 0; green, 0; blue, 0 }  ][line width=0.75]    (10.93,-3.29) .. controls (6.95,-1.4) and (3.31,-0.3) .. (0,0) .. controls (3.31,0.3) and (6.95,1.4) .. (10.93,3.29)   ;
\draw    (90,80) -- (210,100) ;
\draw    (90,80) -- (310,220) ;
\draw    (210,100) -- (310,220) ;
\draw  [dash pattern={on 4.5pt off 4.5pt}]  (90,220) -- (210,100) ;
\draw  [pattern=_jum3ag698,pattern size=6pt,pattern thickness=0.75pt,pattern radius=0pt, pattern color={rgb, 255:red, 0; green, 0; blue, 0}] (90,80) -- (210,100) -- (310,220) -- (90,80) -- cycle ;
\draw    (210,100) -- (248.59,61.41) ;
\draw [shift={(250,60)}, rotate = 135] [color={rgb, 255:red, 0; green, 0; blue, 0 }  ][line width=0.75]    (10.93,-3.29) .. controls (6.95,-1.4) and (3.31,-0.3) .. (0,0) .. controls (3.31,0.3) and (6.95,1.4) .. (10.93,3.29)   ;

\end{tikzpicture}
        \caption{The case of no ghost vertices.}
    \end{figure}
    \par From this observation, it follows that there is only a finite number of possible vectors \(\iota\) which satisfy the equation \(\Re(\wh{w}^j)=\la \Re(w),\iota\ra\). Hence, only a finite number of non-zero terms exist in the Taylor expansion of \(\phi\). Thus \(\phi\) is a polynomial automorphism.
    
    The proof of the general case proceeds by induction. Assume that we proved the result for the fan \(\Sigma=\GF\). We now want to prove for the fan \(\Sigma\cup\{(s',v')\}\). We note that since \(s\) is a ghost vertex, the possible powers in the coordinate \(s\) in the presentation \eqref{eq:aut_series_decomposition} can be negative. 

    By induction, we can assume that there exist only finitely many values \(\iota\in \Z^S\) such \(a_\iota^{s}\neq 0\). We now observe that since \(\Sigma\) is complete by Lemma \ref{lem:comb_ghost} both \(v\) and \(-v\) can be expressed as positive combinations of \(\rho(s)\) for \(s\in S\). We denote such combinations by \(w_v\) and \(w_{-v}\). By construction, and the inductive assumption the intersection \(\{\la w_{v},\iota\ra=c, \la w_{-v},\iota\ra=c'\}\) with the appropriate semi-orthant is compact. Hence, the automorphism in question is polynomial. See Figures \ref{fig:ghost_1}, \ref{fig:ghost_2} below for an illustration of the argument.
\end{proof}
\begin{figure}[H]

\tikzset{every picture/.style={line width=0.75pt}} 

\begin{tikzpicture}[x=0.75pt,y=0.75pt,yscale=-1,xscale=1]

\draw    (200,160) -- (478,160) ;
\draw [shift={(480,160)}, rotate = 180] [color={rgb, 255:red, 0; green, 0; blue, 0 }  ][line width=0.75]    (10.93,-3.29) .. controls (6.95,-1.4) and (3.31,-0.3) .. (0,0) .. controls (3.31,0.3) and (6.95,1.4) .. (10.93,3.29)   ;
\draw    (200,160) -- (200,22) ;
\draw [shift={(200,20)}, rotate = 90] [color={rgb, 255:red, 0; green, 0; blue, 0 }  ][line width=0.75]    (10.93,-3.29) .. controls (6.95,-1.4) and (3.31,-0.3) .. (0,0) .. controls (3.31,0.3) and (6.95,1.4) .. (10.93,3.29)   ;
\draw    (200,160) -- (200,288) ;
\draw [shift={(200,290)}, rotate = 270] [color={rgb, 255:red, 0; green, 0; blue, 0 }  ][line width=0.75]    (10.93,-3.29) .. controls (6.95,-1.4) and (3.31,-0.3) .. (0,0) .. controls (3.31,0.3) and (6.95,1.4) .. (10.93,3.29)   ;
\draw    (200,260) -- (280,180) -- (310,150) ;
\draw    (200,50) -- (310,150) ;
\draw  [dash pattern={on 0.84pt off 2.51pt}]  (310,150) -- (450,10) ;
\draw  [dash pattern={on 0.84pt off 2.51pt}]  (310,150) -- (480,300) ;
\draw  [dash pattern={on 0.84pt off 2.51pt}]  (200,260) -- (160,300) ;
\draw  [dash pattern={on 0.84pt off 2.51pt}]  (150,0) -- (200,50) ;

\draw (241,72.4) node [anchor=north west][inner sep=0.75pt]    {$\langle w_{v} ,\iota \rangle =c$};
\draw (251,202.4) node [anchor=north west][inner sep=0.75pt]    {$\langle w_{-v} ,\iota \rangle =c'$};

\end{tikzpicture}
    \caption{The case of a single ghost vertex.}
    \label{fig:ghost_1}
\end{figure}
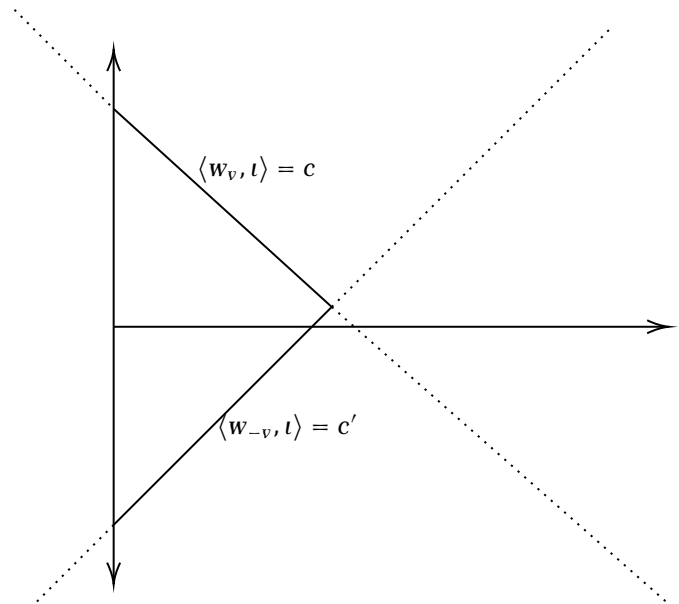
\begin{figure}[H]

 
\tikzset{
pattern size/.store in=\mcSize, 
pattern size = 5pt,
pattern thickness/.store in=\mcThickness, 
pattern thickness = 0.3pt,
pattern radius/.store in=\mcRadius, 
pattern radius = 1pt}
\makeatletter
\pgfutil@ifundefined{pgf@pattern@name@_mmn6l2qjp}{
\pgfdeclarepatternformonly[\mcThickness,\mcSize]{_mmn6l2qjp}
{\pgfqpoint{0pt}{0pt}}
{\pgfpoint{\mcSize+\mcThickness}{\mcSize+\mcThickness}}
{\pgfpoint{\mcSize}{\mcSize}}
{
\pgfsetcolor{\tikz@pattern@color}
\pgfsetlinewidth{\mcThickness}
\pgfpathmoveto{\pgfqpoint{0pt}{0pt}}
\pgfpathlineto{\pgfpoint{\mcSize+\mcThickness}{\mcSize+\mcThickness}}
\pgfusepath{stroke}
}}
\makeatother

 
\tikzset{
pattern size/.store in=\mcSize, 
pattern size = 5pt,
pattern thickness/.store in=\mcThickness, 
pattern thickness = 0.3pt,
pattern radius/.store in=\mcRadius, 
pattern radius = 1pt}
\makeatletter
\pgfutil@ifundefined{pgf@pattern@name@_cvm5rd9om}{
\pgfdeclarepatternformonly[\mcThickness,\mcSize]{_cvm5rd9om}
{\pgfqpoint{0pt}{-\mcThickness}}
{\pgfpoint{\mcSize}{\mcSize}}
{\pgfpoint{\mcSize}{\mcSize}}
{
\pgfsetcolor{\tikz@pattern@color}
\pgfsetlinewidth{\mcThickness}
\pgfpathmoveto{\pgfqpoint{0pt}{\mcSize}}
\pgfpathlineto{\pgfpoint{\mcSize+\mcThickness}{-\mcThickness}}
\pgfusepath{stroke}
}}
\makeatother

 
\tikzset{
pattern size/.store in=\mcSize, 
pattern size = 5pt,
pattern thickness/.store in=\mcThickness, 
pattern thickness = 0.3pt,
pattern radius/.store in=\mcRadius, 
pattern radius = 1pt}
\makeatletter
\pgfutil@ifundefined{pgf@pattern@name@_vk4b4ayr3 lines}{
\pgfdeclarepatternformonly[\mcThickness,\mcSize]{_vk4b4ayr3}
{\pgfqpoint{0pt}{0pt}}
{\pgfpoint{\mcSize+\mcThickness}{\mcSize+\mcThickness}}
{\pgfpoint{\mcSize}{\mcSize}}
{\pgfsetcolor{\tikz@pattern@color}
\pgfsetlinewidth{\mcThickness}
\pgfpathmoveto{\pgfpointorigin}
\pgfpathlineto{\pgfpoint{\mcSize}{0}}
\pgfusepath{stroke}}}
\makeatother

 
\tikzset{
pattern size/.store in=\mcSize, 
pattern size = 5pt,
pattern thickness/.store in=\mcThickness, 
pattern thickness = 0.3pt,
pattern radius/.store in=\mcRadius, 
pattern radius = 1pt}
\makeatletter
\pgfutil@ifundefined{pgf@pattern@name@_ebmq33xwo}{
\pgfdeclarepatternformonly[\mcThickness,\mcSize]{_ebmq33xwo}
{\pgfqpoint{0pt}{-\mcThickness}}
{\pgfpoint{\mcSize}{\mcSize}}
{\pgfpoint{\mcSize}{\mcSize}}
{
\pgfsetcolor{\tikz@pattern@color}
\pgfsetlinewidth{\mcThickness}
\pgfpathmoveto{\pgfqpoint{0pt}{\mcSize}}
\pgfpathlineto{\pgfpoint{\mcSize+\mcThickness}{-\mcThickness}}
\pgfusepath{stroke}
}}
\makeatother
\tikzset{every picture/.style={line width=0.75pt}} 

\begin{tikzpicture}[x=0.75pt,y=0.75pt,yscale=-1,xscale=1]

\draw    (200,160) -- (478,160) ;
\draw [shift={(480,160)}, rotate = 180] [color={rgb, 255:red, 0; green, 0; blue, 0 }  ][line width=0.75]    (10.93,-3.29) .. controls (6.95,-1.4) and (3.31,-0.3) .. (0,0) .. controls (3.31,0.3) and (6.95,1.4) .. (10.93,3.29)   ;
\draw    (200,160) -- (200,22) ;
\draw [shift={(200,20)}, rotate = 90] [color={rgb, 255:red, 0; green, 0; blue, 0 }  ][line width=0.75]    (10.93,-3.29) .. controls (6.95,-1.4) and (3.31,-0.3) .. (0,0) .. controls (3.31,0.3) and (6.95,1.4) .. (10.93,3.29)   ;
\draw    (200,160) -- (200,288) ;
\draw [shift={(200,290)}, rotate = 270] [color={rgb, 255:red, 0; green, 0; blue, 0 }  ][line width=0.75]    (10.93,-3.29) .. controls (6.95,-1.4) and (3.31,-0.3) .. (0,0) .. controls (3.31,0.3) and (6.95,1.4) .. (10.93,3.29)   ;
\draw    (200,160) -- (111.41,248.59) ;
\draw [shift={(110,250)}, rotate = 315] [color={rgb, 255:red, 0; green, 0; blue, 0 }  ][line width=0.75]    (10.93,-3.29) .. controls (6.95,-1.4) and (3.31,-0.3) .. (0,0) .. controls (3.31,0.3) and (6.95,1.4) .. (10.93,3.29)   ;
\draw    (200,160) -- (298.59,61.41) ;
\draw [shift={(300,60)}, rotate = 135] [color={rgb, 255:red, 0; green, 0; blue, 0 }  ][line width=0.75]    (10.93,-3.29) .. controls (6.95,-1.4) and (3.31,-0.3) .. (0,0) .. controls (3.31,0.3) and (6.95,1.4) .. (10.93,3.29)   ;
\draw  [pattern=_mmn6l2qjp,pattern size=6pt,pattern thickness=0.75pt,pattern radius=0pt, pattern color={rgb, 255:red, 0; green, 0; blue, 0}] (200,90) -- (260,100) -- (330,160) -- (200,90) -- cycle ;
\draw  [pattern=_cvm5rd9om,pattern size=6pt,pattern thickness=0.75pt,pattern radius=0pt, pattern color={rgb, 255:red, 0; green, 0; blue, 0}] (130,230) -- (200,90) -- (330,160) -- cycle ;
\draw  [pattern=_vk4b4ayr3,pattern size=6pt,pattern thickness=0.75pt,pattern radius=0pt, pattern color={rgb, 255:red, 0; green, 0; blue, 0}] (130,230) -- (200,250) -- (330,160) -- cycle ;
\draw  [pattern=_ebmq33xwo,pattern size=6pt,pattern thickness=0.75pt,pattern radius=0pt, pattern color={rgb, 255:red, 0; green, 0; blue, 0}] (200,250) -- (260,100) -- (330,160) -- cycle ;

\end{tikzpicture}
    \caption{Our ability to change the sign in the ghost-vertex coordinates gives us a compact region of possible non-zero coefficients.}
    \label{fig:ghost_2}
\end{figure}
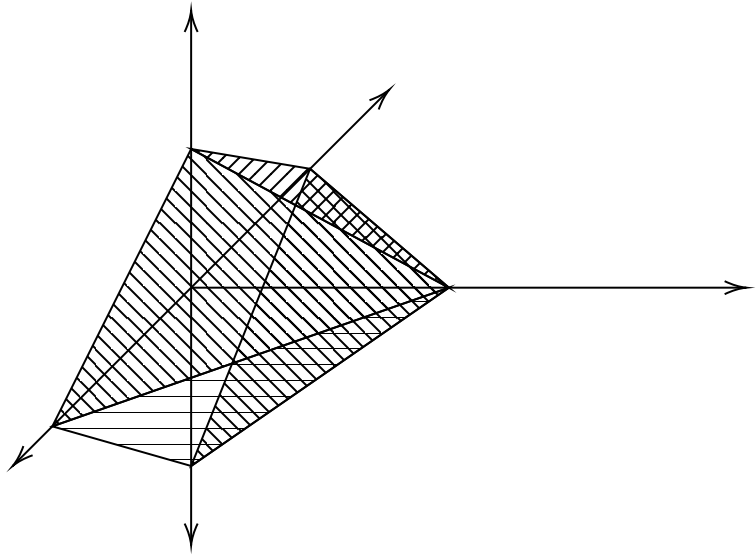
\subsection{Normalizer of the real fan torus}
\begin{proposition}
    Let \(\Sigma=(S,K,V,\rho)\) be a complete generalized fan.  Let \(\phi\) be a holomorphic automorphism of \(\C^S\) normalizing the group \(H_\Sigma\). Then \(\phi\) is a polynomial automorphism. 
\end{proposition}
\begin{proof}
    The proof proceeds much like the proof of Proposition \ref{prop:Norm_of_G} does. Consider the Taylor expansion of \(\phi\). Then the condition \(\phi H\phi^{-1}=H\) is equivalent to the following system of equations for all \(w\in\ker A_\Sigma\), and some \(\wt{w}\in A\).
    \[
        \begin{cases}
        e^{\wh{w}^s}=e^{\la w,\iota\ra}\\
        \text{or }\\
        a_\iota^s=0
        \end{cases},\quad \iota\in\N^S,\quad j\in S.
    \]
    Since the exponential map \(e\) is injective on \(\h_\Sigma\), we have the following system of equations.
    \[
    \begin{cases}
    \wh{w}^s=\la w,\iota\ra\\
    \text{or }\\
    a_\iota^s=0
    \end{cases},\quad \iota\in\N^S,\quad s\in S.
    \]
    Once again, from Lemma \ref{lemma:positive_combination}, we deduce the existence of \(w\in\Re(\h_\Sigma)\) with strictly positive coordinates. Hence, the intersection of the level set \(\Re\la w,\iota\ra=c\) for any \(c\in \R\) with the non-negative orthant is compact by Lemma \ref{lem:compact_intersection}. Thus, it follows that there is only a finite number of vectors \(\iota\in\N^S\) satisfying the equation
    \[
        \wh{w}^s=\la w,\iota\ra.
    \]
    This, in turn, implies that there is only a finite number of non-zero coefficients in the Taylor series expansion of \(\phi\). Hence, \(\phi\) is polynomial. The case of a fan with ghost vertices is handled similarly to the proof of Proposition \ref{prop:Norm_of_G}.
\end{proof}
\subsection{Zariski closures and fan rationalization}
Note that the group $H_\Sigma$ is never algebraic. The natural question to ask is what is Zariski closure of the group $H_\Sigma$ in the torus $(\C^\times)^S$. The following proposition provides an answer.
	\begin{proposition}\label{prop:zariski_closure}
		Denote by $\ol{H}_\Sigma$ Zariski closure of the group $H_\Sigma$ in the torus $(\C^\times)^S.$ Then the following holds.
		$$
			\ol{H}_\Sigma=G_{\ul{\Sigma}}.
		$$
	\end{proposition} 
    \begin{proof}
        Let $\ol{H}_\Sigma$ denote the closure of $H_\Sigma$ in the group $G_\Sigma.$ Since $H_\Sigma$ is connected it follows that $\ol{H}_\Sigma$ is connected. Thus $\ol{H}_\Sigma\ss G_{\ul{\Sigma}}.$ 
		Since $\ol{H}_\Sigma$ is a subtorus of an algebraic torus $G_{\ul{\Sigma}}$  it is an intersection of a finite collection of \(G_{\ul{\Sigma}}\)-characters. Consider such a presentation.
		$$
			\ol{H}_\Sigma=\bigcap_{i=1}^k \Ker\chi_i.
		$$  
		Here $\chi_i$ are characters of the group $G_{\ul{\Sigma}}.$
		By construction of the group \(G_{\ul{\Sigma}}\), we have the following exact sequence of finitely generated abelian groups.
		\begin{equation}\label{eq:exact_characters}
			0\to \Z^n\xrightarrow{A^*_\Z} \Z^S \xrightarrow{\pi} \mf{X}(G_{\ul{\Sigma}}) \to 0.
		\end{equation}
		Assume that $\chi$ is a character $\chi \in \mf{X}((\C^\times)^S)=\Z^S$ such that $\chi(H_\Sigma)=\{1\}.$ We have the following commutative diagram:
		\begin{equation*}
		\begin{tikzcd}
		\mf{c}\rar{\Psi}\arrow[swap]{rrrd}{0}&\C^S\rar[two heads]{\Re}&\ker A_\Sigma^{\R} \rar[hook]& \R^S \rar{A_\Sigma^\R}\dar[swap]{\chi\otimes\R} & \R^n \arrow[dashed]{dl}{\overline{\chi\otimes\R}}\\
		& & & \R&
		\end{tikzcd}
		\end{equation*}
		Since $A=A^\Z_\Sigma\otimes_\Z\R$ we have that $\chi\in \Im(A^*_\Z).$ It is clear that $A^*_\Z(\ol{\chi\otimes\R}\vert_{\Z^n})=\chi.$
		From the exact sequence \eqref{eq:exact_characters} we see that $\im A^*_\Z=\ker[\pi:\Z^S\to\mf{X}(G_{\ul{\Sigma}})].$ Hence $\chi(G_{\ul{\Sigma}})=\{1\}.$
		Since $\chi$ was an arbitrary character such that $\chi(\ol{H}_\Sigma)=\{1\}$ we see that $\ol{H}_\Sigma=G_{\ul{\Sigma}}.$
    \end{proof}
Proposition \ref{prop:zariski_closure} immediately implies the following useful corollary.
\begin{corollary}\label{cor:normalizers_coincide}
    The normalizer of the group \(H_\Sigma\) in the group of regular automorphisms of the variety \(U(\Sigma)\) is isomorphic to the normalizer of the group \(G_{\ul{\Sigma}}\).
\end{corollary}
\section{Equivariant automorphisms of the Cox construction}\label{sec:equi_aut_of_cox}

\subsection{Reduction of ghost vertices}
\begin{construction}\label{con:ghost_reduction}
    Let \(\Sigma=\GF\) be a complete rational generalized fan. Recall that by Construction \ref{con:reduced_fan} there is a \emph{reduced} generalized fan \(\Sigma^\red\) associated to \(\Sigma\). The fan \(\Sigma^\red\) has no ghost vertices. Below we construct a fibration of the following form.
    \[\begin{tikzcd}
        {U(\Sigma)} \\
        \\
        {U(\Sigma^\red)}
        \arrow["\TT_\Sigma", from=1-1, to=3-1]
    \end{tikzcd}\]
    Here, \(\TT_\Sigma\) denotes an algebraic torus of complex dimension equal to the cardinality of the set of ghost vertices.
    
    Let \(s\) be a ghost vertex of \(\Sigma\). Observe that by Lemma \ref{lem:comb_ghost}, there exists a positive combination with integer coefficients that expresses \(e_s\) as a sum of \(e_{s'}\) for non-ghost \(s'\). We pick one such combination for each ghost vertex \(s\) and denote them by \(\lambda_{s}^{s'}\).
    
    Consequently, we can take the subgroup of \(G_\Sigma\) generated by exponentials of such combinations for all ghost vertices. We denote this subgroup by \(\TT_\Sigma\). Consider the following regular map. 
    \[
        \pi:U(\Sigma)\to U(\Sigma^\red),\quad \left((z_s)_{s\in S^\red},(t_{s'})_{s'\in S\setminus S^\red}\right)\xrightarrow{\pi} \left(z_{s'}\cdot \prod_{s\in S\setminus S^\red}t_{s}^{-\lambda_{s}^{s'}}\right)_{s'\in S^\red}.
    \]
    The idea is that the product on the right is canceling out the action of \(\TT_\Sigma\) on the variable \(z_{s'}\). The map \(\pi\) is clearly \(\TT_\Sigma\)-invariant. Moreover, since the value of \(t_{s}\) can be made equal to \(1\) through \(\TT_\Sigma\)-action we see that \(\pi\) is injective on~\(\TT_\Sigma\)-orbits. For the same reason, the action of the group \(\TT_\Sigma\) is free. Finally, the map \(\pi\) is clearly surjective, and consequently, \(U(\Sigma^\red)\) is a geometric quotient of \(U(\Sigma)\) with respect to~\(\TT_\Sigma\)-action.
\end{construction}
\begin{definition}
    Let \(\Sigma=(S,K,V,\rho)\) be a complete generalized fan.
    Consider the associated complex torus \(G_\Sigma.\) Then for any fan \(\Sigma\) we define a group
    \[
        \wt{\Aut}(\Sigma)=N_{\Aut_{alg}(U(\Sigma))}(G_\Sigma)
    \]
    to be the normalizer of the group \(G_\Sigma\) in algebraic automorphisms of the variety \(U_\Sigma.\) Denote by \(\NAut^0(\Sigma)\) the centralizer of the group \(G_\Sigma\) in the group \(\wt{\Aut}(\Sigma).\)
\end{definition}
\begin{proposition}\label{prop:red_of_gv}
    There is the following isomorphism of algebraic group schemes.
    \begin{align*}
        \wt{\Aut}^0(\Sigma)\cong \TT_\Sigma\times \wt{\Aut}^0(\Sigma^\red).
    \end{align*}
\end{proposition}
\begin{proof}
    First, we observe that the subgroup \(\TT_\Sigma\) is clearly normal in \(\wt{\Aut}^0(\Sigma)\). Hence, there is the following descent homomorphism to automorphisms of \(U(\Sigma)/\TT_\Sigma\). 
    \[p:\wt{\Aut}^0(\Sigma)\to \wt{\Aut}^0(\Sigma^\red).\]
    This homomorphism has a section \(s\). We construct \(s\) as follows. Let \(\vp\in \wt{\Aut}^0(\Sigma^\red)\) be an automorphism of \(U(\Sigma^\red)\). 
    \begin{itemize}
        \item We define \(s(\vp)\) on the points of the form \(\left((z_{s})_{s\in S^\red},(1)_{s'\in S\setminus S^\red}\right)\) as follows.
        \[
            s(\vp)\left((z_{s})_{s\in S^\red},(1)_{s'\in S\setminus S^\red}\right)=\left(\vp((z_s)_{s\in S^\red}), (1)_{s'\in S\setminus S^\red}\right)
        \]
        \item Now, for an arbitrary point \((z_s)_{s\in S^\red},(t_{s'})_{s'\in S\setminus S^\red}\) we extend the automorphism \(s(\vp)\) using \(\TT_\Sigma\)-action. Concretely, we define the image of a generic point \(((z_s)_{s\in S^\red},(t_{s'})_{s'\in S\setminus S^\red})\) under \(s(\vp)\) as follows. Let \(g\) be a unique element in \(\TT_\Sigma\) such that \(g_{S\setminus S^\red}\) coincides with \(\left(t_{s'}\right)_{s'\in S\setminus S^\red}\). Then the value of \(s(\vp)\) on \(((z_s)_{s\in S^\red},(t_{s'})_{s'\in S\setminus S^\red})\) is defined as follows.
        \[
            s(\vp)\left((z_s)_{s\in S^\red},(t_{s'})_{s'\in S\setminus S^\red}\right)=g\cd s(\vp)\left((g_s^{-1}z_s)_{s\in S^\red},(1)_{s'\in S\setminus S^\red}\right).
        \]
    \end{itemize}
    
    Consequently, the map \(p\) is surjective. Now, we study the kernel of the map \(p\). Let \(\psi\in \wt{\Aut}^0(\Sigma)\) be an automorphism of \(U(\Sigma)\) such that \(p(\psi)=\id_{U(\Sigma^\red)}.\) Then it provides a gauge transformation that commutes with \(\TT_\Sigma\)-action of the bundle \(U(\Sigma)\to U(\Sigma^\red)\). Consequently, it could be described as a regular map \(U(\Sigma^\red)\to \Aut_{\TT_\Sigma}(\TT_\Sigma).\) The group \(\Aut_{\TT_\Sigma}(\TT_\Sigma)\) is clearly isomorphic to \(\TT_\Sigma\) itself. It remains now to observe that by the purity of the branch locus, it extends to a non-vanishing regular function on \(\C^{S^\red}\); hence it is constant. As a result, the kernel of \(p\) coincides with the group \(\TT_\Sigma\). 
    
    Finally, since \(p\) has a section, there is a semidirect product splitting of \(\NAut^0(\Sigma)\). Moreover, since \(\TT_\Sigma \sse G_\Sigma\) the image of \(\NAut^0(\Sigma^\red)\) commutes with \(\TT_\Sigma\) and the proposition follows.
\end{proof}

\begin{definition}
    Let \(\Sigma=(S,K,V,\rho)\) be a complete generalized fan.
    Then to each element \(s\in S\) we associate its \emph{degree} the element \(e_s+\im (A^\Z_\Sigma)^\vee\) of the group \(\Z^S/\im (A^\Z_\Sigma)^\vee.\)
\end{definition}
\begin{definition}
    \par Define {\it symmetry group} \(\mf{S}(\Sigma)\) of the fan \(\Sigma\) to be a subgroup of \(\Aut(\Gamma(\Sigma))\) preserving the fan \(\Sigma\) and mapping ghost rays to ghost rays. Alternatively, this group could be described as the group of isomorphisms of fans \(\Sigma\to\Sigma\). Define {\it the inertia group} \(\mf{I}(\Sigma)\) of the fan \(\Sigma\) to be a subgroup in \(\Sm(\Sigma)\) given by the product of permutation groups over groups of rays of the same degree.
\end{definition}
\begin{lemma}
    Subgroup \(\mf{I}(\Sigma)\) is normal in \(\mf{S}(\Sigma).\)
\end{lemma}
\begin{proof}
    Consider an arbitrary automorphism \(\vp\in\Sm(\Sigma).\) Then for each \({\psi\in\In(\Sigma)}\) the conjugate automorphism \(\vp\psi\vp^{-1}\) also permutes rays of the same (possibly different from the initial) degree.
\end{proof}
The previous lemma allows us to give the following definition. 
\begin{definition}\label{def:eff_symm_group}
    {Effective symmetry group} \(\ESm(\Sigma)\) is the quotient group:
    \[
        \ESm(\Sigma)=\Sm(\Sigma)/\In(\Sigma).
    \]
\end{definition}
\begin{example}
    Consider fan \(\Sigma_{\CP^2}\) defining \(\CP^2.\) Then the group \(\Sm(\Sigma)\) is isomorphic to \(\Sm_3.\) Since all rays of the fan \(\Sigma\) have the same degree we have \(\In(\Sigma)=\Sm(\Sigma).\) Thus the group \(\ESm(\Sigma)\) is trivial. 
\end{example}
\begin{proposition}
    The subgroup \(\NAut^0(\Sigma)\) is normal in \(\NAut(\Sigma)\).
\end{proposition}
\begin{proof}
    This immediately follows from the description of \(\NAut^0(\Sigma)\) as the centralizer subgroup of \(G_\Sigma\) in \(\NAut(\Sigma).\)
\end{proof}
\begin{proposition}\label{prop:con_comp_of_norm}
    We have the following exact sequence of group schemes.
    \[
        1\to \wt{\Aut}^0(\Sigma)\to \wt{\Aut}(\Sigma)\to \ESm(\Sigma)\to 1.
    \]
\end{proposition}
The following argument is essentially the same as in the proof of Theorem 4.2 of Cox \cite{Cox_Aut}. 
\begin{proof}
    First, we observe that there is an inclusion of the group \(\Sm(\Sigma)\) into the group \(\NAut(\Sigma)\). Indeed, let \(\Phi\) be an automorphism of the fan \(\Sigma\). Then, it naturally acts on the space \(U(\Sigma)\) since it preserves the combinatorial structure of the fan, and moreover, it maps ghost rays into ghost rays. Since \(\Phi\) is an automorphism of \(\Sigma\) it also preserves the kernel of \(A_\Sigma^\C\) and consequently normalizes \(G_\Sigma\)-action on \(U(\Sigma)\). Consequently, \(\Phi\) belongs to \(\NAut(\Sigma)\). 

    Now, we want to show, that \(\NAut^0(\Sigma)\) and \(\Sm(\Sigma)\) generate the group \(\NAut(\Sigma)\). Let \(\vp\in \NAut(\Sigma)\) be a regular automorphism of \(U(\Sigma)\) that normalizes the \(G_\Sigma\)-action. Then since \(\NAut^0(\Sigma)\) is a normal subgroup of \(\NAut(\Sigma)\) we see that \(\vp (\CT)^S\vp^{-1}\) is a maximal torus in \(\NAut^0(\Sigma)\). Consequently, there is an element \(\psi\in \NAut^0(\Sigma)\) such that \(\psi\vp(\CT)^S(\psi\vp)^-1=(\CT)^S\). Thus, the automorphism \(\psi\vp\) is \((\CT)^S\)-equivariant. Now by a theorem of Oda \cite[Theorem 1.13]{Oda} this map is induced by an automorphism of \(\wt{\Sigma}\). It remains to note that since \(\vp\in \NAut(\Sigma)\) this automorphism of \(\wt{\Sigma}\) actually descends to an automorphism of \(\Sigma\). This completes the proof of the generation claim.

    Finally, it remains to observe that the condition \(\vp\in \NAut^0(\Sigma)\cap \Sm(\Sigma)\) is equivalent to the condition \(\vp\in\mf{I}(\Sigma).\) This completes the proof of the proposition.  
\end{proof}
\begin{remark}
    As we will see below, the group \(\wt{\Aut}^0(\Sigma)\) is connected. Thus, Proposition \ref{prop:con_comp_of_norm} shows that the group of components of the algebraic group scheme \(\wt{\Aut}(\Sigma)\) is isomorphic to \(\ESm(\Sigma)\).
\end{remark}

\subsection{Graded automorphisms of the Cox ring and automorphisms of the Cox construction}
Note that Definition \ref{def:Demazure_roots} does not take into account any convex structure of the generalized \(\Sigma\) except for the position of the \(1\)-dimensional cones. Consequently, the set \(\DR(\Sigma)\) corresponds to the roots of the group \(\Aut_g(\Sigma),\) but it does not depend on the geometry of~\(U(\Sigma)\) and in particular on \(K.\) To establish this connection we give the following definition.
\begin{definition}\label{def:geometric_Demazure_roots}
    We define a set of \emph{geometric Demazure roots} of a generalized fan \(\Sigma=\GF\) denoted by \(\DRG(\Sigma)\) as follows.
    \begin{itemize}
        \item Given a face \(\sigma\in K\), and a Demazure root \(m\in \DR(\Sigma)\) we denote by \(\sigma_m\) the subset of \(\sigma\) consisting of vertices \(s\in \sigma\) such that \(\la m, \rho(s)\ra=0.\)
        \item We now define \(\DRG(\Sigma)\) via the following formula.
        \begin{gather*}
        \DRG(\Sigma)=\left\{m_s\in \DR(\Sigma)\mid \text{ for all } \sigma\in K\text{ there is } \tau\in K\text{ such that } s\in \tau,\sigma_{m_s}\sse \tau\right\}.
        \end{gather*}
    \end{itemize}
    We also define the subset of \emph{semisimple} and \emph{unipotent} geometric Demazure roots as intersections of the set \(\DRG(\Sigma)\) with the subsets \(\DR_s(\Sigma)\) and \(\DR_u(\Sigma)\) respectively. Consequently, we denote these intersections as \(\DRG_s(\Sigma),\) and \(\DRG_u(\Sigma).\)
\end{definition}
\begin{lemma}\label{lem:geom=non_geom_DR_for_str_complete}
    Let \(\Sigma=\GF\) be a strongly complete generalized fan. Then the sets \(\DR(\Sigma)\) and \(\DRG(\Sigma)\) coincide.
\end{lemma}
\begin{proof}
    This is an immediate corollary of Definition \ref{def:strongly_complete} of the strongly complete fans and Definition \ref{def:geometric_Demazure_roots}. 
\end{proof}
\begin{lemma}\label{lem:ext_lem}
    Let \(\Sigma=\GF\) be a generalized with no ghost vertices, i.e., \(S^\red=S\). Then any holomorphic or regular automorphism of \(U(\Sigma)\) extends to an automorphism of \(\C^S\).  
\end{lemma}
\begin{proof}
    First, we observe that by \cite[Lemma 1.4]{Cox_Aut} the codimension of \(\C^S\setminus U(\Sigma)\) is at least \(2\) in \(\C^S\). Now, the result in the holomorphic setting is implied by the Hartogs phenomenon and by the purity of branch locus in the algebraic setting. 
\end{proof}
\begin{proposition}
    A one parametric subgroup \(y_{m}\) preserves the subvariety \(U(\Sigma)\) inside of \(\C^S\) if and only if \(m\) is a geometric Demazure root.
\end{proposition}
\begin{proof}
    One direction is very close to the proof of \cite[Proposition 4.6]{Cox_Aut} and \cite[Proposition 3.14]{Oda}. We provide it here for the sake of completeness and also to motivate the notion of geometric Demazure roots introduced above.
    First, we observe that by Lemma \ref{lem:uni_space_struct}, the space \(U(\Sigma)\) can be presented as the following union of affine varieties.
    \begin{align}
        U(\Sigma)=\bigcup_{\sigma\in K} U(\Sigma,\sigma).
    \end{align}
    For a face \(\sigma\in K\) and a tuple \((z_s)_{s\in S}\) we denote by \(z^{\wh{\sigma}}\) the product \(\prod_{s\in S\setminus \sigma} z_s\). Each of the varieties \(U(\Sigma,\sigma)\) is characterized as follows.
    \[
        U(\Sigma,\sigma)=\left\{(z_s)_{s\in S}\in \C^S\mid z^{\wh{\sigma}}\neq 0\right\}
    \]
    There are two options to consider.
    \begin{enumerate}
        \item\label{prop:invariance_part_a} If \(m\) corresponds to a vertex \(s\) of \(\sigma\) we see that \(y_m(\lambda)(z_s)^{\wh{\sigma}}=(z_s)^{\wh{\sigma}}\). Hence \(y_m(\lambda)(z_s)\) also belongs to \(U(\Sigma,\sigma)\).
        \item If \(m\) correspond to a vertex \(s\) not in \(\sigma\) we consider the following splitting of \(U(\Sigma,\sigma)\).
        \begin{align}
            \begin{split}
            U(\Sigma,\sigma)'=\left\{(z_\bt)_{\bt\in S}\in U(\Sigma,\sigma)\bigg\vert z_s\neq -\lambda \prod_{s'\neq s} z_{s'}^{\la -m,\rho(s')\ra}\right\}\\
            U(\Sigma,\sigma)''=\left\{(z_\bt)_{\bt\in S}\in U(\Sigma,\sigma)\bigg\vert \prod_{s'\neq s} z_{s'}^{\la -m,\rho(s')\ra}\neq 0\right\}
            \end{split}
        \end{align}
        Since \((z_t)\in U(\Sigma,\sigma)\) we see that \(z_s\neq 0\) and consequently \(U(\Sigma,\sigma)=U(\Sigma,\sigma)'\cup U(\Sigma,\sigma)''\). Now, for any element \((z_t)\) of \(U(\Sigma,\sigma)'\) the image \(y_m(\lambda)((z_t))\) is in \(U(\Sigma,\sigma)\). 

        Observe, that since \(m\) is a geometric Demazure root there exists a cone \(C(\tau)\) such that \(s\) and \[\{s'\in \sigma\mid \la \rho(s'),m\ra=0\} \] are elements of \(\tau\). We see that \(U(\Sigma,\sigma)''\sse U(\Sigma,\tau)\). Clearly, \(s\) is a vertex of \(\tau\) and hence part \ref{prop:invariance_part_a} applies to give that \(y_m(\lambda)U(\Sigma,\sigma)''\) is a subset of \(U(\Sigma,\tau)\).
    \end{enumerate}

    For the other direction, assume that \(m_s\) is a non-geometric Demazure root. Then there exists a cone \(\sigma\) such that \(\sigma_m\) is not contained in any face of \(K\) together with the element \(s.\) Consider the open subset of the form \(U(\Sigma,\sigma)\) corresponding to \(\sigma\). We know that by assumption \(s\notin \sigma\). Consequently, \(t_s\neq 0\) for all \((t_{s'})_{s'\in S}\in U(\Sigma,\sigma)\). We now examine the action of \(y_{m_s}\) on \(U(\Sigma,\sigma)\). Consider the point of the form \((z_\bt)_{\bt\in S}=((1)_s, (0)_{s'\in \sigma},(1)_{s''\notin \sigma\cup\{s\}})\). For \(\lambda=-1\) we get \(y_{m_s}(-1)(t_\bt)\notin U(\Sigma,\sigma)\). Now, clearly, there is no other \(\tau\) which would contain both \(\sigma\) and \(s\). Consequently, the point \(y_{m_s}(-1)(t_\bt)\) lies outside of~\(U(\Sigma)\).
\end{proof}
\begin{proposition}\label{prop:struct_of_NAut0}
    Let \(\Sigma=\GF\) be a complete generalized fan. Then, the group \(\wt{\Aut}^0(\Sigma)\) is isomorphic to the subgroup of \(\Aut_g(\Sigma)\) generated by the maximal torus \((\CT)^S\) and the one-parametric subgroups corresponding to the geometric Demazure roots of \(\Sigma.\)
\end{proposition}
The following is very close to the proof of Proposition 4.5 of Cox \cite{Cox_Aut}. The only non-triviality is the addition of the geometric condition on Demazure roots.
\begin{proof}
    First, by Corollary \ref{cor:aut_g_red_of_gv} and Proposition \ref{prop:red_of_gv}, we can assume that \(\Sigma\) has no ghost vertices. 
    Now, we observe that the argument in the proof of Proposition 4.5 of Cox \cite{Cox_Aut} can be adopted verbatim to show that the group \(\Aut_g(\Sigma)\) is generated by \((\CT)^S\) and root subgroups of the form \(y_m(\lambda)\) (see Construction \ref{con:aut_assoc_to_a_root}). 
    \par Moreover, there is an injective homomorphism \(\NAut^0(\Sigma)\to \Aut_g(\Sigma)\). Indeed, by Lemma \ref{lem:ext_lem}
    any automorphism of \(U(\Sigma)\) can be extended to an automorphism of \(\C^S\) if \(\Sigma\) has no ghost vertices, and the latter automorphism clearly preserves grading on \(R(\Sigma)\).

    We now observe that by Proposition \ref{prop:Alisa_Lie}, the Lie algebra of the group \(\NAut^0(\Sigma)\) is generated by the root subspaces and the invariant the centralizer of the Lie algebra of \((\CT)^S\). By Corollary \ref{cor:non_triv_act_on_roots_sub} this centralizer coincides with the Lie algebra of the torus. Consequently, the Lie algebra of \(\NAut^0(\Sigma)\) is generated by the root subspaces and the Lie algebra of the torus. Thus, the group \(\NAut^0(\Sigma)\) is generated by the root subgroups and the torus.
\end{proof}
\subsection{Main result}
\begin{construction}\label{con:quot_fan}
    Let \(\Sigma=\GF\) be a generalized fan. Let \(L\) be a subspace in the vector space \(V\). Denote by \(\Sigma/L\) the pushforward fan of \(\Sigma\) along the projection \(V\to V/L\).
\end{construction}
\begin{theorem}\label{thm:equiv_on_Cox_constr}
    Let \(\Sigma=\GF\) be a generalized fan. Let \(H\) be an arbitrary connected complex Lie subgroup of the torus \((\CT)^S\). Denote by \(\h^\R\) the image of the subspace \(\Lie(H)\sse \Lie((\CT)^S)=\C^S\) under the map \(\Re\). Assume that the fan \(\Sigma/\h^\R\) is complete. Then, the group of holomorphic automorphisms of the Cox construction of the fan \(\Sigma\) that normalize the action of \(H\) is described as follows.
    \begin{enumerate}
        \item There is an isomorphism of the following form.
        \[
            \NAut(\ul{\wt{\Sigma}/\h^\R})=\NAut\left(\wt{\Sigma}/\ol{\h^\R}\right).
        \]
        \item The group \(\wt{\Aut}(\Sigma)\) is isomorphic to a semidirect product of the following form.
        \[
            \NAut^0(\Sigma/\h^\R)\rtimes\mf{ES}(\Sigma/\h^\R).
        \]
        \item The action of \(\mf{ES}(\Sigma/\h^\R)\) on $\Aut_g(\Sigma/\h^\R)$ is described by the following formula.
        \[
            \sigma(g)=g\sigma^*,\quad \sigma\in \ESm(\Sigma/\h^\R),\quad g\in \Aut_g(\Sigma/\h^\R).
        \]
        \item  The algebraic group \(\NAut^0(\Sigma/\h^\R)\) is connected and its Lie algebra decomposes into the following direct sum.
        \[
            \Lie(\NAut^0(\Sigma/\h^\R))=\Lie((\CT)^S)\oplus \bigoplus_{m\in \DRG(\Sigma/\h^\R)}\Lie(y_m(\C^\times)).
        \]
        \item If the fan \(\Sigma/\h^\R\) is strongly complete the Lie algebra of \(\NAut^0(\Sigma/\h^\R)\) decomposes into the following direct sum.
        \[
            \Lie(\NAut^0(\Sigma/\h^\R))=\Lie((\CT)^S)\oplus \bigoplus_{m\in \DR(\Sigma/\h^\R)}\Lie(y_m(\C^\times)).
        \]
        That is, all Demazure roots are geometric.
    \end{enumerate}
\end{theorem}
\begin{proof}
    Essentially, we already proved the theorem above. Now, we just quickly summarize the previous results in the relevant order.
    \begin{enumerate}
        \item This part follows from Proposition \ref{prop:zariski_closure}, and \ref{cor:normalizers_coincide}.
        \item This part follows from Proposition \ref{prop:con_comp_of_norm}, and the observation that the map \(\NAut(\Sigma)\to \ESm(\Sigma)\) admits a section given by the obvious action on \(\C^S\).
        \item This is immediately evident from the description of the group \(\ESm(\Sigma/\h^\R)\) provided in Definition \ref{def:eff_symm_group}.
        \item This part is immediate from Proposition \ref{prop:struct_of_NAut0}.
        \item This part is an immediate consequence of Lemma \ref{lem:geom=non_geom_DR_for_str_complete}. \qedhere
    \end{enumerate}
    
\end{proof}
\begin{remark}
    An instance of the theorem when the fan \(\Sigma/\h^\R\) is strongly complete is given by a complete fan \(\Sigma\) and the space \(\h^\R\) containing \(\wt{H}_\Sigma^\R\).
\end{remark}
Below, we provide a calculation of the group \(\NAut^0(\Sigma/\h^\R)\) for a simple example of a non-embedded fan.
\begin{example}
    Consider the standard fan \(\Sigma_{\CP^2}\) of the projective space \(\CP^2.\) See Example \ref{ex:CP2_fan} for the complete description. Consider the subspace \(\h^\R\) in \(\R^2\) spanned by \(e_1+\sqrt{2}e_2\). In Example \ref{ex:CP2_fan_proj_rat} we saw that the rationalization~\(\ul{\Sigma_{\CP^2}/\h^\R}\) has a zero dimensional underlying vector space. Consequently, the group~\(\NAut(\Sigma_{\CP^2}/\h^\R)\) is isomorphic to \((\CT)^S\).
\end{example}
\part{Applications}\label{part:applications}
\section{Equivariant automorphisms of complete simplicial toric varieties}\label{sec:aut_toric}
\subsection{The automorphism group of a complete simplicial toric variety}
\begin{theorem}[{\cite[Theorem 4.2]{Cox_Aut}}]\label{thm:aut_struct_Cox}
    Let \(\Sigma\) be a rational complete marked fan. Then, the group \(\Aut(V_\Sigma)\) of regular automorphisms of the corresponding toric variety fits into the following exact sequence.
    \begin{gather*}
        1\to G_\Sigma \to \NAut(\Sigma)\to \Aut(V_\Sigma)\to 1.
    \end{gather*}
\end{theorem}
In particular, any regular automorphism of \(V_\Sigma\) lifts to the Cox construction of \(\Sigma\). 
\subsection{Analytic automorphisms of algebraic varieties over \texorpdfstring{\(\C\)}{complex numbers}}
The following fundamental result of Serre comes from the famous \enquote{GAGA} paper \cite{GAGA}.
\begin{proposition}[{\cite[Proposition 15]{GAGA}}]\label{prop:Serre_GAGA_aut}
    All holomorphic maps out of a proper complex variety into an algebraic variety are algebraic. 
\end{proposition}
We have the following immediate corollary of the previous result of Serre.
\begin{corollary}
    Let \(\Sigma\) be a complete rational embedded fan. Then the group \(\Aut(V_\Sigma)\) of holomorphic automorphisms of the corresponding toric variety is isomorphic to the group \(\Aut(V_\Sigma)\) of algebraic automorphisms of the toric variety.
\end{corollary}
\begin{proof}
    Since the fan \(\Sigma\) is complete, the toric variety \(V_\Sigma\) is compact as a complex space; see, for instance, \mbox{\cite[Theorem 1.11]{Oda}}. Thus, the result follows by Proposition \ref{prop:Serre_GAGA_aut}.
\end{proof}
\subsection{Equivariant automorphisms}
\begin{construction}\label{con:quot_fan_2}
    Let \(\Sigma=\GF\) be a rational embedded fan. Denote by \(\pi_\Sigma\) the standard projection map \(\wt{\Sigma}\to \Sigma\) from the canonical developement of the fan \(\Sigma\) to the fan \(\Sigma\), see Construction \ref{con:can_developement}.
    Consider the associated toric variety~\(V_\Sigma\). Denote by \(T\) the standard torus acting on \(V_\Sigma\). Let \(H\) be an arbitrary complex Lie subgroup in \(T\) such that the preimage of \(H\) under the Cox construction map \(\pi:U(\Sigma)\to V_\Sigma\) is connected in \((\CT)^S\). Denote by \(\h^\R\) the image of \(\Lie H\) under the map \(\Re:\Lie T\to V\). Finally, denote by \(\Sigma/\h^\R\) the pushforward fan along the projection~\(V\to V/\h^\R\). 
\end{construction}
\begin{theorem}\label{thm:auto_of_toric_var}
    In terms of Construction \ref{con:quot_fan_2}, the subgroup \(\Aut_H(V_\Sigma)\) of holomorphic automorphisms of \(V_\Sigma\) that normalize the action of \(H\) fits into the following exact sequence of complex Lie groups.
    \[
        1\to \pi^{-1}_\Sigma(H)\to \NAut(\Sigma/\h^\R)\to \Aut_H(V_\Sigma)\to 1.
    \]
\end{theorem}
\begin{proof}
    By Proposition \ref{prop:Serre_GAGA_aut}, any holomorphic automorphism of \(V_\Sigma\) is regular. By Theorem \ref{thm:aut_struct_Cox} any regular automorphism of \(V_\Sigma\) lifts to \(U(\Sigma)\). Consequently, the result follows from Theorem \ref{thm:equiv_on_Cox_constr}.
\end{proof}
\section{Moment-angle manifolds}\label{sec:moment_angle}
In this section, we describe a large class of non-K\"{a}hler complex manifolds called moment-angle manifolds. In the next section we show how the results of Part \ref{part:gen_theory} apply to give a description of their holomorphic automorphism groups. 
\subsection{Moment angle manifolds as a quotient of the Cox construction}\label{sec:MA_complex_struct}
The following key result of Panov--Ustinovsky \cite[Theorem 3.7]{panov_ustinovsky_12} explains how the quotient space \(U(\Sigma)/H_\Sigma\) can be given a complex structure. See \cite[Theorem 6.6.3]{bupa15} for a textbook treatment.
\begin{theorem}[{\cite[Theorem 3.7]{panov_ustinovsky_12}, \cite[Theorem 6.6.3]{bupa15}}]\label{thm:ma_complex_structure}
    Let \(\Sigma=(S,K,V,\rho)\) be a complete (embedded) marked fan. Such that \(H_\Sigma^\R\) admits a complex structure in terms of Construction \ref{con:complex_struct}. Then, the action of the group \(H_\Sigma\) is free and proper, and the quotient \(U(\Sigma)/H_\Sigma\) has a natural structure of a compact complex manifold.
\end{theorem}
\begin{remark}
    The condition that \(\Sigma\) is an embedded fan is essential because it guarantees that the quotient space is Hausdorff in classical topology.
\end{remark}
\begin{definition}\label{def:rat_mam}
    We call a moment-angle manifold \(Z_\Sigma\) \emph{rational} if \(\Sigma\) is a rational fan.
\end{definition}
\subsection{Canonical foliation}
By a result of Ishida \cite{ishida_15}, and Kasuya--Ishida \cite{ishida_19} moment-angle manifolds are equipped with a canonical orbit foliation. 
\begin{construction}\label{con:can_fol}
    Let \(T\) be a maximal torus of the connected component \(\Aut^0(Z_\Sigma)\) of unity in the holomorphic automorphism group \(\Aut(Z_\Sigma)\). Denote its Lie algebra by \(\mf{t}\). We denote by \(F\) the connected holomorphic automorphism Lie group with Lie algebra \(\mf{t}\cap \ol{\mf{t}}.\)
\end{construction}
\begin{proposition}[{\cite[Lemma 2.1]{ishida_19}}]
    The Lie group \(F\) of construction \ref{con:can_fol} does not depend on the choice of the maximal torus \(T\). Moreover, it is centralized by all elements of \(\Aut^0(Z_\Sigma)\).
\end{proposition}
The following result is immediate.
\begin{corollary}
    The Lie group \(F\) is normalized by all holomorphic automorphism of the complex moment-angle manifold \(Z_\Sigma\).
\end{corollary}
There is the following characterization of the group \(F\) extrinsically in terms of of Theorem \ref{thm:ma_complex_structure} construction. The following observation is essentially due to Ishida \cite{ishida_13}, \cite{ishida_15}.
\begin{proposition}
    There is a canonical isomorphism of complex Lie groups \(F_\Sigma\) and \(F\).
\end{proposition}
\begin{proof}
    By \cite[Proposition 5.2]{ishida_15}, the Lie algebra of the Lie group \(F\) is an image under the canonical projection of the subspace \(\wt{G}_\Sigma\). Now, it is immediate that this image is canonically isomorphic to the quotient \(\wt{G}_\Sigma/\h_\Sigma\). It is also clear that the latter is precisely the Lie algebra of \(F_\Sigma\).
\end{proof}
\subsection{Comparison with Ishida's approach}
In \cite{ishida_13}, Ishida developed an approach to the classification of complex manifolds with maximal torus action. This work is of fundamental importance to the field in general and our results in particular. Consequently, we briefly recall some of these results here. 
\begin{definition}[{\cite[Definition 2.1]{ishida_13}}]
    Let \(M\) be a compact connected smooth manifold equipped with an action of a \emph{compact} torus \(T\). For a point \(m\in M\) denote by \(T_m\) its stabilizer subgroup in \(T\). Then the action of \(T\) on \(M\) is \emph{maximal} if there exists a point \(m\in M\) such that the following holds.
    \[
        \dim T_m+\dim T=\dim M.
    \]
\end{definition}
\begin{theorem}[{\cite[Theorem 7.9]{ishida_13}}]\label{thm:ishida_class}
    Let \(M\) be a compact connected complex manifold \(M\) equipped with a maximal action of a compact torus \(T\) which preserves the complex structure. Then, there exists an embedded fan \(\Sigma\) in \(\mf{t}\) and a complex subspace \(\mf{h}\) such that \(M\) is \(T\) equivariantly biholomorphic to \(V_\Sigma/H\), where \(H := \exp(\mf{h}) \sse T^\C\).
\end{theorem}
Let \(\Re:\mf{t}^\C\to \mf{t}\otimes 1\cong \mf{t}\) be the projection map that takes the real part of an element of the tangent Lie algebra of \(T\). The pair \((\Sigma,\mf{h})\) of Theorem \ref{thm:ishida_class} satisfies the following two conditions.
\begin{enumerate}
    \item The map \(\Re\) restricted to \(\h\) is injective;
    \item The pushforward of \(\Sigma\) under the projection \(\mf{t}\to \mf{t}/\h\) is a complete embedded fan.
\end{enumerate}
\begin{proposition}[{\cite[Theorem 10.4]{ishida_13}}]
    Let \(Z_\Sigma\) be a moment-angle manifold associated with an embedded fan \({\Sigma=\GF}\). Then \(Z_\Sigma\) is a complex manifold with maximal torus action.
\end{proposition}
\section{Automorphisms of rational moment-angle manifolds}\label{sec:aut_mam}
\subsection{Descent to the toric base}
\begin{proposition}\label{prop:desc_well_def}
    Let \(\Sigma=(S,K,V,\rho)\) be a rational generalized fan. Then there is the canonical {descent homomorphism} \({\pi_*:\Aut(Z_\Sigma)\to\Aut(V_\Sigma)}\) that maps holomorphic automorphisms of the moment-angle manifold to the automorphisms of the corresponding toric variety. That is any automorphism \(\vp\in \Aut(Z_\Sigma)\) normalizes the quotient action of \(F_\Sigma=G_\Sigma/H_\Sigma\) on \(Z_\Sigma\).
\end{proposition}
\begin{proof}
    By \cite[Lemma 2.1]{ishida_19}, the orbit foliation of this action can be identified with the intersection of the orbit foliation of a maximal torus of the group \(\Aut(Z_\Sigma)\) with its complex conjugate. Moreover, by the same Lemma, this foliation is independent of the choice of maximal tori. Hence, all holomorphic automorphisms of \(Z_\Sigma\) normalize \(G_\Sigma/H_\Sigma\)-action and thus descend to the quotient variety. 
\end{proof}
\begin{proposition}\label{prop:desc_ker}
    Under the assumptions of Proposition \ref{prop:desc_well_def}, there is a canonical identification of the following complex Lie subgroups in \(\Aut(Z_\Sigma)\).
    \[
        \ker \pi_*= F_\Sigma.
    \]
\end{proposition}
\begin{proof}
    Since \(\Sigma\) is rational \(F_\Sigma\) is a compact complex torus. Denote by \(\Gamma\) the lattice corresponding to \(F_\Sigma\). Consider an automorphism \(\vp\in \Aut(Z_\Sigma)\) such that \(\pi_*(\vp)=\id_{V_\Sigma}\). Then \(\vp\) defines a gauge transformation of the \(F_\Sigma\)-principal bundle \(Z_\Sigma\to V_\Sigma.\) Hence, it is given by a holomorphic map \(\Upsilon:V_\Sigma\to \Aut(F_\Sigma).\)We know, that the right-hand side is isomorphic to \(F_\Sigma\rtimes GL(\Gamma).\) Since, \(V_\Sigma\) is connected this map factors through \(F_\Sigma\). Now, because the fundamental group of \(V_\Sigma\) is finite there exists a lift \(\wt{\Upsilon}:V_\Sigma\to \wt{F}_\Sigma\rtimes GL(\Gamma_\Sigma)\). Hence, by compactness of \(V_\Sigma\), this map is constant, and we see that \(\vp\in F_\Sigma\). Thus, we have proved that \(\ker\pi_*\sse F_\Sigma\). The converse inclusion is obvious.
\end{proof}
\subsection{Lifts to the universal space}
\begin{proposition}\label{prop:desc_surj}
    Let \(\Sigma=(S,K,V,\rho)\) be a rational embedded complete fan. The following descent homomorphism is surjective. \[\pi^H_*:\wt{\Aut}(\Sigma)\to \Aut(Z_\Sigma)\]
\end{proposition}
\begin{proof}
    Let \(\vp\in\Aut(Z_\Sigma)\) be an automorphism of the moment-angle manifold. Then there exists a descent \(\pi^F_*\vp\) of \(\vp\) to the variety \(V_\Sigma.\) We can then lift \(\pi^F_*\vp\) along \(\pi^G_*\) and obtain an automorphism \(\wt{\vp}\) of \(U(\Sigma).\) From Proposition \ref{prop:desc_ker} descent \(\pi^H_*\wt{\vp}\) differs from \(\vp\) by an element of \(F_\Sigma\). Hence \(\pi^H_*\) is surjective.
\end{proof}
\begin{proposition}\label{prop:ker_of_desc_from_uni}
    The kernel of the descent homomorphism \(\pi_*^H\) coincides with the group \(H_\Sigma\).
\end{proposition}
\begin{proof}
    By Proposition \ref{prop:Norm_of_G} the kernel consists of elements in \(\wt{\Aut}(\Sigma)\) that map to \(\id_{V_\Sigma}\) under the homomorphism \(\pi^G_*\). Hence, by Theorem \ref{thm:equiv_on_Cox_constr} they belong to \(G_\Sigma\). It remains to observe that the kernel of the restriction \(\pi_*^H\big\vert_{G_\Sigma}\) is precisely~\(H_\Sigma\).
\end{proof}
\subsection{Structure of the automorphism group}
\begin{theorem}\label{thm:ma_aut_group_structure}
    Let \(\Sigma=(S,K,V,\rho)\) be a complete rational generalized fan. Then there is a short exact sequence of complex Lie groups.
    \begin{gather*}
        1\to H_\Sigma\to \wt{\Aut}(\Sigma)\to \Aut(Z_\Sigma)\to 1.
    \end{gather*}
\end{theorem}
\begin{proof}
    The result immediately follows from Propositions \ref{prop:desc_ker}, \ref{prop:desc_surj}, \ref{prop:ker_of_desc_from_uni} and the description of the automorphism group of the toric variety of \S \ref{sec:aut_toric}.
\end{proof}
\begin{corollary}\label{cor:aut_detect_complex_struct}
    Let \(\h_\Sigma\) and \(\h_\Sigma'\) be two complex structure on \(H_\Sigma^\R\) in Construction \ref{con:complex_struct}. Then \(U(\Sigma)/H_\Sigma\) is biholomorphic to \(U(\Sigma)/H_\Sigma'\) if and only if the quotients \(G_\Sigma/H_\Sigma\) and \(G_\Sigma/H_\Sigma'\) are isomorphic as complex Lie groups.
\end{corollary}
\begin{proof}
    One direction is immediate. To prove the converse statement we observe that by Theorem \ref{thm:ma_aut_group_structure} the center of the group \(\Aut(Z_\Sigma)\) is isomorphic to the quotient of the center of \(\NAut(\Sigma)\) by the group \(H_\Sigma\). Consequently, the result follows from the fact that the center of the group \(\NAut(\Sigma)\) is isomorphic to the center of the group \(\NAut(\Sigma')\) if and only if the quotients \(G_\Sigma/H_\Sigma\) and \(G_\Sigma/H_\Sigma'\) are isomorphic as complex Lie groups.
\end{proof}
The above corollary shows that the holomorphic automorphism group of \(Z_\Sigma\) allows one to distinguish between different complex structures on the underlying smooth manifold \(Z_\Sigma\).
\section{Automorphisms of Calabi--Eckmann and Hopf manifolds}\label{sec:CE_Hopf}
\subsection{Automorphisms of rational Calabi--Eckmann manifolds}\label{sec:CE_aut}
\subsubsection{Definition of Calabi--Eckmann manifolds} In this section we generally follow exposition at \cite[Example 6.7.9]{bupa15}. Let \(\Sigma=\GF\) be the following generalized fan.
\begin{itemize}
    \item The set \(S\) consists of \(p+q+2\) elements, that we present as follows.
        \[
             S=\left\{0,\ldots,p,\ol{0},\ldots,\ol{q}\right\}.
        \]
    \item The simplicial complex \(K\) is a product of a \(p\)-simplex and a \(q\)-simplex with the standard triangulated structure.
    \item The vector space \(V\) is a product of \(\R^p\) and \(\R^q.\)
    \item The map \(\rho\) restricted to vertices \(\left\{0,1\ldots,p\right\}\) and corestricted to \(\R^p\times 0\) coincides with the map of the standard fan defining \(\CP^p\). Similarly, the restriction of \(\rho\) to \(\left\{\ol{0},\ol{1}\ldots,\ol{p}\right\}\) corestricted to \(0\times \R^q\) coincides with the map for the standard fan defining \(\CP^q\).
\end{itemize}
\begin{remark}
    The above definition is a particular case of the general construction of a product for generalized fans. We will treat this construction in greater detail in the sequel.
\end{remark}
In more concrete terms, the moment-angle manifold \(Z_\Sigma\) is the total space of a Hopf-type fibration over the product of \(\CP^p\) and \(\CP^q.\) Topologically, it is a product \(S^{2p+1}\times S^{2q+1}\) of two odd-dimensional spheres. We will denote this manifold by \(\CE(p,q).\)

We can also describe the fan operator more explicitly. Denote by \(\{g_1,\ldots g_p,g_{\ol{1}},\ldots,g_{\ol{q}}\}\) the standard basis of the vector space \(\R^p\times \R^q\). Then, we have the following description of the fan operator. 
\begin{gather*}
    A_\Sigma^\R(e_i)=\begin{cases}
        g_{i+1}, & \text{for } 0\le i< p,\\
        -g_1-\ldots - g_{p-1}, & i=p.
    \end{cases}\\
    A_\Sigma^\R(e_{\ol{i}})=\begin{cases}
        g_{\ol{i}+1}, & \text{for } 0\le i< q,\\
        -g_{\ol{1}}-\ldots - g_{\ol{q-1}}, & i=q.
    \end{cases}
\end{gather*}
By the previous description, the kernel of the operator \(A_\Sigma^\R\) is given by the following formula.
\[
    \ker A_\Sigma^\R=\la e_0+\ldots+e_p,e_{\ol{0}}+\ldots+e_{\ol{q}}\ra.
\]
We introduce the complex structure on the space \(\ker A_\Sigma^\R\) as follows.
\begin{itemize}
    \item Recall that we need to find a complex subspace \(\wt{\h}_\Sigma\) of \(\C^S\) that projects isomorphically to \(\ker A^\R_\Sigma\) under the map \(\Re:\C^S\to \R^S.\) 
    \item Fix a non-real complex number \(\alpha\in\C\setminus \R\) Consider a complex linear map \(\Psi\) defined via the following formula.
    $$
        \Psi:\C\to\C^{S},\quad w\mapsto(\underbrace{w,\ldots,w}_{p+1},\underbrace{\alpha w,\ldots,\alpha w}_{q+1}).
    $$
    \item We put \(\wt{\h}_\Sigma\) to be the image of the map \(\Psi\). It is now an easy calculation to see that the subspace \(\wt{\h}_\Sigma\) projects isomorphically to the subspace \(\ker A_\Sigma^\R\) under the map \(\Re\). 
\end{itemize}
	\begin{definition}
        By the above discussion, the moment-angle manifold $Z_\Sigma$ has a complex structure corresponding to the subspace $\wt{\h}_\Sigma.$ Complex manifold we just constructed is called a {\it Calabi--Eckmann manifold of type \((p,q)\)}. 
    \end{definition}

    These manifolds were first introduced by Calabi and Eckmann in \cite{Calabi_Eckmann}, as an early example of non-K\"ahler complex manifolds. The Dolbeault cohomology of Calabi--Eckmann manifolds was computed by Borel in \mbox{\cite[Appendix II, \S 9]{Hirz_appendix}}, as an application of his famous spectral sequence.
	\subsubsection{An abstract computation of an automorphism group}
    We will now compute the automorphism group of the Calabi--Eckmann manifold of type \((p,q)\) using the general theory developed in the previous sections. 
    \begin{notation}
        For notational convenience, we denote by \(\Sigma_{\CP^k}\) the standard generalized fan of \(\CP^k.\) Also, when we want to specify the type of the Calabi--Eckmann manifold, we will denote its fan by \(\Sigma_{p,q}.\)
    \end{notation}
	\begin{proposition}\label{Corollary:CE_abstract}
		$\Aut(\CE(p,q))\cong(GL(p+1)\times GL(q+1))/H_\Sigma.$
	\end{proposition}
	\begin{proof}
		We know that the Cox ring of a product $\VS=\CP^p\times\CP^q$ is 
		\begin{gather*}
			R(\Sigma)=\C[x_0,\ldots, x_p;y_0,\ldots, y_q],\\ \deg(x_i)=(1,0),\; \deg(y_j)=(0,1),\;\text{for } 0\le i\le p,\;0\le j \le q. 
		\end{gather*}
		It follows from the fact that $S=\{0,\ldots,p\}\sqcup\{\ol{0},\ldots,\ol{q}\},$ and the grading group
		$$
		L_\Sigma=L_{\Sigma_{\CP^p}}\oplus L_{{\Sigma_{\CP^q}}}\cong\Z\oplus\Z.
		$$
		We have ${R(\Sigma)_{(1,0)}=\C\la x_i\;|\;0\le i\le p\ra,}$ and ${R(\Sigma)_{(0,1)}=\C\la y_j\;|\;0\le j\le q\ra.}$
		\par To describe the group $\NAut$ we use Theorem \ref{thm:ma_aut_group_structure}, Theorem \ref{thm:equiv_on_Cox_constr}, and Proposition \ref{prop:cox_struct}. 
        Since \(\Sigma\) is an embedded complete fan, we deduce that \(\NAut^0(\Sigma)\) coincides with \(\Aut_g(\Sigma)\). Consequently, by Proposition \ref{prop:cox_struct} there is an isomorphism 
        \[
            \Aut_g(\Sigma)\cong GL(p+1)\times GL(q+1).
        \]
        It remains to observe that the inertia subgroup \(\mf{I}(\Sigma)\) is contained in the group \(H_\Sigma.\) Hence, we have the following isomorphism.
        \[
            \NAut(\Sigma)\cong \Aut_g(\Sigma).
        \]
		Now, it is immediate from Theorem \ref{thm:ma_aut_group_structure} that the automorphism group of the Calabi--Eckmann manifold of type \((p,q)\) is isomorphic to the group \(\left(GL(p+1)\times GL(q+1)\right)/H_\Sigma.\)
	\end{proof}
    We immediately obtain the following description of the automorphism group of (diagonalizable) Hopf manifolds. 
    \begin{corollary}
        Automorphism groups of Hopf manifolds, i.e. Calabi--Eckmann manifolds of type \((p,0)\), are isomorphic to the following complex Lie groups.
        \[
            \Aut(\CE(p,0))\cong GL(p+1)\times \CT/H_{\Sigma_{p,0}}.
        \]
    \end{corollary}
	\subsubsection{A concrete example}
	Consider a well-known example of two complex structures on a moment-angle manifold corresponding to the simplicial complex $\Delta^1\times \Delta^1\times \Delta^2\times \Delta^2,$ see \cite[Example 6.7.10]{bupa15}.
	These two structures come from two different decompositions:
	$$
		Z_\Sigma=\CE(1,1)\times \CE(2,2);\quad Z_\Sigma=\CE(1,2)\times \CE(1,2).
	$$
	\begin{corollary}
		For the first decomposition, we have the following isomorphism.
		$$\Aut(\CE(1,1)\times \CE(2,2))\cong(GL(2)\times GL(2))/H_{\Sigma_{1,1}}\times (GL(3)\times GL(3))/H_{\Sigma_{2,2}}.$$
		For the second decomposition, we have the following isomorphism.
		$$\Aut(\CE(1,2)\times \CE(1,2))\cong(GL(2)\times GL(3))/H_{\Sigma_{1,2}}\times (GL(2)\times GL(3))/H_{\Sigma_{1,2}}.$$
	\end{corollary}
	\begin{proof}
		In both cases, the toric variety associated to the fan \(\Sigma\) is isomorphic to the following product.
		$$
			\VS=\CP^1\times\CP^1\times\CP^2\times\CP^2.
		$$
		Hence, the grading group \(L_\Sigma\) is isomorphic to $\Z^4,$ and the Cox ring is isomorphic to the following graded \(\C\)-algebra.
		\begin{gather*}
			\C[x_0,x_1;y_0,y_1;z_0,z_1,z_2;w_0,w_1,w_2],\\ \deg(x_i)=(1,0,0,0),\;\deg(y_j)=(0,1,0,0), \text{ for } i,j=0,1;\\ 
			\deg(z_k)=(0,0,1,0),\;\deg(w_s)=(0,0,0,1), \text{ for } k,s=0,1,2.
		\end{gather*}
		Similarly to the proof of Proposition \ref{Corollary:CE_abstract} we have that
		$$
			\wt{\Aut}(\Sigma)=GL(2)\times GL(2) \times GL(3)\times GL(3).
		$$
		Applying Proposition \ref{Corollary:CE_abstract}, we see that in the first case, there is the following isomorphism.
		$$
			\Aut^0(\CE(1,1)\times \CE(2,2))\cong(GL(2)\times GL(2))/H_{\Sigma_{1,1}}\times (GL(3)\times GL(3))/H_{\Sigma_{2,2}}.
		$$
		Analogously, in the second case, we have an isomorphism:
		$$
			\Aut^0(\CE(1,2)\times \CE(1,2))\cong(GL(2)\times GL(3))/H_{\Sigma_{1,2}}\times (GL(2)\times GL(3))/H_{\Sigma_{1,2}}.
		$$
        We observe that by the description of the Cox ring obtained above, the inertia subgroup of \(\Sigma\) is isomorphic to the following product of symmetric groups.
        \[
            \Im(\Sigma)=\Sm_2\times \Sm_2\times \Sm_3\times \Sm_3.
        \]
        Hence, the group of \(\ESm(\Sigma)\) is isomorphic to the quotient of the automorphism group of the product of the square, and two triangles modulo the group \(\Im(\Sigma)\).
	\end{proof}
    
\subsection{Comparison with automorphisms of Hopf surfaces}
\par In this section we provide computation of automorphism groups of Hopf surfaces based on \cite{namba}. This computation does not rely on Theorem \ref{thm:ma_aut_group_structure}. However, we provide it as an example of a previously known result that can be compared with the general theory we developed.
\par In terms of section \ref{sec:CE_aut} Hopf surfaces are Calabi--Eckmann manifolds of the type $\CE(1,0).$ However, there is a classical description of Hopf surfaces, which we state below.
\subsubsection{Matrix construction of Hopf surfaces}
	Let $GL(2)$ be the group of invertible $2\times 2$ complex matrices. This group acts from the right on a complex vector space $\C^2.$\\
	Consider a subset $M\ss GL(2)$ given by
	$$
		M= \left\{\begin{pmatrix}\alpha & 0\\ 0 & \beta \end{pmatrix} \Bigg\vert\; \alpha , \beta \in \C, \; 0 < |\alpha| < 1,\; 0<|\beta|< 1\right\}.
	$$
	The set $M$ has the structure of a complex manifold. Denote by $0$ the origin in~$\C^2.$ We are interested in the action of elements of the set $M$ on the space~$\C^2\setminus 0.$

    \begin{construction}
        For each element $u\in M$ consider a group ${G_u=\{u^n\,|\,n\in\Z\}\cong\Z.}$ The group $G_u$ acts on $\C^2\setminus 0$ properly discontinuous. We consider a complex manifold $\Hopf_u=(\C^2\setminus 0)/G_u.$
        The manifold $\Hopf_u$~is called {\it a Hopf surface.} 
    \end{construction}
	
    The holomorphic automorphism group $\Aut(\Hopf_u)$ of a Hopf surface $\Hopf_u$ is described in the following theorem.
	{\begin{theorem}[{\cite[\S 2 Cases 1-4]{namba}}]\label{theorem:aut_struct_Hopf}
		Let $C_u$ be the centralizer of the group $G_u$ in the group $GL(2).$
		\begin{enumerate}
			\item Assume that $G_u=\left\{\begin{pmatrix}\alpha^n & 0\\ 0 & \alpha^n\end{pmatrix}\Bigg\vert\; \alpha\neq 0,\; n\in\Z\right\},$ then we have an exact sequence:
			$$
				1\to G_u\to C_u\to \Aut(\Hopf_u)\to 1.
			$$\label{theorem:aut_struct_Hopf-1}
			Moreover, we have $C_u/G_u=GL(2)/G_u.$
			\item Assume that $G_u=\left\{\begin{pmatrix}\alpha^{mn} & 0\\ 0 & \alpha^n\end{pmatrix}\Bigg\vert\; \alpha\neq 0,\; n\in\Z\right\},$ where $m$ is some positive ineteger. Then we have an exact sequence:
			$$
				1\to G_u\to C_u\times\C\to \Aut(\Hopf_u)\to 1.
			$$\label{theorem:aut_struct_Hopf-2}
			\item Assume that \scalebox{0.85}{$G_u=\left\{\begin{pmatrix}\alpha^n & 0\\ 0 & \beta^n\end{pmatrix}\Bigg\vert\; \alpha,\beta\neq 0\text{ and }\alpha^q\neq\beta \text{ for any } q\in\Q,\; n\in\Z\right\}.$} Then we have an exact sequence:
			 $$
			1\to G_u\to C_u\to \Aut(\Hopf_u)\to 1.
			$$\label{theorem:aut_struct_Hopf-3}
		\end{enumerate}
	\end{theorem}}
	\begin{lemma}\label{lemma:Hopf_mono}
		There is a monomorphism of the following form. 
		$$
			 C_u\times \C \hookrightarrow \NAut(\C^2\setminus \{0\}).
		$$
		Here $\NAut(\C^2\setminus 0)$ is the algebraic automorphism group of the variety $\C^2\setminus 0.$
	\end{lemma}
	\begin{proof}
		To prove the assertion we describe explicitly the group $C_u\times\C.$ From the results of \cite[\S 2]{namba} this group consists of pairs:
		$$
		(v,b),\quad\text{ where } v=\begin{pmatrix}a & 0\\ 0 & d\end{pmatrix},\; ad\neq 0 \text{ and } b \in \C.
		$$
		These pairs form a group under multiplication $\circ$ defined by
		$$
			(v, b)\circ (v',b') = (vv', a'b + b'd^m).
		$$
		Here $v=\begin{pmatrix} a & 0 \\
		0 & d
		\end{pmatrix},$ $v'=\begin{pmatrix}
		a'&0\\
		0&d'
		\end{pmatrix},$ and $b,b'\in\C.$ Hence, an action of the group $C_u\times\C$ on the space $\C^2$ is given by
		$$
			\bigg(\begin{pmatrix} a & 0 \\ 0 & d\end{pmatrix}, b\bigg)\circ \begin{pmatrix} z \\ w \end{pmatrix} = \begin{pmatrix} a z + b w^m \\ d w\end{pmatrix},\quad \begin{pmatrix} z \\ w \end{pmatrix}\in \C^2.
		$$
		Thus we have a monomorphism $C_u\times \C \hookrightarrow \NAut(\C^2\setminus 0).$
	\end{proof}
	\subsubsection{Hopf surfaces and moment-angle manifolds}\label{subsec:hopf_and_ma}
	Let $K$ denote a simplicial complex on the set $\{1,2,3\}.$ The maximal simplices of $K$ are $\{1\},$ and $\{2\}.$ Consider a configuration $\mb{a}_1,\mb{a}_2,\mb{a}_3$ of real numbers corresponding to the fan of the complex $K.$ This means that we have two linear dependencies of the following form.
	$$
		\lambda_1\mb{a}_1+\lambda_2\mb{a}_2=0;\; \mu_1\mb{a}_1+\mu_2\mb{a}_2+\mb{a}_3=0,\; \text{ for some } \lambda_i, \mu_j>0.
	$$ 
	We endow moment-angle manifold $Z_\Sigma\cong S^3\times S^1$ with a complex structure given in \S \ref{sec:MA_complex_struct}. Recall that $Z_\Sigma$ is the quotient of the following form $U(\Sigma)/H_\Sigma.$ We have $U(\Sigma)=(\C^2\setminus 0)\times \C^\times,$ and $C_\Psi\cong\C$ as a complex Lie group.\\
	By definition, we define the group \(H_\Sigma\) as follows.
	$$
		H_\Sigma=\{(e^{\zeta_1 w}, e^{\zeta_2 w}, e^w), \text{ where }\zeta_j=\lambda_j+i\cdot \mu_j\}\subset (\C^\times)^3.
	$$
	If the fan $\GF$ is rational, we can describe the moment-angle manifold \(Z_\Sigma\) as a quotient of the following form. 
	\begin{gather*}
	\Hopf_u=(\C^2\setminus\{0\})\times\C^\times/\{(z_1,z_2,t)\sim(e^{2\pi i \zeta_1 w}z_1,e^{2\pi i \zeta_2 w}z_2,e^w t)\}\cong\\
	\cong(\C^2\setminus\{0\})/\{(z_1,z_2)\sim(e^{2\pi i \zeta_1}z_1,e^{2\pi i \zeta_2}z_2)\}\cong(\C^2\setminus\{0\})/G_u.
	\end{gather*}
	\subsubsection{An exact sequence for automorphism groups}
	Here, we relate the matrix approach to Hopf surfaces and the description of Hopf surfaces as moment-angle manifolds.
	If the numbers $\zeta_1,\zeta_2$ are rational, we have by Proposition \ref{prop:zariski_closure} the following isomorphism of algebraic groups.
	$$
	\NAut(\Sigma)=\ms{N}_{\Aut(U(\Sigma))}(G).
	$$
	Here $\ms{N}_{\NAut(U(\Sigma))}(H)$ denotes the normalizer of the group $H$ in the algebraic automorphism group $\NAut(U(\Sigma))$ of the variety $U(\Sigma).$ Denote by $\NAut$ the group $\NAut(\Sigma).$
	\begin{proposition}
		Assume that $\Hopf_u$ is a manifold given in \ref{subsec:hopf_and_ma}. Then there is an exact sequence 
		$$
			1 \to H_\Sigma \xrightarrow{\alpha} \NAut \xrightarrow{\beta} \Aut(\Hopf_u) \to 1.
		$$
	\end{proposition}
	\begin{proof}
		It is clear that the composite map $\beta\circ\alpha$ is constant. Further, by definition, $\alpha$ is a monomorphism.
		\par We prove exactness in $\NAut.$ Assume that $\vp$ is an automorphism 
		$$
		\vp\in \NAut.
		$$ 
		Moreover, let $\vp$ induce a trivial automorphism of the manifold $\Hopf_u.$ It implies that $\vp(A)=A$ for any $H_\Sigma$-invariant subset of the variety $U(\Sigma).$ Further, it is equivalent to
		\begin{equation}\label{eq:orbit_pres}
			\vp(H_\Sigma (a,b,c))=H_\Sigma (a,b,c),\text{ for any } (a,b,c)\in U(\Sigma).
		\end{equation}
		
		\par It is clear that $\{z_i=0\}\cap U(\Sigma)$ is dense in $\{z_i=0\}$ for $i=1,2.$ Further, $\{z_i=0\}\cap U(\Sigma)$ is $C_\Psi$-invariant. Hence we have
		$$\vp(\{z_i=0\}\cap U(\Sigma))=\{z_i=0\}\cap U(\Sigma).$$
		It implies that $\vp^*(z_i)=r_i z_i$ for some regular functions $r_i$ on $U(\Sigma).$ Applying similar considerations to $\vp^{-1}$ we get $r_i\in \C[z_1,z_2,z_3,z_3^{-1}]^\times.$ A well-known fact is that:
		$$
			\C[z_1,z_2,z_3,z_3^{-1}]^\times=\{\lambda z_3^\nu\;\vert\; \nu \in\Z,\; \lambda \in \C[z_1,z_2]^\times=\C^\times\}.
		$$
		Thus we obtain 
		$
		\vp^*(z_i)=\lambda_i z_3^{\nu_i} z_i,\; \text{for } i =1,2.
		$
		Since $\vp^*$ maps invertible elements of $\C[U(\Sigma)]$ to invertible elements it is clear that $\vp^*(z_3)=\lambda_3 z_3^{\nu_3}.$ Applying the similar considerations to $\vp^{-1}$ we get that $\nu_3=\pm 1.$
		\par Now we can rewrite an equation \eqref{eq:orbit_pres} explicitly:
		\begin{gather}\label{orb_eq}
		\frac{\phi^*(z_i)}{\phi^*(z_3)^{\zeta_i}}=\frac{\lambda_i z_3^{\nu_i} z_i}{\lambda_3^{\zeta_i} z_3^{\nu_3\cdot\zeta_i}}=\frac{z_i}{z_3^{\zeta_i}},\quad i=1,2.
		\end{gather}
		We can restate this in a simpler form:
		\begin{equation}
			\lambda_3^{\zeta_i} z_3^{(\nu_3-1)\zeta_i}=z_3^{\nu_i}\iff (\nu_3-1)\zeta_i=\nu_i;\; \lambda_i=\lambda_3^{\zeta_i},\; i=1,2.
		\end{equation}
		By definition, $\nu_i$ is an integer, and at least on of $\zeta_1, \zeta_2$ is a non-real complex number. It is clear that $\nu_3=1.$ Hence $\vp \in C_\Psi.$
		\par We now show that a map $\beta$ given in the assertion of the proposition is surjective. We recall that there is a monomorphism 
		$$
			\NAut(\C^2\setminus 0)\hookrightarrow \NAut(U(\Sigma)),\; \vp \mapsto \vp \times id_{\C^\times}.
		$$
		From Theorem \ref{theorem:aut_struct_Hopf} we know that in cases \ref{theorem:aut_struct_Hopf-1}, and \ref{theorem:aut_struct_Hopf-3} there is an epimorphism $C_u \twoheadrightarrow \Aut(\Hopf_u).$ Recall that $C_u$ is the centralizer of the group $G_u$ in the group $GL(2).$ Hence we have a monomorphism:
		$$
			C_u \hookrightarrow GL(2) \hookrightarrow \NAut(U(\Sigma)). 
		$$
		For the case \ref{theorem:aut_struct_Hopf-2} of Theorem \ref{theorem:aut_struct_Hopf} we have an epimorphism
		$$
			C_u\times\C \twoheadrightarrow \Aut(\Hopf_u). 
		$$
		Finally, we apply Lemma \ref{lemma:Hopf_mono} to see that there is a monomorphism:
		$$
			C_u\times\C \hookrightarrow \NAut.
		$$ 
		This completes a proof of the proposition.
	\end{proof}
The previous result shows that the automorphism group description of Namba provided by Theorem \ref{theorem:aut_struct_Hopf} fits into the general result of Theorem \ref{thm:ma_aut_group_structure}.
\bookmarksetup{startatroot}
\section{Concluding remarks and future results}\label{sec:conclusion}
\subsection{A brief review of stacks}
For the purposes of this paper, it is enough to consider only \(1\)-truncated stacks. These could essentially be described as homotopy sheaves valued in groupoids on the category of Stein spaces. More concretely, the notion of \(1\)-stack is described by the following definition.
\begin{definition}[{\cite[5.4.4]{Toen}}]
    Let \(\Cc\) be a hypercomplete site; see \cite{Toen} for the definition of a hypercomplete site.
     Then a~\emph{\(1\)-stack} is a presheaf \(\Gs:\Cc^\op\to \Grpd\) on \(\Cc\) valued in groupoids such that for any open cover \(\{U_i\to X\}\) on the object \(X\) we have the following equivalence of groupoids.
\[\begin{tikzcd}
	{\Gs(X)} && {\holim\Bigg(\prod\limits_i \Gs(U_i)} && {\prod\limits_{ij} \Gs(U_{ij})} && {\prod\limits_{ijk}\Gs(U_{ijk})\Bigg)}
	\arrow["\simeq", from=1-1, to=1-3]
	\arrow[shift right=2, from=1-3, to=1-5]
	\arrow[from=1-5, to=1-7]
	\arrow[shift left=2, from=1-3, to=1-5]
	\arrow[shift right=3, from=1-5, to=1-7]
	\arrow[shift left=3, from=1-5, to=1-7]
\end{tikzcd}\]
Here, the homotopy limit is taken in the category of groupoids; see \cite{Toen} for a description of the homotopical structure on this category.
\end{definition}
\begin{example}
    Denote by \(\Sch_\O\) the category of Stein spaces with Grothendieck topology of open covers. Let
    \(\Cc=\Sch_\O\) then the corresponding category of \(1\)-stacks are \emph{complex analytic \(1\)-stacks}. 
\end{example}
\begin{example}
    Denote by \(\Sch_\textrm{aff}\) the category of affine schemes of finite type with Grothendieck topology of Zariski open covers. Let
    \(\Cc=\Sch_\textrm{aff}\) then the corresponding category of \(1\)-stacks are \emph{algebraic \(1\)-stacks}. 
\end{example}
The following example is ubiquitous in geometry.
\begin{example}
    Let \(G\) be a complex Lie group acting on a complex analytic space \(M\). Then, we can form the quotient stack for this action. Let \(X\) be an object in the category \(\Sch_\O.\) Then the groupoid \(\left[M/G\right]\) is defined as follows.
    \begin{itemize}
        \item Objects are the diagrams of the following form, where the vertical arrow is a principal \(G\)-bundle, and the horizontal arrow is a \(G\)-equivariant map.
        \[\begin{tikzcd}
            {\wt{X}} && M \\
            \\
            X
            \arrow[from=1-1, to=1-3]
            \arrow["G"', from=1-1, to=3-1]
        \end{tikzcd}\]
        \item The morphisms are the following commutative diagrams.
        \[\begin{tikzcd}
            {\wt{X}'} \\
            && {\wt{X}} && M \\
            X \\
            && X
            \arrow[from=2-3, to=2-5]
            \arrow["G"', from=2-3, to=4-3]
            \arrow[from=1-1, to=2-3]
            \arrow["G"', from=1-1, to=3-1]
            \arrow["{\Id_X}", from=3-1, to=4-3]
            \arrow[curve={height=-12pt}, from=1-1, to=2-5]
        \end{tikzcd}\]
    \end{itemize}
    We then extend this to a functor valued in groupoids on the category \(\Sch_\O^\op\) using the evident base change definition. 
\end{example}
\subsection{Toric Stacks} The theory of \emph{quantum toric varieties} very closely related to the notion we describe below was developed in great detail in \cite{KLVM}. There, however, the fans used to describe the objects in question were presumed to be embedded, which corresponds to a certain separation property. Below, we give a construction that does not impose this restriction and provides a simultaneous generalization of the approach in \cite{KLVM}. This construction also generalizes other theories of generalized toric varieties found in the literature, such as \cite{prevarieties}, \cite{Borisov_Chen_Smith}. We will explain the precise relationship with each of the existing theories in future work.
\begin{definition}
    Let \(\Sigma=\GF\) be a generalized fan. We define the analytic \emph{toric stack} associated to \(\Sigma\) as the following quotient stack.
    \[
        \V_\Sigma=\left[U(\Sigma)/G_\Sigma\right].
    \]
\end{definition}
\begin{claim}
    The construction of the quotient stack \(\V_\Sigma\) describes all proper analytical stacky compactifications of the quotient group-stack \(\left[(\CT)^S/G_\Sigma\right]\). 
\end{claim}
If \(\Sigma\) is a rational generalized fan. We can additionally recover an algebraic structure on \(\V_\Sigma\).
\begin{definition}
    Let \(\Sigma=\GF\) be a generalized fan. We define the algebraic \emph{toric stack} associated to \(\Sigma\) as the following quotient stack.
    \[
        \V_\Sigma=\left[U(\Sigma)/G_\Sigma\right].
    \]
\end{definition}
\subsection{Automorphisms of toric stacks}
In the coming work, we will address the following claim, which allows one to calculate automorphism group stacks of general toric stacks.
\begin{claim}\label{claim:exact_toric}
    Let \(\Sigma=(S,K,V,\rho)\) be a complete generalized fan. Then, the automorphism group stack of the associated toric stack \(V_\Sigma\) fits naturally into the following exact sequence.
    \begin{gather*}
        *\to \ul{G_\Sigma}\to \ul{\wt{\Aut}(\Sigma)}\to \ul{\Aut(V_\Sigma)}\to *.
    \end{gather*}
\end{claim}
The proof is based on the results of Lurie \cite{Lurie} and Proposition \ref{prop:Norm_of_G}.
\begin{example}
    Consider the toric variety \(\AA^1\) with the standard torus action. Then, the automorphism group stack of the toric stack \(\left[\AA^1/\CT\right]\) fits into the following exact sequence.
    \[
        *\to \CT\to \CT\to \Aut(\left[\AA^1/\CT\right])\to *.
    \]
\end{example}
\subsection{Automorphism groups of general moment-angle manifolds}
A straightforward application of Claim \ref{claim:exact_toric} is the following description of the automorphism stack of an arbitrary moment-angle manifold.
\begin{claim}
    Let \(\Sigma=(S,K,V,\rho)\) be a complete rational marked fan. Then, the automorphism group of the associated moment-angle manifold \(Z_\Sigma\) fits naturally into the following exact sequence.
    \begin{gather}\label{eq:seq_aut_ma_gen}
        *\to H_\Sigma\to \wt{\Aut}(\Sigma)\to \Aut(Z_\Sigma)\to *.
    \end{gather}
\end{claim}
\begin{proof}
    Consider a moment-angle manifold \(Z_\Sigma\) associated with a complete generalized fan \(\Sigma\). Then by \cite[Lemma 2.1]{ishida_19} any automorphism of a complex moment-angle manifold descends to an automorphism of the toric stack that corresponds to \(\Sigma.\) It remains to apply the argument of Proposition \ref{prop:desc_ker} to see that the kernel consists precisely of the quotient \(F_\Sigma\) and see that the sequence \eqref{eq:seq_aut_ma_gen} is exact.
\end{proof}


\end{document}